\documentclass[11pt]{article}

\usepackage[margin=1in]{geometry}
\usepackage[T1]{fontenc}
\usepackage{amsmath,amsthm,mathtools}
\usepackage{tgtermes}
\usepackage{newtxtext}
\usepackage{newtxmath}
\usepackage{bm}
\usepackage{array}
\usepackage{booktabs}
\usepackage{comment}
\usepackage{float}
\usepackage{algorithm}
\usepackage[noend]{algpseudocode}
\usepackage{subcaption}
\usepackage{xcolor}
\usepackage{microtype}
\usepackage[sort&compress,numbers]{natbib}
\usepackage[colorlinks=true,citecolor=blue,linkcolor=blue,urlcolor=blue,hypertexnames=false]{hyperref}

\allowdisplaybreaks[3]

\bibpunct[, ]{[}{]}{,}{n}{}{,}
\providecommand{\texorpdfstring}[2]{#1}

\providecommand{\argmax}{\operatorname*{arg\,max}}

\theoremstyle{plain}
\newtheorem{theorem}{Theorem}
\newtheorem{lemma}{Lemma}
\newtheorem{proposition}{Proposition}
\newtheorem{corollary}{Corollary}

\theoremstyle{definition}
\newtheorem{assumption}{Assumption}
\newtheorem{definition}{Definition}

\newtheorem{example}{Example}

\theoremstyle{remark}
\newtheorem{remark}{Remark}

\title{\bfseries Wasserstein Distributionally Robust Regret Optimization}

\author{%
Lukas-Benedikt Fiechtner\\
Institute of Computational and Mathematical Engineering\\
Stanford University
\and
Jose Blanchet\\
Institute of Computational and Mathematical Engineering\\
Department of Management Science and Engineering\\
Stanford University
}

\date{}

\begin{document}

\maketitle

\begin{abstract}
Distributionally robust optimization (DRO) is widely used for decision-making under uncertainty, but
its adversarial focus on worst-case loss can lead to overly conservative policies. To mitigate this,
we study ex-ante Distributionally Robust Regret Optimization (DRRO) with Wasserstein ambiguity sets,
designed to balance robustness with upside potential. We develop a theory
of Wasserstein DRRO (WDRRO) paralleling Wasserstein DRO. Under smoothness
and regularity, WDRRO selects among ERM optima by a first-order gradient-discrepancy rule. If the
ERM optimizer is unique, first-order sensitivity vanishes and a second-order expansion governs
deviations. For convex quadratics ERM and DRRO coincide for any radius. We then study regimes where
these assumptions fail: nondifferentiable max-affine losses, discrete references, and
larger radii, where WDRRO can differ from ERM and WDRO. We show that computing WDRRO regret is
NP-hard even without bilinear terms. Nevertheless, we develop exact algorithms, a tractable convex relaxation with
guarantees, and experiments showing tightness and loss-dependent behavior.
\end{abstract}


\section{Introduction}

Decision making under uncertainty often relies on empirical risk minimization (ERM), which
assumes that the empirical distribution adequately represents the deployment environment. ERM
minimizes expected loss and, under classical assumptions such as i.i.d.\ data and smooth losses,
enjoys strong out-of-sample guarantees in a second-order stochastic dominance sense; see
\cite{lam_impossibility_2021}. These assumptions, however, are frequently unverifiable or fail in
practice—for example, the newsvendor loss is merely Lipschitz—so ERM may perform poorly when
training and deployment distributions differ significantly.

Distributionally robust optimization (DRO) addresses such misspecification by allowing an
adversary to select the worst-case distribution from within an ambiguity set
\cite{hansen_robustness_2008}. A widely used choice is a type-$p$ Wasserstein ball of radius
$\delta$ around the empirical measure \cite{mohajerin_esfahani_data-driven_2018}, which permits meaningful perturbations of the
data. Wasserstein DRO (WDRO) has seen broad adoption, ranging from algorithmic fairness
\cite{taskesen_distributionally_2020} to portfolio optimization
\cite{pflug_ambiguity_2007, blanchet_distributionally_2022}, and is known to induce norm
regularization \cite{gao_wasserstein_2024} and to unify many estimators in statistics and machine
learning \cite{blanchet_distributionally_2024}.

Despite its versatility, DRO can be overly conservative because it focuses solely on worst-case
expected loss. This drawback is especially pronounced in settings with asymmetric upside
opportunities (e.g., early-stage investment or low-cost, high-margin inventory), where even small
adversarial perturbations can suppress favorable outcomes. We illustrate this in
Section~\ref{sec:motivation} using the classical newsvendor problem as an example.

To mitigate this over-conservatism we study \emph{Distributionally Robust Regret Optimization (DRRO)}. 
Rather than minimizing the worst-case loss, DRRO evaluates a policy 
\emph{relative to the optimal policy under each candidate distribution}, and minimizes the resulting 
\emph{worst-case regret}. There are two standard forms. 
\emph{Ex-post DRRO} compares a policy’s loss to the loss of the best alternative 
\emph{at each realization of the uncertainty}. 
For every distribution in the ambiguity set, one computes the expected regret of choosing the policy 
versus the pointwise best decision, and then takes the worst case over all distributions. This formulation has been widely 
studied in single-stage~\citep{averbakh_minmax_2004, averbakh_complexity_2005, bitar_distributionally_2024} 
and multistage linear--quadratic settings~\citep{taha_distributionally_2023, kargin_wasserstein_2023, hajar_wasserstein_2023}. 
In contrast, \emph{ex-ante DRRO} compares \emph{expected} losses: for each distribution in the ambiguity set, 
one evaluates the difference between the policy’s expected loss and the expected loss of that distribution’s 
optimal policy, and then takes the worst case over all distributions. Ex-ante DRRO directly captures regret due to 
distributional misspecification. Apart from the early work of \cite{perakis_regret_2008}, this perspective has only 
recently begun to be studied \citep{cho_wasserstein_2024, chen_regret_2021, agarwal_minimax_2022, 
poursoltani_technical_2024}, and it is the focus of this paper. 

Recent work by \cite{kurtz2026globalizedadversarialregretoptimization} introduces
Globalized Adversarial Regret Optimization (GARO), which controls adversarial regret---the
gap between the worst-case cost of a fixed decision and the oracle robust cost---uniformly over
all uncertainty-set sizes. One motivation for GARO is the critique that robust-regret models may allow
policies to become more optimistic as the uncertainty set grows. We take a different view: under asymmetric payoffs, the
direction of the regret-optimal response should depend on the relative upside from favorable outcomes
and downside from unfavorable ones. For example, in low-cost, high-margin newsvendor instances,
larger ambiguity makes under-ordering regret more salient and can justify more aggressive orders;
when the payoff asymmetry reverses, WDRRO becomes conservative instead.

\medskip
Given this background, we develop a theory of Wasserstein DRRO around two main messages.
The first is that Wasserstein DRRO is \emph{not} a blanket replacement for
ERM. In the regular small-radius regime, WDRRO often has no first-order reason to move away from
ERM. When the ERM solution is unique and the loss is smooth enough in the uncertainty, the WDRRO
value has zero first-order sensitivity and near-optimal WDRRO decisions satisfy
$\theta_\delta=\theta_{\mathrm{ERM}}+o(\delta)$; by contrast, Wasserstein DRO typically moves away
from ERM at first order, $\theta_\delta^{\mathrm{DRO}}=\theta_{\mathrm{ERM}}+O(\delta)$. When the
ERM solution is not unique, WDRRO acts instead as a regularizer over the ERM solution set, selecting
those ERM optima with the smallest worst gradient discrepancy. For strictly convex quadratic losses
and Wasserstein orders $p\ge2$, the conclusion is even stronger: the ERM policy is exactly
WDRRO-optimal for every radius.
Thus, in a broad benign regime, regret-based robustification is unnecessary; the paper identifies
when one can safely stay with ERM.
This regular regime is nevertheless structurally informative: when first-order effects vanish, the
second-order expansion identifies the first nonzero regret correction, the stability scale of
near-optimal policies, and the adversarial benchmark directions that attain the regret.

\medskip
The second message is that the interesting regime is the complement of this benign case. When the
loss is nonsmooth or the radius is large enough for higher-order effects to matter, WDRRO and ERM 
can select genuinely different policies. We therefore
study this regime in detail. Unfortunately, it is also where WDRRO becomes computationally fragile:
evaluating WDRRO regret is NP-hard already for the simplest nonlinear max-affine losses, with no
bilinear terms, type-$1$ Wasserstein ambiguity, and an $\ell_1$ ground norm. This is not a generic
feature of Wasserstein robust optimization. The corresponding WDRO, robust optimization, and robust
regret formulations remain tractable for the same max-affine loss class. For this max-affine
model class, including separable bilinear interactions between decisions and uncertainty, we develop
exact finite algorithms for WDRRO minimization and convex relaxations with guarantees. In these
regimes, the WDRRO policy can become more or less aggressive as the radius grows, depending on the
loss profile.

\medskip
\noindent\textbf{Contributions.}
We develop a structural and computational theory of Wasserstein DRRO, paralleling the now classical
WDRO framework, organized around this dichotomy:
ERM is already sufficient in regular regimes, while the regimes where WDRRO is useful require new
algorithms and relaxations.

\begin{enumerate}
    \item\label{item:contr1} \textbf{Sensitivity theory and ERM-set selection.}
    We derive first-order WDRRO expansions for possibly nonunique ERM solution sets and show
    that WDRRO selects ERM optima through a minimax gradient-discrepancy game. Under a unique
    ERM optimizer, we identify the second-order expansion, the $o(\delta)$ optimizer stability,
    and the active opposing directions.

    \item\label{item:contr3} \textbf{Exact ERM optimality for quadratic models.}
    For strictly convex quadratic losses under type-$p$ Wasserstein distances with $p\ge2$, the
    ERM policy is WDRRO-optimal for every radius, not only asymptotically.

    \item\label{item:perakis} \textbf{Tractable special case: the newsvendor problem.}
    For the univariate newsvendor model, WDRRO regret reduces to two one-dimensional concave
    maximizations. The key structural step is a Wasserstein extension of the piecewise concavity
    theorem stated in \cite{perakis_regret_2008}. Combined with specialized evaluators based on
    fractional knapsack for $p=1$ and water-filling for $p>1$, this yields a polynomial-time
    minimization algorithm for every finite Wasserstein order $p\ge1$. The resulting algorithm
    complements the type-$\infty$ Wasserstein newsvendor analysis of \cite{chen_regret_2021} by
    covering the full finite-order range.

    \item\label{item:contr5} \textbf{Hardness frontier.}
    For discrete reference distributions and max-affine losses, we prove that evaluating the original
    WDRRO regret is NP-hard even without bilinear interaction terms, with $p=1$ and an $\ell_1$
    ground norm. The result separates WDRRO from WDRO, robust optimization, and robust regret
    optimization, which are tractable for the same max-affine model class.

    \item\label{item:contr6} \textbf{Exact algorithms and convex relaxations.}
    We derive a finite nonconvex formulation for exact WDRRO regret evaluation and use it to build
    ground-truth algorithms for moderate-sized instances. We then introduce a relaxation for max-affine
    losses without bilinear terms between the decision parameter and the randomness that can be minimized
    as a single convex program. We prove that it is bounded by ex-post regret, exact at radius zero and
    at support-diameter radii, and has an $O(\delta)$ gap. We also give a local-cover extension for
    bilinear interaction terms; if $\rho$ denotes the covering radius of the chosen benchmark centers,
    this extension incurs an additional $O(\rho\delta)$ approximation error. Compared with the
    relaxation in \cite{cho_wasserstein_2024}, which remains NP-hard, our formulation is a convex
    problem rather than an NP-hard one and applies more broadly: to arbitrary finite Wasserstein
    orders, arbitrary ground norms, and unbounded decision sets.

    \item\label{item:contr7} \textbf{Empirical evaluation.}
    Experiments on univariate and two-item newsvendor models and on mean--CVaR portfolio allocation
    show that the relaxation closely tracks exact WDRRO whenever exact computation is possible. In
    the setting where the relaxation of \cite{cho_wasserstein_2024} applies, our relaxation delivers
    comparable regret performance with substantially lower computation time. Separately, the experiments illustrate qualitative behavior that DRO cannot
    capture, including item-specific aggressive/conservative responses and nonmonotone risky-asset
    exposure.
\end{enumerate}

The remainder of the paper is structured as follows:
Section~\ref{sec:motivation} motivates ex-ante regret through the newsvendor example and develops its specialized WDRRO structure.
Section~\ref{sec:sensitivity} establishes the sensitivity, ERM-selection, and quadratic exactness results.
Section~\ref{sec:problems_of_interest} identifies the max-affine computational regime and proves NP-hardness.
Section~\ref{sec:algorithms} develops exact algorithms and convex relaxations.
Section~\ref{sec:experiments} reports numerical results, and Section~\ref{sec:conclusion} concludes.

\vspace{1em}

\textbf{Notation.} We denote by $[N]$ the set $\{1,\ldots,N\}$. We write $a\odot b$ for the 
Hadamard product of two vectors or matrices $a$ and $b$. Further, we let $\vec{1}_K$ denote 
the vector of length $K$ with all entries equal to $1$. For a set $\Xi$ we denote by $\chi_\Xi$ 
the indicator function of the set ($\chi_\Xi(\xi)=0$ for $\xi\in\Xi$ and $\infty$ otherwise) 
and by $\sigma_\Xi$ the support function $\sigma_\Xi(u) = \sup_{\xi\in\Xi}u^\top\xi$. The 
recession cone of a convex set $\Xi$ is defined as 
$\operatorname{rec}(\Xi)=\{ v\in\mathbb{R}^n \mid \xi+\lambda v\in \Xi \quad \forall\, \xi\in\Xi,\ \forall\, \lambda\ge0\}$. 
Further, we denote by $\mathrm{ri}(C)$ and $\mathrm{cl}(C)$ the relative interior and closure 
of a set $C$. The Fenchel conjugate of a function $f(\xi)$ is $f^*(u)=\sup_{\xi\in\mathbb{R}^n}\{u^\top\xi - f(\xi)\}$. 
For a proper, lower semicontinuous, and convex function $f$ we write $\mathrm{dom}(f)=\{\xi\, \vert\, f(\xi) < \infty\}$ 
for the domain of $f$. We define its convex perspective via 
$\underline{f}:\mathbb{R}^n\times\mathbb{R}_+\rightarrow \bar{\mathbb{R}}$ with $\underline{f}(x, t)=tf\left(\frac{x}{t}\right)$ 
for $t>0$ and $\underline{f}(x, 0)=\sigma_{\mathrm{dom}(f^*)}(x)$. This definition ensures that 
the perspective function is also proper, lower semicontinuous, and convex. In the sequel we write 
$tf\left(\frac{x}{t}\right)$ instead of $\underline{f}(x, t)$ even for $t=0$ with the understanding 
that this denotes the extended definition from the perspective function. We also use the convention 
$0\cdot\infty=\infty$.\\ For a norm $\Vert \cdot\Vert$ we write $\Vert \cdot\Vert_*$ for its dual norm 
defined as $\Vert u\Vert_* = \sup_{\Vert \xi\Vert\leq 1}u^\top\xi$.\\ Throughout, we work with losses 
$\ell:\Theta\times \Xi\rightarrow \mathbb{R}$ where $\Theta\subseteq\mathbb{R}^d$ is a closed convex 
parameter set and $\Xi\subseteq\mathbb{R}^n$ is a closed convex set.

\section{Motivation and Problem Statement}\label{sec:motivation}

The regret-based approach underlying DRRO can be well motivated using a standard single-item newsvendor model. 
In this model, a newsvendor can buy newspapers at a price $b$ and sell them at a higher price $s$. The newsvendor 
has to decide on an order quantity $\theta$ before the random demand $X$ is revealed. If the random demand is 
known to follow the distribution $\mathbb{P}_0$, the newsvendor's objective is to solve
\begin{align*}
    \min_{\theta\geq 0}\mathbb{E}_{\mathbb{P}_0}[\ell(\theta, X)],
\end{align*}
where the loss is given by the negative profit, i.e.
\begin{align*}
    \ell(\theta, \xi) = b\theta - s\min(\theta, \xi). 
\end{align*}
It is well known that under this setting the optimal order quantity is given by $F^{-1}(1 - b/s)$ where $F$ is the 
cumulative distribution function (CDF) of $X$ under $\mathbb{P}_0$ and $F^{-1}(\cdot)$ denotes the generalized 
inverse of $F$ \cite[p.260]{nahmias_production_2015}. In circumstances where the demand distribution is unknown, 
but relatively well-modeled by $\mathbb{P}_0$, the DRO approach is to solve the problem
\begin{align}\label{eq:newsvendor_dro}
    \min_{\theta\geq 0}\sup_{\mathbb{P}\in B_\delta^p(\mathbb{P}_0)}\mathbb{E}_{\mathbb{P}}[\ell(\theta, X)]
\end{align}
where $B_\delta^p(\mathbb{P}_0)$ denotes a distributional ambiguity set centered at $\mathbb{P}_0$ with radius $\delta$. 
Such ambiguity sets can be defined in several ways, for instance through moment constraints, divergences, or probability metrics. 
In this paper we consider the Wasserstein case: $B_\delta^p(\mathbb{P}_0)$ is the ball of radius $\delta$ around $\mathbb{P}_0$ 
with respect to the type-$p$ Wasserstein distance $\mathcal{W}_p$. To define this distance, let $\Xi\subseteq\mathbb{R}^n$ 
be the support of the random variable $X$ and $\mathcal{P}(\Xi)$ be the set of all probability distributions supported on $\Xi$. 
For $\mathbb{P}, \mathbb{Q}\in\mathcal{P}(\Xi)$ define $\Pi(\mathbb{P}, \mathbb{Q})$ to be the set of all joint distributions 
with marginals $\mathbb{P}$ and $\mathbb{Q}$. For $p\geq 1$ and a norm $\Vert \cdot\Vert$ the type-$p$ Wasserstein distance is defined as 
\begin{align}
    W_p(\mathbb{P}, \mathbb{Q}) = \left(\inf_{\pi\in\Pi(\mathbb{P}, \mathbb{Q})}\int_{\Xi\times\Xi}\Vert x - y\Vert^p \pi(\mathrm{d}x, \mathrm{d}y)\right)^{\frac{1}{p}}.
\label{eq:wasserstein_distance}
\end{align}

By construction, the DRO order quantity in \eqref{eq:newsvendor_dro} is chosen to maximize the expected profit under the 
worst possible distribution in the Wasserstein ball. Choosing an order quantity in this way ensures that no matter which 
distribution from the Wasserstein ball materializes, the profit of the DRO order quantity will be at least as good as the 
optimal value of the DRO problem, thus providing a worst-case guarantee. However, the adversarial nature of the DRO formulation can also lead to order quantities that might 
be too conservative in practice, focusing too much on the worst-case, and in turn giving up a lot of upside if the correct 
model turns out to be a favorable distribution instead. 

To illustrate this conservativeness, consider an instance of the newsvendor problem where the agent can buy the 
newspapers at the wholesale price $b=0.1$ and sell them at the considerably higher retail price $s=2$. 
Further, assume that the reference distribution $\mathbb{P}_0$ for the demand is Gaussian with mean $100$ 
and standard deviation $10$, i.e. $X\sim N(100, 10^2)$. For this setting we approximate both the ERM order 
quantity and the DRO quantity for the type-2 Wasserstein ball of radius $\delta=10$ using an i.i.d.\ sample of 1000 points. 
The resulting optimal order quantities are $116.60$ for ERM and $111.54$ for DRO\@. 
We provide a heatmap in Figure~\ref{fig:dro_vs_erm_newsvendor} showing the difference between the expected 
profit of the ERM and DRO order quantities for the intersection of the Wasserstein ball with the family of 
normal distributions. The restriction to Gaussian models is only for illustration purposes as it allows for 
a convenient two-dimensional visualization where each point in the ball corresponds to the pair $(\mu, \sigma)$ 
of the mean and standard deviation of a normal distribution $N(\mu, \sigma^2)$. The color indicates the difference 
between the expected profit of the ERM and the DRO stocking strategies when applied to the corresponding Gaussian demands. 
Values larger than zero correspond to Gaussians in the Wasserstein ball under which the ERM strategy outperforms 
the DRO strategy. Likewise, values smaller than zero correspond to Gaussian distributions under which the DRO strategy 
is preferable. It can be seen that for a large portion of the Gaussians in the Wasserstein ball the direct ERM strategy 
outperforms the DRO strategy. Moreover, when the ERM strategy outperforms the DRO strategy, it does so by a significantly 
larger margin than when the DRO strategy proves advantageous.

\begin{figure}[ht]
    \centering
    \includegraphics[width=0.5\textwidth]{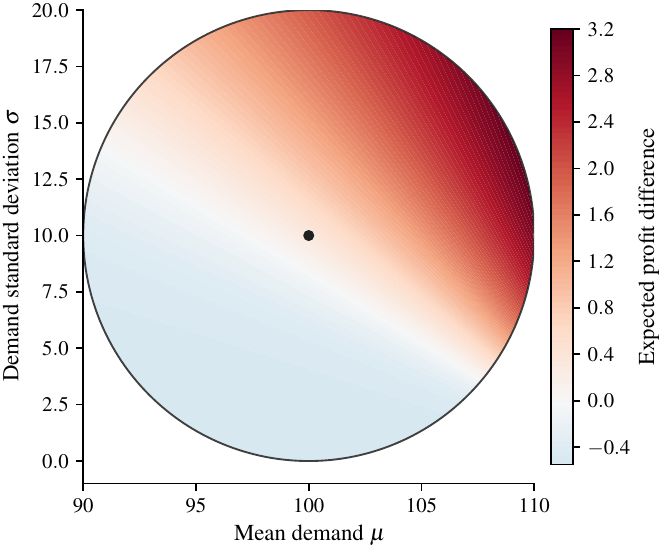}    
    \caption{Difference in expected profit between ERM and DRO policy for all normal distributions in the Wasserstein ball. Positive differences (red) indicate the gap by which the ERM strategy outperforms DRO in an out-of-sample environment corresponding to the pair $(\mu, \sigma)$. Negative differences represent the corresponding gap (blue) when the DRO strategy outperforms the ERM strategy.}
    \label{fig:dro_vs_erm_newsvendor}
\end{figure}

The key takeaway from Figure~\ref{fig:dro_vs_erm_newsvendor} is that for most Gaussians within the Wasserstein ball, 
one would regret choosing the DRO strategy over the ERM strategy in the sense that if any of these Gaussians were 
the true distribution, the expected profit of the ERM policy would exceed that of the DRO strategy. While there are 
also Gaussians for which one would regret picking the ERM over the DRO strategy, the regret incurred in these cases 
is much smaller than in the cases where one regrets opting for the DRO strategy.

The primary motivation for using DRO is that the true distribution is unknown, although it is assumed to 
lie within a specified distributional uncertainty set. When a policy $\theta$ is selected, and the true 
distribution $\mathbb{P}$ is revealed retrospectively, we can compare the expected loss associated with 
policy $\theta$ under $\mathbb{P}$ against the loss of the optimal policy for $\mathbb{P}$. 
The difference between these expected losses represents the regret from having chosen $\theta$ given 
a distribution $\mathbb{P}$. Given that all distributions within the uncertainty set are considered 
plausible candidates for the true distribution, it is rational to minimize this type of regret uniformly 
over the uncertainty set. Thus, the objective is to minimize the following notion of regret over $\theta$
\begin{align}\label{def:regret}
    R(\theta) = \sup_{\mathbb{P}\in B^p_\delta(\mathbb{P}_0)}\left\{\mathbb{E}_\mathbb{P}[\ell(\theta, X)] - \inf_{\beta\in\Theta}\mathbb{E}_\mathbb{P}[\ell(\beta, X)]\right\}.
\end{align}
where $\Theta$ is the parameter set for $\theta$. This is known as \textit{ex-ante regret} in the literature as it measures the 
regret from not knowing the true distribution. When $\delta=0$ it recovers the ERM solution. In contrast, \textit{ex-post regret} is defined as
\begin{align}\label{eq:ex_post_regret}
    R^{\mathrm{post}}(\theta) = \sup_{\mathbb{P}\in B^p_\delta(\mathbb{P}_0)}\mathbb{E}_\mathbb{P}\left[\ell(\theta, X) - \inf_{\beta\in\Theta}\ell(\beta, X)\right]
\end{align}
and measures the worst-case average sample-by-sample regret and is, at least formally, a DRO problem with a modified loss. 
When $\delta=0$, minimizing the ex-post regret still recovers an ERM optimizer, since the hindsight term is independent of $\theta$. 
However, its optimal value need not be zero, because the benchmark is allowed to choose the best decision separately for each realization of $X$. 
Thus ex-post regret measures performance relative to a different, perfect-information benchmark. 
In this paper we focus on ex-ante regret, as it directly addresses the regret associated with uncertainty about the actual distribution.  
When speaking of regret, we henceforth refer to ex-ante regret.

Note that if the loss function $\ell(\theta, \xi)$ is convex in $\theta$, then $R(\theta)$ is also convex, as convexity is preserved under expectations and suprema. However, writing the regret as $R(\theta) = \sup_{\beta \in \Theta} h(\beta; \theta)$ with
\begin{align}\label{eq:g_def}
h(\beta; \theta) := \sup_{\mathbb{P} \in B^p_\delta(\mathbb{P}_0)} \mathbb{E}_\mathbb{P}[\ell(\theta, X) - \ell(\beta, X)],
\end{align}
it becomes evident that the computation of the regret can be challenging. This difficulty arises because $h(\beta; \theta)$ 
is generally not concave in $\beta$, as concavity is not preserved under the supremum operation. However, for the univariate 
newsvendor model we establish the following Wasserstein analogue of a piecewise-concavity assertion appearing in 
\cite{perakis_regret_2008}. We give a self-contained proof in Appendix~\ref{app:newsvendor}, because the finite-order 
Wasserstein formulation requires optimizer structure that is not provided by the earlier argument.

\begin{theorem}\label{thm:piecewise_concavity}
Fix $\theta \ge 0$, a Wasserstein order $p \in [1,\infty)$, and an empirical
reference distribution
$\mathbb{P}_0=N^{-1}\sum_{i=1}^N\delta_{\hat\xi_i}$ on $\mathbb{R}_+$.
For the newsvendor loss $\ell(\theta,\xi)=b\theta-s\min(\theta,\xi)$,
the map $\beta \mapsto h(\beta;\theta)$ defined in \eqref{eq:g_def} is concave on
$[0,\theta]$ and on $[\theta,\infty)$.
\end{theorem}
The restriction to empirical reference distributions is not essential for the structure of the result:
the same argument extends directly to finitely supported reference
distributions with non-uniform weights, and the corresponding statement for
general reference distributions can be obtained by approximation.

The theorem implies that the regret can be written as the maximum of two concave maximization problems:
\begin{align}\label{eq:newsvendor_regret_split}
    R(\theta) = \max\left(\sup_{\beta\in[0, \theta]}h(\beta;\theta), \sup_{\beta\in[\theta, \infty)}h(\beta;\theta)\right).
\end{align}
This representation in turn yields an efficient procedure for minimizing
$R(\cdot)$. In Appendix~\ref{app:newsvendor_wc_algorithms}, we show that each
evaluation of $h(\beta;\theta)$ can be computed by a fractional-knapsack
algorithm when $p=1$ and by a water-filling algorithm when $p>1$. These
evaluations are then used inside the one-dimensional minimization over the DRRO
policy described in Appendix~\ref{app:newsvendor_policy_algorithms}. This
complements and extends the minimax-regret algorithm in \cite{chen_regret_2021}
for a newsvendor model under type-$\infty$ Wasserstein ambiguity. In
particular, we give a direct proof of the piecewise concavity used above and
cover finite Wasserstein orders $p$.

Table~\ref{tab:erm_dro_comparison} shows the worst- and best-case performance (both measured in profits rather than losses) 
as well as the regret for the ERM, DRO, and DRRO strategies in the simple newsvendor example. 
Here worst-case performance is the smallest expected profit over the Wasserstein ball, while best-case performance is the largest expected profit over the same ball.
The best performing policy for each criterion among these three policies is emphasized in boldface. 
The value in parentheses in each cell shows the optimality gap between the policy in the corresponding 
row and the best-performing policy among these three under the criterion indicated by the column. By construction, the DRO strategy is 
optimal for worst-case performance and therefore outperforms both ERM and DRRO on that criterion. 
Although none of the three strategies is designed to optimize this best-case criterion, ERM and DRRO 
have smaller optimality gaps with respect to best-case performance, with DRRO achieving the smallest 
gap among the three. Finally, DRRO is designed to be optimal with respect to regret. 
The amount by which DRRO outperforms DRO in terms of regret is substantial and much larger than the 
gap between the two strategies in worst-case performance, which is the quantity optimized by DRO.

\begin{table}[ht]
    \centering
    \begin{tabular}{lcccc}
        \toprule
        & \textbf{Order Quantity} & \textbf{Worst-Case Performance} & \textbf{Best-Case Performance} & \textbf{Regret}\\
        \midrule
        \textbf{ERM}   & 116.60 & 167.36 (-0.21)    & 206.07 (-0.32)    & 6.04 (+2.97)\\
        \textbf{DRO}   & 111.54 & \textbf{167.57}            & 204.53 (-1.86)    & 8.79 (+5.72)\\
        \textbf{DRRO}  & 125.48 & 166.55 (-1.02)    & \textbf{206.39}            & \textbf{3.07}\\
        \bottomrule
    \end{tabular}
    \caption{Performance comparison of ERM, DRO, and DRRO policies for $\delta=10$. The worst-case and best-case 
    performance refer to profits (instead of losses).}
    \label{tab:erm_dro_comparison}
\end{table}

Another insight from Table~\ref{tab:erm_dro_comparison} concerns the optimal order quantities $\theta$. 
The DRO strategy always becomes more conservative as the uncertainty radius 
$\delta$ increases. The behavior of the DRRO strategy is more nuanced and depends on the relationship 
between the selling price $s$ and the buying price $b$. In the present example, the newsvendor buys 
cheaply at $b=0.1$ and sells at a high price of $s=2$. Ordering more than the ERM policy is therefore 
natural: the upside is large when demand is high, while the downside from over-ordering is limited because 
the newspapers are cheap. Consistent with the order quantities in Table~\ref{tab:erm_dro_comparison}, 
Figure~\ref{fig:comparison_newsvendor} shows that the DRRO policy becomes more aggressive as the radius 
increases, whereas the DRO policy becomes more conservative. In a regime with a larger buying price, 
however, DRRO can also become more conservative as the uncertainty radius grows, as we show in the 
numerical experiments in Section~\ref{sec:experiments}. Thus, while DRO becomes more conservative as 
$\delta$ increases independently of the relation between $s$ and $b$, the direction of the DRRO response 
depends on this price regime. In terms of regret, the benefit of the DRRO strategy becomes clear in 
Figure~\ref{fig:comparison_newsvendor}, which plots the regret of the various strategies against the uncertainty radius:
in this example, even the ERM policy significantly outperforms the DRO strategy. 

\begin{figure}[ht]
    \centering
    \includegraphics[width=0.8\textwidth]{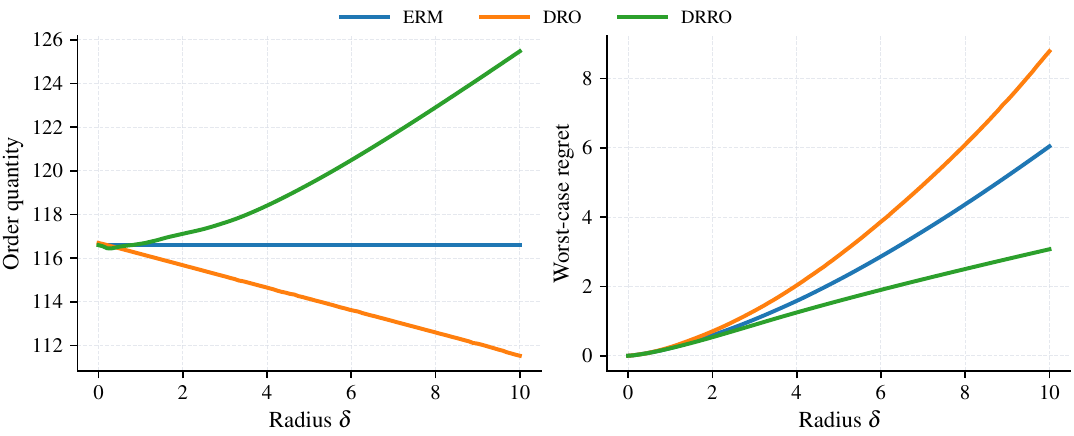}
    \caption{Comparison of order quantities and regret for ERM, DRO, and DRRO with respect to $\delta$.}
    \label{fig:comparison_newsvendor}
\end{figure}

Finally, we show in Figure~\ref{fig:worst_case-dists_drro} the empirical CDF of the original data distribution
together with the CDFs of the two distributions that maximize the regret for the optimal DRRO policy. 
The first distribution shifts mass to the left and corresponds to the regret from ordering too much. 
The other distribution shifts mass to the right and hence corresponds to the regret of ordering too little. 
These two distributions correspond to the two terms inside the maximum in \eqref{eq:newsvendor_regret_split}. 
For the DRRO optimal policy the two expressions in the maximum will be equal and the corresponding worst-case 
distributions from \eqref{eq:g_def} yield two distributions that both attain the regret.
The fact that there are always at least two worst-case distributions in DRRO 
should be contrasted with the standard DRO formulation where under suitable assumptions there is only one worst-case distribution. 
In summary, this example illustrates that a regret-focused policy can maintain robustness while capturing more upside than 
a worst-case-focused policy.

\begin{figure}[ht]
    \centering
    \includegraphics[width=0.5\textwidth]{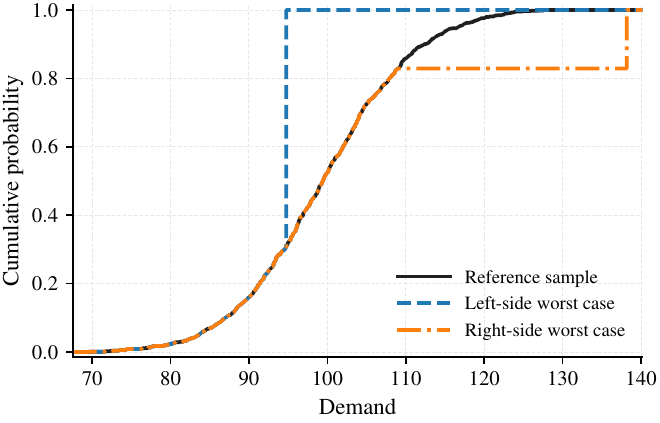}    
    \caption{Empirical CDFs of the reference sample and the two distributions attaining the worst-case regret for the DRRO order quantity.}
    \label{fig:worst_case-dists_drro}
\end{figure}

\section{Sensitivity and Optimizer Asymptotics}
\label{sec:sensitivity}

This section analyzes the small-radius behavior of WDRRO, extending the WDRO
sensitivity theory of \cite{bartl_sensitivity_2021} from worst-case risk to
worst-case regret. We write $L(\theta)=\mathbb E_{\mathbb P_0}[\ell(\theta,X)]$ 
for the expected loss under the reference measure $\mathbb{P}_0$.
The standing support and moment
assumptions are that $\Xi$ is closed and convex, $p\in(1,\infty)$ with
conjugate exponent $q=p/(p-1)$, $\mathbb P_0(\partial\Xi)=0$, and
$\mathbb E_{\mathbb P_0}[\|X\|^p]<\infty$.
The main message is that, in the regular small-radius regime and especially when the ERM optimizer is unique, WDRRO does not require 
a first-order departure from ERM\@. Robust-regret effects either select among multiple ERM solutions or appear at higher order.

To make the dependence on the Wasserstein radius explicit, for $\delta\ge0$ set
\[
F_\delta(\theta,\beta)=\sup_{\mathbb P\in B_\delta^p(\mathbb P_0)}\mathbb E_{\mathbb P}[\ell(\theta,X)-\ell(\beta,X)],
\quad R_\delta(\theta)=\sup_{\beta\in\Theta}F_\delta(\theta,\beta),\quad V_\delta=\inf_{\theta\in\Theta}R_\delta(\theta).
\]
Thus $R_\delta$ is the
WDRRO regret from \eqref{def:regret} at radius $\delta$, and
$R_0(\theta)=L(\theta)-L^*$, where $L^*=\inf_{\theta\in\Theta}L(\theta)$ and we also define
$\Theta^*=\operatorname*{argmin}_{\theta\in\Theta}L(\theta)$.

In our subsequent arguments we use near-optimizers rather than assuming attainment. A sequence
$\theta_\delta\in\Theta$ is $o(\delta^r)$-optimal if
$R_\delta(\theta_\delta)\le V_\delta+o(\delta^r)$. Given $\theta_\delta$,
a sequence $\beta_\delta\in\Theta$ is $o(\delta^r)$-active if
$F_\delta(\theta_\delta,\beta_\delta)\ge
R_\delta(\theta_\delta)-o(\delta^r)$. Exact optimizers and active $\beta$'s
correspond to zero error.
All proofs for this section are given in Appendix~\ref{app:sensitivity}.

\subsection{First-order approximation and nonunique ERM}
\label{subsec:first_order_wdrro}

The first result describes the generic case in which the ERM set may be nonunique. The leading-order term is 
governed by the worst first-order gradient discrepancy between ERM solutions. 
This gives a minimax selection rule over the ERM set.

\begin{assumption}[First-order regularity on bounded parameter sets]
\label{ass:fo_regularity}
For every bounded set $K\subseteq\Theta$, the map $(\theta,x)\mapsto \ell(\theta,x)$ is continuous on $K\times\Xi$, 
the map $x\mapsto \ell(\theta,x)$ is differentiable on $\Xi^\circ$, and $(\theta,x)\mapsto\nabla_x\ell(\theta,x)$ 
admits a continuous extension to $K\times\Xi$. Moreover, for every bounded $K\subseteq\Theta$, there is $C_K<\infty$ such that
\[
    |\ell(\theta,x)|\le C_K(1+\|x\|^p),\qquad \|\nabla_x\ell(\theta,x)\|_*\le C_K(1+\|x\|^{p-1})
\]
for all $\theta\in K$ and $x\in\Xi$.
\end{assumption}

For $\theta,\beta\in\Theta$, define
\[
    \Gamma(\theta,\beta)=\left(\mathbb E_{\mathbb P_0}\left[\|\nabla_x\ell(\theta,X)-\nabla_x\ell(\beta,X)\|_*^q\right]\right)^{1/q}.
\]
For $\theta^*\in\Theta^*$, set $A(\theta^*)=\sup_{\beta^*\in\Theta^*}\Gamma(\theta^*,\beta^*)$ and 
$\alpha=\inf_{\theta^*\in\Theta^*}A(\theta^*)$. Define
$\Theta_1^*=\{\theta^*\in\Theta^*:A(\theta^*)=\alpha\}$ and, for $\bar\theta\in\Theta_1^*$,
$\mathcal B_1(\bar\theta)=\operatorname*{argmax}_{\beta^*\in\Theta^*}\Gamma(\bar\theta,\beta^*)$.

\begin{assumption}[First-order bounded near-optimizer condition]
\label{ass:fo_bounded_theta}
Every $o(\delta)$-optimal sequence $\theta_\delta$ is bounded.
\end{assumption}

\begin{assumption}[First-order bounded active-$\beta$ condition]
\label{ass:fo_bounded_beta}
Whenever $\theta_\delta$ is bounded and satisfies $L(\theta_\delta)\to L^*$, every $o(\delta)$-active sequence 
$\beta_\delta$ for $\theta_\delta$ is bounded.
\end{assumption}

\begin{theorem}[First-order WDRRO expansion and ERM-set selection]
\label{thm:fo_wdrro_selection}
Assume the standing support and moment assumptions, Assumptions~\ref{ass:fo_regularity}--\ref{ass:fo_bounded_beta}, 
$\Theta^*\neq\emptyset$, and $\alpha<\infty$. Then $V_\delta=\alpha\delta+o(\delta)$. Every $o(\delta)$-optimal 
sequence satisfies $\operatorname{dist}(\theta_\delta,\Theta^*)\to0$, and all its cluster points lie in $\Theta_1^*$. 
If $\beta_\delta$ is $o(\delta)$-active for $\theta_\delta$, then $\operatorname{dist}(\beta_\delta,\Theta^*)\to0$. 
Moreover, along any subsequence with $\theta_{\delta_n}\to\bar\theta$ and $\beta_{\delta_n}\to\bar\beta$, 
one has $\bar\theta\in\Theta_1^*$, $\bar\beta\in\mathcal B_1(\bar\theta)$, equivalently 
$\Gamma(\bar\theta,\bar\beta)=A(\bar\theta)=\alpha$.

If, in addition, $L$ has quadratic growth around $\Theta^*$, i.e. $L(\theta)-L^*\ge c\,\operatorname{dist}(\theta,\Theta^*)^2$ 
whenever $\operatorname{dist}(\theta,\Theta^*)\le r$ for some $c,r>0$, then $\operatorname{dist}(\theta_\delta,\Theta^*)=o(\sqrt{\delta})$ 
for every $o(\delta)$-optimal sequence and $\operatorname{dist}(\beta_\delta,\Theta^*)=o(\sqrt{\delta})$ for every 
$o(\delta)$-active sequence.
\end{theorem}

The coefficient $\alpha$ is the first-order value of the game played inside the ERM set. If $\Theta^*$ is nonunique, 
the robust-regret optimizer is not merely an arbitrary ERM solution: it selects those ERM solutions whose worst 
gradient discrepancy against other ERM solutions is smallest. If $\Theta^*$ is a singleton, then $\alpha=0$, 
and Theorem~\ref{thm:fo_wdrro_selection} only says that $V_\delta=o(\delta)$. The next subsection identifies 
the leading term in that case. However, before turning to the singleton case, we record sufficient conditions for the boundedness assumptions.

\begin{lemma}[Sufficient conditions for bounded near-optimizers]
\label{lemma:bounded_near_optimizers}
Assume the standing support and moment assumptions and Assumption~\ref{ass:fo_regularity}. Assumptions~\ref{ass:fo_bounded_theta} 
and~\ref{ass:fo_bounded_beta} hold if either of the following conditions holds.

\begin{enumerate}
    \item $\Theta$ is bounded and thus compact.

    \item $L$ is coercive, i.e. $D(\theta):=L(\theta)-L^*\to\infty$ as $\|\theta\|\to\infty$, and there exist constants $C<\infty$ 
    and $\eta\in[0,1]$ such that $\|\nabla_x\ell(\theta,x)\|_* \le C\bigl(1+D(\theta)^\eta\bigr)(1+\|x\|^{p-1})$ for 
    all $\theta\in\Theta$ and $x\in\Xi$.
\end{enumerate}
\end{lemma}

\begin{lemma}[Compact sufficient conditions for attainment]
\label{lemma:attainment}
Assume that $\Theta$ and $\Xi$ are compact and that $\ell$ is continuous on $\Theta\times\Xi$. Then, for every $\delta\ge0$:
\begin{enumerate}
    \item for fixed $(\theta,\beta)$, the supremum in $F_\delta(\theta,\beta)$ is attained.
    \item for fixed $\theta$, the supremum over $\beta$ in $R_\delta(\theta)$ is attained.
    \item the infimum over $\theta$ in $V_\delta$ is attained.
\end{enumerate}
\end{lemma}

\begin{example}[Rank-deficient empirical least squares]
\label{ex:rank_deficient_regression}
Let $\mathbb P_0=N^{-1}\sum_{i=1}^N\delta_{(x_i,y_i)}$ on $\mathbb R^m\times\mathbb R$,
with squared loss $\ell(\theta,(x,y))=(y-x^\top\theta)^2$, and take $p\ge2$ with the $\ell_2$ ground norm for covariate transport. We use the standard DRO perturbation model for regression in which only the features are transported, as in, for example, \cite{blanchet_robust_2019,kuhn_wasserstein_2024}. Thus labels are fixed
and only covariates are transported. Technically, this is an OT ambiguity set generated by the extended cost
$c((x,y),(\tilde x,\tilde y))=\|x-\tilde x\|_2^p$ when $y=\tilde y$ and $+\infty$ otherwise, not a Wasserstein ball induced by a
finite norm on the joint feature-label space. However, the same calculations apply with this transport cost in place of a norm-based cost.
Let $X\in\mathbb R^{N\times m}$ be the design matrix with rows $x_i^\top$, assume it is rank deficient, and let
$\hat\theta^\dagger=(X^\top X)^\dagger X^\top y$,
and set $r_i=y_i-x_i^\top\hat\theta^\dagger$ and $\hat\sigma^2=N^{-1}\sum_i r_i^2$.

If $\Theta=\mathbb R^m$ and $\hat\sigma^2>0$, then $V_0=0$ but $\lim_{\delta\downarrow0}V_\delta=\hat\sigma^2$.
Thus $\delta\mapsto V_\delta$ is discontinuous at zero; see Appendix~\ref{app:sensitivity} for a proof.

Now let $\Theta=B_R(0):=\{\theta:\|\theta\|_2\le R\}$ be the closed Euclidean ball. If $\|\hat\theta^\dagger\|_2\le R$,
then $\Theta^*=(\hat\theta^\dagger+\ker X)\cap B_R(0)$. All ERM solutions share the residual vector $r$, and for
$\theta,\beta\in\Theta^*$,
\[
    \Gamma(\theta,\beta)=2\left(\frac1N\sum_{i=1}^N |r_i|^q\right)^{1/q}\|\theta-\beta\|_2.
\]
Hence the selected ERM policies are the Chebyshev centers of $\Theta^*$ with
$\alpha=2(N^{-1}\sum_i |r_i|^q)^{1/q}\operatorname{rad}(\Theta^*)$, where
$\operatorname{rad}(\Theta^*)=\inf_{\theta\in\Theta^*}\sup_{\beta\in\Theta^*}\|\theta-\beta\|_2$.
Since $\hat\theta^\dagger\perp\ker X$, the Euclidean-ball Chebyshev center is $\hat\theta^\dagger$ and
$\operatorname{rad}(\Theta^*)=\sqrt{R^2-\|\hat\theta^\dagger\|_2^2}$. This also shows $\alpha\to\infty$ as
$R\to\infty$ which is consistent with the discontinuity in the unbounded case.

When $\|\hat\theta^\dagger\|_2>R$, the constrained least-squares solution is the unique trust-region solution
$\hat\theta_R=(X^\top X+\lambda I)^{-1}X^\top y$, where $\lambda>0$ is the unique multiplier chosen so that
$\|\hat\theta_R\|_2=R$. This is not generally the Euclidean projection of $\hat\theta^\dagger$.
The first-order coefficient from Theorem~\ref{thm:fo_wdrro_selection} then vanishes because the
constrained ERM set is a singleton.
\end{example}

\subsection{Second-order expansion under a unique ERM optimizer}
\label{subsec:unique_second_order}

When the ERM optimizer is unique, the first-order coefficient in Theorem~\ref{thm:fo_wdrro_selection} is zero. The next result identifies the first nonzero term,
 the sharper optimizer rate, and the active opposing directions.

\begin{assumption}[Unique ERM and positive definite second-order expansion]
\label{ass:unique_second_order}
The ERM set is $\Theta^*=\{\theta_0\}$ with $\theta_0\in\Theta^\circ$, and $L(\theta_0+h)=L(\theta_0)+\frac12 h^\top Hh+o(\|h\|^2)$ for some symmetric $H\succ0$.
\end{assumption}

\begin{assumption}[Local mixed smoothness near $\theta_0$]
\label{ass:mixed_smoothness}
There is a neighborhood $U\subseteq\Theta^\circ$ of $\theta_0$ such that $\theta\mapsto\nabla_x\ell(\theta,x)$ is differentiable on $U$ for every $x\in\Xi$,
the mixed derivative $\nabla_x\nabla_\theta\ell(\theta,x)\in\mathbb R^{n\times d}$ is continuous on $U\times\Xi$,
and for every compact $K\subset U$ there is $C_K<\infty$ such that
$\|\nabla_x\nabla_\theta\ell(\theta,x)\|\le C_K(1+\|x\|^{p-1})$, where the matrix norm is
$\|M\|:=\sup_{\|u\|\le1}\|Mu\|_*$.
\end{assumption}

\begin{assumption}[Second-order bounded near-optimizer condition]
\label{ass:so_bounded_theta}
Every $o(\delta^2)$-optimal sequence $\theta_\delta$ is bounded.
\end{assumption}

\begin{assumption}[Second-order bounded active-$\beta$ condition]
\label{ass:so_bounded_beta}
For every bounded sequence $a_\delta\in\mathbb R^d$, every $o(\delta^2)$-active sequence $\beta_\delta$ for $\theta_0+\delta a_\delta$ is bounded.
\end{assumption}

These boundedness assumptions are the second-order analogues of Assumptions~\ref{ass:fo_bounded_theta}--\ref{ass:fo_bounded_beta}, 
but they only need to control sequences on the local scale relevant when the ERM solution is unique. The near-optimality 
requirement is sharper, $o(\delta^2)$ instead of $o(\delta)$, and the active-$\beta$ condition only needs to control 
opponents of decisions of the form $\theta_0+\delta a_\delta$ with bounded $a_\delta$, since the theorem proves that 
second-order optimal decisions satisfy $\theta_\delta=\theta_0+o(\delta)$.

Set $B(x)=\nabla_x\nabla_\theta\ell(\theta_0,x)\in\mathbb R^{n\times d}$, 
so that $B(x)u=\nabla_x(\nabla_\theta\ell(\theta_0,x)^\top u)$. Define
\[
    \gamma(u)=\left(\mathbb E_{\mathbb P_0}[\|B(X)u\|_*^q]\right)^{1/q},\quad
    \rho(a)=\sup_w\{w^\top Ha-\tfrac12 w^\top Hw+\gamma(w)\},
    \quad
    \kappa=\rho(0),
\]
and $\mathcal W^*=\operatorname*{argmax}_w\{\gamma(w)-\tfrac12w^\top Hw\}$.

\begin{theorem}[Second-order WDRRO expansion under a unique ERM]
\label{thm:unique_second_order}
Assume the standing support and moment assumptions, Assumption~\ref{ass:fo_regularity}, and Assumptions~\ref{ass:unique_second_order}--\ref{ass:so_bounded_beta}.
Then $R_\delta(\theta_0+\delta a)=\delta^2\rho(a)+o(\delta^2)$ locally uniformly in $a$, and $V_\delta=\kappa\delta^2+o(\delta^2)$.
Every $o(\delta^2)$-optimal sequence satisfies $\theta_\delta=\theta_0+o(\delta)$. If $\beta_\delta$ is $o(\delta^2)$-active and
$w_\delta=(\theta_\delta-\beta_\delta)/\delta$, then $w_\delta$ is bounded and every cluster point lies in $\mathcal W^*$.
In other words, along convergent subsequences, $\beta_\delta=\theta_0-\delta w+o(\delta)$ with $w\in\mathcal W^*$. Conversely, every $w\in\mathcal W^*$
is realized by the active branch $\beta_\delta^w=\theta_0-\delta w$ for the decision $\theta_0$.

Since $\gamma$ is even, $\mathcal W^*$ is symmetric. Thus, if $\kappa>0$, there are at least two asymptotic active-$\beta$ directions,
and in general these directions are exactly the maximizers in $\mathcal W^*$. Finally, the value can also be written in the
scale-invariant form $\kappa=\frac12\sup_{u\ne0}\gamma(u)^2/(u^\top Hu)$.
\end{theorem}

In this unique-ERM regime, the WDRRO value sensitivity at $\delta=0$ is zero. This is sharper than the analogous WDRO behavior: under comparable smoothness assumptions,
the analysis in \cite{bartl_sensitivity_2021} shows that the WDRO optimizer typically satisfies $\theta_\delta^{\rm DRO}=\theta_0 + O(\delta)$, and the derivative of the WDRO value at $\delta=0$ is
generally nonzero. In WDRRO, by contrast, near-optimal decisions satisfy
$\theta_\delta=\theta_0+o(\delta)$. Thus, under the regularity conditions above, ERM is already first-order optimal for small-radius regret
objectives. The robust-regret correction only appears at second order.

\begin{remark}[Computation of $\kappa$]
\label{rem:kappa_computation}
The Rayleigh quotient for $\kappa$ is one half of a squared operator norm, with the domain measured in the curvature norm
$(u^\top Hu)^{1/2}$ and the codomain measured by the $L^q(\mathbb P_0)$ norm induced by the dual ground norm. In general
this operator norm is hard to compute because it maximizes a convex, possibly nonsmooth seminorm over an ellipsoid. When
$p=2$ and the transport norm is Euclidean, however, $\gamma(u)^2=u^\top Cu$ with
$C=\mathbb E_{\mathbb P_0}[B(X)^\top B(X)]$. The quotient then reduces to finding the largest generalized eigenvalue in
$Cu=\lambda Hu$, so $\kappa=\frac12\lambda_{\max}(H^{-1/2}CH^{-1/2})$. The quotient variable $u$ only identifies a
direction because the quotient is scale invariant. The active directions in Theorem~\ref{thm:unique_second_order} are
the top generalized eigenvectors with the scale that maximizes $\gamma(w)-\frac12w^\top Hw$: if $v$ satisfies
$Cv=\lambda_{\max}Hv$, then
\[
    w=\sqrt{\frac{\lambda_{\max}}{v^\top Hv}}\,v
\]
belongs to $\mathcal W^*$, and all nonzero top generalized eigenvectors give elements of $\mathcal W^*$ in this way.
\end{remark}

\begin{example}[The quadratic loss]
\label{ex:quadratic_second_order}
Consider
\[
    \ell(\theta,\xi)
    =
    \begin{bmatrix}\theta\\\xi\end{bmatrix}^\top
    \begin{bmatrix}Q&S^\top\\ S&R\end{bmatrix}
    \begin{bmatrix}\theta\\\xi\end{bmatrix}
    +
    2
    \begin{bmatrix}\theta\\\xi\end{bmatrix}^\top
    \begin{bmatrix}u\\r\end{bmatrix},
    \qquad
    Q\succ0.
\]
Let $\hat\mu=\mathbb E_{\mathbb P_0}[X]$. Then the unique ERM optimizer is $\theta_0=-Q^{-1}(S^\top\hat\mu+u)$, $H=2Q$, and $B(x)w=2Sw$.
Thus, in the Euclidean $p=2$ case of Remark~\ref{rem:kappa_computation}, $C=4S^\top S$ and the generalized eigenvalue formula gives
$\kappa=\lambda_{\max}(SQ^{-1}S^\top)$. For a general ground norm, the same Rayleigh quotient reads
$\kappa=\sup_{u\ne0}\|Su\|_*^2/(u^\top Qu)$. The next theorem shows that this local expansion is globally exact for any ground norm
and Wasserstein order $p\ge2$.
\end{example}

\begin{theorem}[Exact WDRRO optimality for quadratic losses]
\label{thm:exact_quadratic}
In the setting of Example~\ref{ex:quadratic_second_order}, assume $\Theta=\mathbb R^d$, $\Xi=\mathbb R^n$, and that the Wasserstein ball uses
any norm $\|\cdot\|$ on $\mathbb R^n$ and any order $p\ge2$. Then the ERM policy $\theta_0=-Q^{-1}(S^\top\hat\mu+u)$ is WDRRO-optimal for every
$\delta\ge0$, and $R_\delta(\theta_0)=V_\delta=\delta^2\sup_{\|z\|\le1}z^\top SQ^{-1}S^\top z=\delta^2\kappa$. For the Euclidean ground norm,
this equals $\delta^2\lambda_{\max}(SQ^{-1}S^\top)$.
\end{theorem}

The quadratic loss in Example~\ref{ex:quadratic_second_order} is the loss of a single-stage linear quadratic regulator (LQR).
Theorem~\ref{thm:exact_quadratic} shows that ERM remains globally WDRRO-optimal in this one-stage setting. A subsequent paper 
\cite{fiechtner2026distributionallyrobustregretoptimal} shows that this global ERM optimality is special to the single-stage case: 
for multistage LQR with independent disturbances sharing the same unknown stage law, repeated interaction lets the controller learn about 
that fixed distribution, and ERM need no longer be globally regret-optimal once the horizon exceeds one.

\section{Computational Regimes for WDRRO}\label{sec:problems_of_interest}

The sensitivity analysis in Section~\ref{sec:sensitivity} shows that, in the regular small-radius regime, WDRRO stays close to ERM\@. 
In particular, when the ERM optimizer is unique, Theorem~\ref{thm:unique_second_order} gives zero value sensitivity and 
$o(\delta^2)$-optimal decisions satisfy $\theta_\delta=\theta_0+o(\delta)$. Thus, in regimes where this theorem applies, ERM is already 
first-order optimal for regret. The rest of the paper therefore studies settings where one might actually use WDRRO: primarily 
non-infinitesimal radii, together with nonsmooth losses, and $p=1$ cases that are not settled by the smooth small-radius theory. 
Additionally, we will focus on the data-driven regime where $\mathbb{P}_0$ is an empirical distribution, although all results
in the following also apply to any discrete distribution. The modeling assumptions should at least make the inner
Wasserstein risk tractable in these decision-relevant regimes, so that the remaining difficulty can be isolated as the
outer regret maximization.

For fixed $(\theta,\beta)$, WDRRO evaluates the worst-case risk of 
$\tilde\ell_{\theta,\beta}(\xi)=\ell(\theta,\xi)-\ell(\beta,\xi)$. Large-radius tractability therefore requires finite 
convex reformulations of this inner risk problem that are valid for all radii $\delta$. Such global finite programs are available for 
empirical reference distributions and losses that are pointwise maxima of concave functions 
\citep{mohajerin_esfahani_data-driven_2018, kuhn_wasserstein_2024, zhen_unified_2023}. By contrast, other algorithmic approaches 
are designed for small radii and are typically stochastic-gradient in nature, so they do not provide finite all-radius convex 
programs that can be embedded in the outer minimax over $\theta$ and $\beta$ 
\citep{blanchet_optimal_2021, sinha_certifying_2020}.

To use these finite convex reformulations inside WDRRO, we need the difference $\ell(\theta,\xi)-\ell(\beta,\xi)$ to be representable as a pointwise 
maximum of concave functions in $\xi$. A simple way to ensure this is to make $\ell(\theta,\xi)$ max-affine in $\xi$, so that 
subtracting $\ell(\beta,\xi)$ leaves a maximum of concave piecewise-linear functions. We additionally impose structure in $\theta$ 
and allow a separable bilinear interaction term. The resulting modeling assumptions are as follows.

\begin{assumption}\label{ass:model_choice}\textbf{}
    \begin{itemize}
\item The reference distribution is discrete. Throughout the paper we focus 
on the empirical case $\mathbb{P}_0 = \hat{\mathbb{P}}_N = \frac{1}{N}\sum_{i=1}^N \delta_{\hat\xi_i}$.
All results extend straightforwardly to general discrete reference measures.
        \item The loss function is of the form $\ell(\theta, \xi) = \max_{k\in[K]} \bigl(a_k^\top \xi + b_k^\top \theta + c_k + (d_k \odot \theta)^\top \xi\bigr)$, i.e., a maximum of affine terms with an additional separable bilinear interaction.
    We use $A \in \mathbb{R}^{K\times n}$ and $B \in \mathbb{R}^{K\times d}$ to denote the matrices with rows $a_k^\top$ and $b_k^\top$, respectively, and write $c \in \mathbb{R}^K$ and $D \in \mathbb{R}^{K\times n}$ for the vector and matrix collecting $c_k$ and $d_k$. We assume that $d=n$ whenever $d_k\neq0$.
    \end{itemize}
\end{assumption}

\begin{remark}\label{rem:approx_u}
The above restriction is less limiting than it may appear. On compact domains, continuous jointly convex losses can be 
approximated arbitrarily well by max-affine functions. The form also covers approximations of losses 
$\ell(\theta,\xi)=u(\xi^\top\theta)$ with scalar convex $u$, since $u$ can be approximated on the compact range of 
$\xi^\top\theta$ by a maximum of affine functions, yielding terms of the form $c_k+d_k\xi^\top\theta$.
\end{remark}

Assumption~\ref{ass:model_choice} ensures that the inner worst-case risk $\sup_{\mathbb{P}\in B_\delta^p(\mathbb{P}_0)}\mathbb{E}_{\mathbb{P}}\big[\ell(\theta, X) - \ell(\beta, X)\big]$ is tractable for arbitrary radii $\delta$, via the convex reformulations in \cite{kuhn_wasserstein_2024}. 
However, computing the regret $R(\theta)$ also requires taking the supremum over $\beta$, namely
$R(\theta) = \sup_{\beta\in \Theta}\sup_{\mathbb{P}\in B_\delta^p(\mathbb{P}_0)}\mathbb{E}_{\mathbb{P}}[\ell(\theta, X) - \ell(\beta, X)]$.
As we show next, this additional maximization makes regret evaluation NP-hard even in this tractable inner-risk regime. 
Section~\ref{sec:algorithms} then develops exact and approximate methods that perform well in practice.

\subsection{NP-Hardness}\label{subsec:nphard}

We now show that, even under Assumption~\ref{ass:model_choice}, computing the
regret $R(\theta)$ is NP-hard in general. This is not obvious a priori: for
instance, the classical newsvendor problem is tractable, as discussed in
Section~\ref{sec:motivation}. One might suspect that the hardness is caused by
nonlinear modeling ingredients, such as the bilinear terms
$(d_k\odot\theta)^\top\xi$, a nonlinear transportation penalty with $p>1$, or a
nonlinear ground norm. Our result rules this out. NP-hardness persists even when
$d_k\equiv0$, $p=1$, and the ground norm is $\ell_1$, so all primitive modeling
ingredients are linear or max-affine.

Our proof relies on a dual representation of the worst-case Wasserstein risk.
Applying Theorem~1 of \cite{blanchet_quantifying_2019} to
$\tilde\ell(\xi)=\ell(\theta,\xi)-\ell(\beta,\xi)$ yields

\begin{small}
\begin{align}
    R(\theta)
    &=
    \sup_{\beta\in\Theta}
    \inf_{\lambda \geq 0}
    \left\{
        \lambda\delta^p
        +
        \frac{1}{N}\sum_{i=1}^N
        \sup_{z_i\in \Xi}
        \left[
            \max_{k\in[K]}\bigl(a_k^\top z_i + b_k^\top\theta + c_k\bigr)
            -
            \max_{m\in[K]}\bigl(a_m^\top z_i + b_m^\top\beta + c_m\bigr)
            -
            \lambda \Vert z_i - \hat \xi_i\Vert^p
        \right]
    \right\}.
    \label{eq:regret_data_driven}
\end{align}
\end{small}

The hardness comes from the interaction between the adversarial samples $z_i$ and the benchmark decision $\beta$. For fixed 
$z_i$, the bracketed term is concave in $\beta$, but the supremum over $z_i$ does not preserve concavity because 
$\max_{k\in[K]}(a_k^\top z_i+b_k^\top\theta+c_k)$ is convex in $z_i$. Thus the remaining maximization over $\beta$ can be 
computationally hard.

\begin{theorem}\label{thm:np_hardness_regret}
Consider the class of WDRRO instances in \eqref{eq:regret_data_driven} with
$p=1$, $\ell_1$ ground norm, and max-affine losses with $d_k\equiv0$. The
decision problem of determining whether $R(\theta)\ge \tau$ for a given
instance and a given $\theta\in\Theta$ is NP-hard for the fixed threshold
$\tau=2/3$. Consequently, exact evaluation of $R(\theta)$ is NP-hard in
general.
\end{theorem}

The proof reduces the NP-complete 1-in-3 SAT+ problem to this WDRRO decision problem and is
provided in Appendix~\ref{app2}.

\begin{remark}[Relation to Cho and Yang's minimax-swapped surrogate]
The authors of \cite{cho_wasserstein_2024} propose a relaxation of the WDRRO objective by
swapping the infimum over the Wasserstein dual variable and the supremum over
the benchmark decision. In our notation, their surrogate optimizes
$\inf_{\theta\in\Theta,\lambda\ge0}R^\lambda(\theta)$, where, for fixed
$\lambda$,

\begin{small}
\begin{align}
 R^\lambda(\theta)
 =
 \sup_{\beta\in\Theta}
 \left\{
 \lambda\delta^p
 +
 \frac1N\sum_{i=1}^N
 \sup_{z_i\in\Xi}
 \left[
 \max_{k\in[K]}(a_k^\top z_i+b_k^\top\theta+c_k)
 -
 \max_{m\in[K]}(a_m^\top z_i+b_m^\top\beta+c_m)
 -
 \lambda\|z_i-\hat\xi_i\|^p
 \right]
 \right\}.
 \label{eq:soft_regret_rewritten}
\end{align}
\end{small}

The quantity $R^\lambda(\theta)$ acts as a soft regret.
Using the same construction as in the proof of
Theorem~\ref{thm:np_hardness_regret}, with $\lambda$ fixed to a suitable
rational value, it can be seen that evaluating this soft regret is still
NP-hard, even for $p=1$, the $\ell_1$ ground norm, and max-affine losses with
$d_k\equiv0$. Thus the minimax swap does not by itself remove the evaluation
hardness.

The NP-hardness argument in \cite{cho_wasserstein_2024} is complementary to
ours. Their Section~IV shows intractability of the surrogate minimization
problem $\inf_{\theta,\lambda}R^\lambda(\theta)$ and uses losses of the form
$\ell(\theta,\xi)=|\theta-\xi^\top M\xi|=\max\{\theta-\xi^\top M\xi,-\theta+\xi^\top M\xi\}$, 
i.e., maxima of quadratic functions of the uncertainty. In contrast,
Theorem~\ref{thm:np_hardness_regret} concerns pointwise evaluation of the
\emph{original} WDRRO regret $R(\theta)$, before the minimax swap, and the hardness
already holds for the considerably simpler class of max-affine losses without
bilinear terms.
\end{remark}

\begin{remark}[Comparison with WDRO, RO, and RRO]
It is also instructive to notice that many standard robust formulations admit
tractable reformulations under max-affine losses. For standard WDRO,
Theorem~8 of \cite{kuhn_wasserstein_2024} gives exact conic reformulations for
max-affine losses. When the resulting support-function computations are
tractable, these reformulations are of polynomial size. Robust optimization
(RO) corresponds to the limiting WDRO case with $p=\infty$, $N=1$, and a
sufficiently large radius such that the uncertainty set $\Xi$ is contained in
$B_\delta(\hat\xi_1)$. Under the same representation assumptions on $\Xi$, RO
is therefore also polynomial-time solvable in this setting.

A similar conclusion holds for robust regret optimization (RRO),
\begin{align}
\inf_{\theta\in\Theta}\sup_{\xi\in\Xi}
\left\{
\ell(\theta,\xi)-\inf_{\beta\in\Theta}\ell(\beta,\xi)
\right\}
=
\inf_{\theta\in\Theta}\sup_{\xi\in\Xi}\sup_{\beta\in\Theta}
\{\ell(\theta,\xi)-\ell(\beta,\xi)\}.
\label{eq:rro_rewritten}
\end{align}
For max-affine losses with $d_k\equiv0$, standard duality results under the applicable Slater conditions give the convex reformulation
\[
\begin{aligned}
\inf_{\theta\in\Theta,\ s,\ \mu^1,\ldots,\mu^K}\quad & s\\
\mathrm{s.t.}\quad
& b_k^\top\theta+c_k
+\sigma_{\Xi}(a_k-A^\top\mu^k)
+\sigma_{\Theta}(-B^\top\mu^k)
-c^\top\mu^k
\le s,\qquad k\in[K],\\
& \mu^k\in\mathbb R_+^K,\quad \vec 1_K^\top\mu^k=1,\qquad k\in[K].
\end{aligned}
\]
If $\Theta$ and $\Xi$ admit tractable conic representations, for example if
they are polyhedral, then the support functions above do as well, and the RRO
reformulation is a polynomial-sized conic program. Thus, although worst-case
regret optimization can be NP-hard for more general loss classes
\citep{Poursoltani2019AdjustableRO}, max-affine losses with $d_k\equiv0$
constitute a regime where WDRO, RO, and RRO remain tractable while WDRRO regret
evaluation is already NP-hard.
\end{remark}

\section{Algorithmic Solutions for the Max-Affine Loss Model Class}\label{sec:algorithms}

The previous section showed that regret evaluation is NP-hard even for
$\ell(\theta,\xi)=\max_{k\in[K]}a_k^\top\xi+b_k^\top\theta+c_k$. Thus WDRRO cannot be solved in polynomial time 
over the full model class unless $\mathrm{P}=\mathrm{NP}$. This section therefore has two goals: first, to give an exact 
finite formulation that can be used as a benchmark on moderate-sized instances, and second, to develop convex relaxations 
that are the main computational tool for larger problems. We first derive a finite-dimensional nonconvex formulation of 
the regret using Wasserstein duality
\citep{kuhn_wasserstein_2024, zhen_unified_2023, mohajerin_esfahani_data-driven_2018}. We then develop a convex relaxation 
for the case $d_k=0$ with an $O(\delta)$ gap, followed by a local-cover extension for $d_k\neq0$ with an additional 
$O(\rho\delta)$ approximation error, where $\rho$ is the covering radius of the chosen benchmark centers.

\subsection{Finite Reformulation of the Regret}

The first step is to remove the infinite-dimensional optimization over $\mathbb P$. The standard Wasserstein duality in \eqref{eq:regret_data_driven}
does this for fixed $(\theta,\beta)$, but it produces a finite minimization problem, so the outer supremum over
$\beta$ cannot simply be appended. We therefore use the bidual reformulation in \cite[Theorem~9]{kuhn_wasserstein_2024},
which represents the worst-case risk as a finite maximization problem and allows $\beta$ to be included as a decision
variable. The theorem applies when $\ell(\theta,\xi)-\ell(\beta,\xi)$ is a maximum of concave functions of $\xi$.
This is the reason for the model choice in Section~\ref{sec:problems_of_interest}: max-affineness of $\ell$ in $\xi$
ensures that the regret loss $\tilde\ell(\xi)=\ell(\theta,\xi)-\ell(\beta,\xi)$ remains a maximum of concave piecewise-linear functions.
Indeed, $\tilde\ell$ can be written as $\tilde\ell(\xi)=\max_{k\in[K]} \tilde{\ell}_k(\xi)$ with
\[
    \tilde{\ell}_k(\xi) = (a_k + d_k\odot \theta)^\top \xi + b_k^\top \theta + c_k - \max_{m\in[K]}\left\{(a_m + d_m\odot\beta)^\top \xi + b_m^\top \beta + c_m\right\}.
\]
Using the bidual reformulation in \cite[Theorem~9]{kuhn_wasserstein_2024}, introducing epigraphical variables, and
including $\beta$ in the maximization gives the following formulation.

\begin{corollary}\label{cor:finite_regret_formulation}
For an empirical reference distribution $\hat{\mathbb{P}}_N$, the regret $R(\theta)$ can be computed as the optimal value of
{\small
\begin{equation}\label{eq:kuhn_dual_applied}
\begin{alignedat}{2}
\sup\quad
& \frac{1}{N}\sum_{i=1}^N \sum_{k=1}^K
    \Big[
        \gamma_{ik}
        \Big(
            (a_k + d_k \odot \theta)^\top \hat{\xi}_i
            + b_k^\top \theta
            + c_k
        \Big)
        + (a_k + d_k \odot \theta)^\top q_{ik}
        - t_{ik}
    \Big]
\\
\mathrm{s.t.}\quad
& \beta \in \Theta,
&&
\\
& \gamma_{ik} \in \mathbb{R}_+,\quad
  q_{ik} \in \mathbb{R}^n,\quad
  t_{ik} \in \mathbb{R},
&& \forall i \in [N],\; k \in [K],
\\
& \hat{\xi}_i + \frac{q_{ik}}{\gamma_{ik}} \in \Xi,
&& \forall i \in [N],\; k \in [K],
\\
& \sum_{k=1}^K \gamma_{ik} = 1,
&& \forall i \in [N],
\\
& \frac{1}{N}\sum_{i=1}^N \sum_{k=1}^K
    \gamma_{ik}
    \left\|
        \frac{q_{ik}}{\gamma_{ik}}
    \right\|^p
    \leq \delta^p,
&&
\\
& t_{ik}
    \geq
    \gamma_{ik}
    \Big(
        (a_m + d_m \odot \beta)^\top \hat{\xi}_i
        + b_m^\top \beta
        + c_m
    \Big)
    + (a_m + d_m \odot \beta)^\top q_{ik},
&& \forall i \in [N],\; k,m \in [K].
\end{alignedat}
\end{equation}
}
\end{corollary}

The above program is convex in all variables except for the bilinear terms $\gamma_{ik}\beta$ and $\beta^\top q_{ik}$
in the last constraint. These bilinear couplings are precisely what make the problem NP-hard in general. Moreover,
since bilinearity appears in the constraints rather than the objective, the feasible set is not a polytope even in
the case $p=1$ and $\Vert \cdot\Vert\in\{\Vert\cdot\Vert_1, \Vert\cdot\Vert_\infty\}$, which prevents the direct use
of Benders-type cutting-plane methods of the kind employed in \cite{cho_wasserstein_2024}.

\subsection{Exact Algorithms}\label{sec:exact_algs}

Since $R(\theta)$ is convex in $\theta$, one can minimize it by subgradient or cutting-plane methods. Both require solving
\eqref{eq:kuhn_dual_applied} at the current $\theta$. If $(\beta^*(\theta),\gamma^*(\theta),q^*(\theta))$ solves
\eqref{eq:kuhn_dual_applied}, Danskin's theorem \cite[Theorem~9.26]{shapiro_lectures_2009} gives the subgradient
\begin{align}\label{eq:subgradient}
    g(\theta) = \frac{1}{N}\sum_{i=1}^N\sum_{k=1}^K\gamma_{ik}^*(\theta)\bigl(b_k + d_k \odot \hat\xi_i\bigr) + q_{ik}^*(\theta)\odot d_k.
\end{align}
In the experiments we use cutting planes, with interval bisection in one dimension and Chebyshev centers
in low-dimensional polyhedral sets. The main bottleneck is that each oracle call requires a global
solve of the NP-hard problem \eqref{eq:kuhn_dual_applied}.

To the best of our knowledge, the formulation in \eqref{eq:kuhn_dual_applied},
together with the exact solution procedures described in this subsection, constitutes the first exact algorithm
for computing WDRRO in the literature. Previous work---most notably \cite{cho_wasserstein_2024}---
focuses on algorithms for relaxations of the regret objective and therefore does not provide a method for evaluating the
true regret $R(\theta)$ or for identifying an exact regret-optimal policy. As a result, their empirical results cannot
rely on a ground-truth WDRRO benchmark. By contrast, our exact formulation enables us to compute such a benchmark and
thereby rigorously assess the quality of the convex relaxation developed next. This benchmarking capability is the primary
motivation for including the exact method in our framework.

\subsection{Convex Relaxations}\label{sec:convex_relaxations}

We now turn to convex relaxations of the regret, which sacrifice exactness but enable efficient computation and direct joint
optimization over $\theta$ and auxiliary variables in a single convex program.

\subsubsection{\texorpdfstring{Relaxation for the case $d_k=0$}{Relaxation for the case dk = 0}}\label{subsec:dk_zero}

We first consider the case where there are no bilinear terms in the loss, i.e., $d_k=0$ for all $k$.
In this setting, the only bilinear terms in \eqref{eq:kuhn_dual_applied} are the products $\gamma_{ik}\beta$ in the constraints,
while terms of the form $\beta^\top q_{ik}$ are absent. This structure admits a tractable convex relaxation.

For each product $\gamma_{ik}\beta$, introduce a new variable $z_{ik}$ and impose the constraints
\begin{align*}
    \sum_{k=1}^K z_{ik} = \beta \quad \forall i\in[N],\quad \frac{z_{ik}}{\gamma_{ik}} \in \Theta \quad \forall i\in[N],\; k\in[K].
\end{align*}
For $\gamma_{ik}=0$, the second constraint is interpreted as requiring $z_{ik}$ to be a recession direction of $\Theta$.
They hold automatically when $z_{ik}=\gamma_{ik}\beta$, so keeping only these constraints relaxes the bilinear equality.
Substituting $z_{ik}$ into \eqref{eq:kuhn_dual_applied} yields the relaxed maximization problem
\begin{align}\label{eq:regret_opt_problem}
\begin{alignedat}{3}
    \sup & \quad \frac{1}{N} \sum_{i=1}^N \sum_{k=1}^K
    \left[
        \gamma_{ik}\left( a_k^\top \hat{\xi}_i + b_k^\top\theta + c_k\right)
        + a_k^\top q_{ik} - t_{ik}
    \right] &\\
    \mathrm{s.t.} &\quad \beta\in\mathbb{R}^d,\; \gamma_{ik} \in \mathbb{R}_+, \; q_{ik} \in \mathbb{R}^n, \; t_{ik}\in\mathbb{R}, \; z_{ik}\in\mathbb{R}^d & \forall i \in [N], \; \forall k\in [K],\\
    &\quad \beta\in\Theta, &\\
    &\quad \hat{\xi}_i + \frac{q_{ik}}{\gamma_{ik}} \in \Xi & \forall i \in [N], \; \forall k \in [K], \\
    &\quad \frac{z_{ik}}{\gamma_{ik}} \in \Theta & \forall i \in [N], \; \forall k \in [K], \\
    &\quad \sum_{k=1}^K \gamma_{ik} = 1 & \forall i \in [N], \\
    &\quad \sum_{k=1}^K z_{ik} = \beta & \forall i \in [N], \\
    &\quad \frac{1}{N} \sum_{i=1}^N \sum_{k=1}^K \gamma_{ik} \left\| \frac{q_{ik}}{\gamma_{ik}} \right\|^p \leq \delta^p, & \\
    &\quad t_{ik}\geq \gamma_{ik} \left( a_m^\top \hat{\xi}_i + c_m\right) + b_m^\top z_{ik} + a_m^\top q_{ik} & \forall i \in [N], \; \forall k, m \in [K].
\end{alignedat}
\end{align}

Taking the dual of \eqref{eq:regret_opt_problem} yields the convex minimization problem stated in the following lemma,
whose proof is given in Appendix~\ref{app:algs_proofs}.

\begin{lemma}\label{lemma:relaxation_full}
Assume that $\Xi$ and $\Theta$ are non-empty closed convex sets. For $p>1$, the dual of \eqref{eq:regret_opt_problem} is
given by the convex program
\begin{align}\label{eq:full_relaxation}
\begin{alignedat}{2}
\inf\quad & \lambda\delta^p + \frac{1}{N} \sum_{i=1}^N s_i + \sigma_\Theta\left( \frac{1}{N} \sum_{i=1}^N \tau_i \right)\\[1ex]
\mathrm{s.t.}\quad
  & \lambda \in \mathbb{R}_+, \quad s_i \in \mathbb{R}, \quad \mu_{ik} \in \mathbb{R}^K_+, \quad \tau_i \in \mathbb{R}^d, \quad u_{ik} \in \mathbb{R}^n, && \forall\, i \in [N],\, k \in [K],\\[1ex]
  & \mu_{ik}^\top \vec{1}_K = 1, && \forall\, i \in [N],\, k \in [K],\\[1ex]
  & a_k^\top \hat{\xi}_i + b_k^\top \theta + c_k - \mu_{ik}^\top\bigl( A\, \hat{\xi}_i + c \bigr)
  - u_{ik}^\top \hat{\xi}_i + \phi(q)\lambda \left\Vert \frac{a_k - A^\top \mu_{ik} - u_{ik}}{\lambda} \right\Vert_*^q\\[1ex]
  & \quad {}+ \sigma_{\Xi}\bigl(u_{ik}\bigr) + \sigma_\Theta\!\Bigl( -\bigl(B^\top \mu_{ik} + \tau_i\bigr) \Bigr)
  \leq s_i, && \forall\, i \in [N],\, k \in [K],
\end{alignedat}
\end{align}
where $q=\frac{p}{p-1}$ and $\phi(q) = \frac{(q-1)^{q-1}}{q^q}$. For $p=1$, the last constraint is replaced, for all
$i\in[N]$ and $k\in[K]$, by the two constraints
\begin{align*}
    s_i&\geq a_k^\top \hat{\xi}_i + b_k^\top \theta + c_k - \mu_{ik}^\top\bigl( A\, \hat{\xi}_i + c \bigr)
  - u_{ik}^\top \hat{\xi}_i + \sigma_{\Xi}\bigl(u_{ik}\bigr) + \sigma_\Theta\!\Bigl( -\bigl(B^\top \mu_{ik} + \tau_i\bigr) \Bigr),\\
  \lambda&\geq \Vert a_k - A^\top \mu_{ik} - u_{ik} \Vert_*.
\end{align*}
\end{lemma}

The following corollary specializes this relaxation to the case where $\Xi$ and $\Theta$ are polyhedral.

\begin{corollary}\label{corr:relaxation_full_polyhedral}
    Let $\Xi = \{x\, \vert \, Px \leq r\}$ and $\Theta = \{\theta\, \vert \, M\theta \leq w\}$ with $P\in\mathbb{R}^{m_1\times n}$, $r\in\mathbb{R}^{m_1}$, $M\in\mathbb{R}^{m_2\times d}$, $w\in\mathbb{R}^{m_2}$. For $p>1$, the program \eqref{eq:full_relaxation} becomes
\begin{align}
\begin{alignedat}{2}\label{eq:full_relaxation_polyhedral}
\inf\quad & \lambda\delta^p + \frac{1}{N} \sum_{i=1}^N s_i + w^\top\eta\\[1ex]
\mathrm{s.t.}\quad
  & \lambda \in \mathbb{R}_+,\quad s_i \in \mathbb{R},\quad \mu_{ik} \in \mathbb{R}^K_+,\quad \tau_i \in \mathbb{R}^d, && \forall\, i\in[N],\, k\in[K],\\
  & \zeta_{ik}\in\mathbb{R}^{m_1}_+,\quad \eta\in\mathbb{R}^{m_2}_+,\quad \nu_{ik} \in \mathbb{R}^{m_2}_+,
  && \forall\, i\in[N],\, k\in[K],\\[1ex]
  & \mu_{ik}^\top \vec{1}_K = 1, && \forall\, i\in[N],\, k\in[K],\\[1ex]
  & M^\top\nu_{ik}=-(B^\top\mu_{ik}+\tau_i), && \forall\, i\in[N],\, k\in[K],\\[1ex]
& M^\top \eta = \frac{1}{N}\sum_{i=1}^N\tau_i,\\[1ex]
  & s_i\geq a_k^\top \hat{\xi}_i + b_k^\top \theta + c_k - \mu_{ik}^\top\bigl( A\, \hat{\xi}_i + c \bigr)
  + \zeta_{ik}^\top(r - P\hat\xi_i)\\
  &\quad\quad  + w^\top \nu_{ik} + \phi(q)\lambda \left\Vert \frac{a_k - A^\top \mu_{ik} - P^\top \zeta_{ik}}{\lambda} \right\Vert_*^q, && \forall\, i\in[N],\, k\in[K].
\end{alignedat}
\end{align}
For $p=1$, the last constraint is replaced, for all $i\in[N]$ and $k\in[K]$, by
\begin{align*}
    s_i&\geq a_k^\top \hat{\xi}_i + b_k^\top \theta + c_k - \mu_{ik}^\top\bigl( A\, \hat{\xi}_i + c \bigr)
  + \zeta_{ik}^\top(r - P\hat\xi_i) + w^\top \nu_{ik},\\
  \lambda&\geq\left\Vert a_k - A^\top \mu_{ik} - P^\top \zeta_{ik}\right\Vert_*.
\end{align*}
\end{corollary}

Since the program \eqref{eq:full_relaxation} and its specialization in \eqref{eq:full_relaxation_polyhedral} are jointly convex
in all decision variables, one can therefore treat $\theta$ as a decision variable and add
the constraint $\theta\in\Theta$, which yields a single convex program whose optimal solution provides an approximation of
the regret-optimal policy.

The relaxation in \cite{cho_wasserstein_2024} is sandwiched between the ex-ante and ex-post
regrets and is exact whenever the radius is $0$ or larger than the diameter of the distribution support.
The next results establish analogous properties for our relaxation. Thus, our relaxation and that of \cite{cho_wasserstein_2024}
share similar bounding and exactness guarantees, but our formulation leads to a convex program that can be solved directly,
whereas their relaxation remains NP-hard. Moreover, their tractable algorithms based on Benders-type decompositions require
$p=1$, $\Vert\cdot\Vert$ equal to the $\ell_1$- or $\ell_\infty$-norm, and a compact polyhedral $\Theta$.
Our relaxation applies to any $p$, any norm defining the Wasserstein distance, and unbounded $\Theta$.
The proofs of the following results are given in Appendix~\ref{app:algs_proofs}.

\begin{theorem}\label{thm:relax_bound}
    The relaxation in \eqref{eq:full_relaxation} is upper bounded by the ex-post regret defined in \eqref{eq:ex_post_regret}.
\end{theorem}

\begin{theorem}\label{thm:exact_relaxation}
    For $\delta=0$ and for $\delta\geq D(\Xi) := \sup_{\xi\neq\xi'}\Vert \xi - \xi'\Vert$, the relaxation
    in \eqref{eq:full_relaxation} is exact. In the latter case, ex-ante and ex-post regret coincide.
\end{theorem}

We additionally provide an explicit bound on the relaxation gap in terms of the Wasserstein radius $\delta$.

\begin{theorem}\label{thm:relax_gap_delta}
    Assume $p\geq 1$ and that $\Xi$ and $\Theta$ are compact. Let $\widehat R(\theta)$ denote the
    optimal value of the relaxation \eqref{eq:regret_opt_problem}, equivalently of its dual \eqref{eq:full_relaxation}, and let
    $R(\theta)$ denote the true ex-ante regret. If $L_\xi = \max_{k\in[K]}\Vert a_k\Vert_*$, then for every $\theta\in\Theta$ and every $\delta\geq0$ the bounds
    $0\leq \widehat R(\theta)-R(\theta)\leq 2L_\xi\delta$ hold.
\end{theorem}

\subsubsection{\texorpdfstring{Relaxation for the case $d_k\neq0$}{Relaxation for the case dk not equal 0}}\label{subsec:dk_not_zero}

We now extend the relaxation to the general case $d_k \neq 0$, which covers approximations of losses such as
$\ell(\theta,\xi)=u(\theta^\top\xi)$ as discussed in Remark~\ref{rem:approx_u}. The obstacle is the additional product $(d_m\odot\beta)^\top q_{ik}$, for which
the $z_{ik}=\gamma_{ik}\beta$ relaxation used above has no useful analogue. We therefore localize the adversarial decision:
choose centers $\beta_1,\ldots,\beta_L\in\Theta$ with covering radius
$\rho=\mathrm{rad}(\{\beta_1,\ldots,\beta_L\},\Theta,\Vert\cdot\Vert_*)$. By definition of $\rho$, this gives the exact
cover-based reformulation of the regret minimization problem
\begin{align}
    \inf_{\theta, t} \, t\quad \mathrm{s.t.}\,
    \sup_{\substack{\mathbb{P}\in B_\delta^p(\mathbb{P}_0), \; \Delta\beta_l: \\
    \beta_l + \Delta\beta_l\in \Theta,\; \Vert\Delta\beta_l\Vert_*\leq \rho}}
    \mathbb{E}_\mathbb{P}[\ell(\theta, X) - \ell(\beta_l + \Delta\beta_l, X)]
    \leq t\quad \forall l\in[L].
    \label{eq:approx_cover_local}
\end{align}
Each supremum in the constraints of \eqref{eq:approx_cover_local} can be written with the same bidual formulation as
\eqref{eq:kuhn_dual_applied}, with $\beta$ replaced by $\beta_l+\Delta\beta_l$ and
$\Delta\beta_l$ constrained to $\mathcal C_l=(\Theta-\beta_l)\cap\{\Delta:\|\Delta\|_*\le\rho\}$.
After this substitution, the new difficulty is again the cross term $(d_m\odot\Delta\beta_l)^\top q_{ik}$. Our approximation first
makes use of the fact that $\Vert\Delta\beta_l\Vert_*$ is small and drops this cross term.
Once it is removed, the local problem has the same structure as the $d_k=0$ case, with
$\Delta\beta_l$ as the benchmark variable over $\mathcal C_l$. We can therefore introduce the relaxation
$z_{ikl}=\gamma_{ik}\Delta\beta_l$ over $\mathcal C_l$ and dualize as above. 

To state the error resulting from dropping the cross term, define the operator norm
\[
    L_D=\max_{m\in[K]}\sup_{\Delta\neq0}\frac{\|d_m\odot\Delta\|_*}{\|\Delta\|_*}.
\]
For $\ell_r$ ground norms this reduces to $\max_m\|d_m\|_\infty$. Since $\|\Delta\beta_l\|_*\le\rho$, each dropped cross
term satisfies, with the perspective convention for zero-weight terms,
\[
    |(d_m\odot\Delta\beta_l)^\top q_{ik}|
    \le
    L_D\rho\,\gamma_{ik}\left\|\frac{q_{ik}}{\gamma_{ik}}\right\|.
\]
Averaging over $i,k$, Jensen's inequality using $N^{-1}\sum_{i,k}\gamma_{ik}=1$, and the transport constraint gives
\[
    \frac{L_D\rho}{N}\sum_{i=1}^N\sum_{k=1}^K \gamma_{ik}
    \left\|\frac{q_{ik}}{\gamma_{ik}}\right\|
    \le
    L_D\rho
    \left(
        \frac1N\sum_{i=1}^N\sum_{k=1}^K \gamma_{ik}
        \left\|\frac{q_{ik}}{\gamma_{ik}}\right\|^p
    \right)^{1/p}
    \le L_D\rho\delta.
\]
Thus this first approximation changes each $l$-indexed worst-case regret term in
\eqref{eq:approx_cover_local} by at most $L_D\rho\delta$. The following lemma gives the resulting polyhedral
formulation and combines this error with the relaxation gap from the $d_k=0$ step. We omit the proof
of the formulation because it is analogous to the duality derivations above.

\begin{lemma}\label{lem:local_relaxation_polyhedral}
    Let $\Xi=\{\xi\in\mathbb{R}^n\mid P\xi\leq r\}$ and
    $\Theta=\{\theta\in\mathbb{R}^d\mid M\theta\leq w\}$, where
    $P\in\mathbb{R}^{m_1\times n}$, $r\in\mathbb{R}^{m_1}$, $M\in\mathbb{R}^{m_2\times d}$, and
    $w\in\mathbb{R}^{m_2}$.
    Let $\beta_1,\ldots,\beta_L\in\Theta$ be chosen such that
    $\rho=\mathrm{rad}(\{\beta_1,\ldots,\beta_L\},\Theta,\Vert\cdot\Vert_*)$. For $p>1$, the local convex relaxation is

\begin{small}
\[
\begin{alignedat}{2}
\inf\quad & t\\[1ex]
\mathrm{s.t.}\quad
  & \theta\in\Theta,\quad t\in\mathbb{R},\quad \lambda_l\in\mathbb{R}_+,\quad s_{il}\in\mathbb{R},\quad \mu_{ikl}\in\mathbb{R}^K_+,\quad \tau_{il}\in\mathbb{R}^d, && \forall\, i\in[N],\, k\in[K],\, l\in[L],\\
  & \zeta_{ikl}\in\mathbb{R}^{m_1}_+,\quad \eta_l\in\mathbb{R}^{m_2}_+,\quad \eta_{ikl}\in\mathbb{R}^{m_2}_+,
  && \forall\, i\in[N],\, k\in[K],\, l\in[L],\\[1ex]
  & t \geq \lambda_l\delta^p + \frac{1}{N}\sum_{i=1}^N s_{il}
  + \eta_l^\top(w-M\beta_l)
  + \rho\left\Vert\frac{1}{N}\sum_{i=1}^N\tau_{il}-M^\top\eta_l\right\Vert,
  && \forall\, l\in[L],\\[1ex]
  & \mu_{ikl}^\top \vec{1}_K = 1,
  && \forall\, i\in[N],\, k\in[K],\, l\in[L],\\[1ex]
  & s_{il} \geq (a_k+d_k\odot\theta)^\top\hat\xi_i+b_k^\top\theta+c_k
  -\mu_{ikl}^\top\Bigl((A+D\mathrm{diag}(\beta_l))\hat\xi_i+B\beta_l+c\Bigr)\\
  &\quad\quad +\zeta_{ikl}^\top(r-P\hat\xi_i)
  + \eta_{ikl}^\top(w-M\beta_l)\\
  &\quad\quad +\rho\left\Vert(B+D\mathrm{diag}(\hat\xi_i))^\top\mu_{ikl}+\tau_{il}+M^\top\eta_{ikl}\right\Vert\\
  &\quad\quad +\phi(q)\lambda_l
  \left\Vert
      \frac{a_k+d_k\odot\theta-(A+D\mathrm{diag}(\beta_l))^\top\mu_{ikl}-P^\top\zeta_{ikl}}{\lambda_l}
  \right\Vert_*^q,
  && \forall\, i\in[N],\, k\in[K],\, l\in[L].
\end{alignedat}
\]
\end{small}

where $q=p/(p-1)$ and $\phi(q)=(q-1)^{q-1}/q^q$.
For $p=1$, the last constraint is replaced, for all $i\in[N]$, $k\in[K]$, and $l\in[L]$, by
\begin{align*}
s_{il}&\geq (a_k+d_k\odot\theta)^\top\hat\xi_i+b_k^\top\theta+c_k
  -\mu_{ikl}^\top\Bigl((A+D\mathrm{diag}(\beta_l))\hat\xi_i+B\beta_l+c\Bigr)\\
  &\quad +\zeta_{ikl}^\top(r-P\hat\xi_i)
  + \eta_{ikl}^\top(w-M\beta_l)
  +\rho\left\Vert-\left((B+D\mathrm{diag}(\hat\xi_i))^\top\mu_{ikl}+\tau_{il}\right)-M^\top\eta_{ikl}\right\Vert,\\
\lambda_l&\geq
\left\Vert a_k+d_k\odot\theta-(A+D\mathrm{diag}(\beta_l))^\top\mu_{ikl}-P^\top\zeta_{ikl}\right\Vert_*.
\end{align*}
For a fixed $\theta$, let $R(\theta)$ denote the true ex-ante regret and let $\widehat R_\rho(\theta)$ 
denote the value of the displayed relaxation
after fixing $\theta$. Under the compactness assumptions of Theorem~\ref{thm:relax_gap_delta}, if
\[
    L_{\xi,\Theta}
    =
    \max_{k\in[K]}\sup_{\vartheta\in\Theta}
    \Vert a_k+d_k\odot\vartheta\Vert_*,
\]
then
\[
    -L_D\rho\delta
    \leq
    \widehat R_\rho(\theta)-R(\theta)
	    \leq
	    2L_{\xi,\Theta}\delta+L_D\rho\delta.
	\]
	The same two-sided bound holds for the corresponding optimal values after minimizing both sides over $\theta\in\Theta$.
\end{lemma}

It remains to choose the centers $\beta_l$. A natural heuristic would be to compute $\theta_{\mathrm{ERM}}$ and then solve
the regret maximization at this fixed decision to obtain adversarial benchmarks. This is unattractive for two reasons: it
requires solving the same NP-hard problem that the relaxation is meant to avoid, and a global solver typically returns only
one maximizer even when several active benchmarks exist, which can bias the relaxation toward a particular direction.
As a tractable alternative we instead compute
\[
    \mathbb P^{\mathrm{wc}}\in\operatorname*{argmax}_{\mathbb P\in B_\delta^p(\hat{\mathbb P}_N)}
    \mathbb E_{\mathbb P}[\ell(\theta_{\mathrm{ERM}},X)],
    \qquad
    \mathbb P^{\mathrm{bc}}\in\operatorname*{argmin}_{\mathbb P\in B_\delta^p(\hat{\mathbb P}_N)}
    \mathbb E_{\mathbb P}[\ell(\theta_{\mathrm{ERM}},X)]
\]
using \cite[Theorem~9]{kuhn_wasserstein_2024}, then set
\[
    \beta^{\mathrm{wc}}\in\operatorname*{argmin}_{\beta\in\Theta}\mathbb E_{\mathbb P^{\mathrm{wc}}}[\ell(\beta,X)],
    \qquad
    \beta^{\mathrm{bc}}\in\operatorname*{argmin}_{\beta\in\Theta}\mathbb E_{\mathbb P^{\mathrm{bc}}}[\ell(\beta,X)].
\]
We use $\beta_1=\theta_{\mathrm{ERM}}$, $\beta_2=\beta^{\mathrm{wc}}$, and $\beta_3=\beta^{\mathrm{bc}}$ as default
centers. The construction can be iterated by resolving the relaxation and recomputing these centers at the new candidate
decision.

\section{Experiments}\label{sec:experiments}

This section evaluates the exact algorithm developed in Section~\ref{sec:exact_algs} and its convex relaxation from 
Section~\ref{sec:convex_relaxations} on a variety of decision problems, comparing their performance to ERM and DRO 
policies. Whenever exact WDRRO computation is tractable, we use it as a ground-truth benchmark for DRRO,
its relaxation, ERM, and DRO\@. In settings where exact regret computation is prohibitive, we use the experiments to assess
the scalability and qualitative behavior of the relaxation rather than to compute ground-truth regret gaps.
All experiments were run on a MacBook Air (M2, 2023)
with an 8-core Apple~M2 chip and 16\,GB of RAM\@. Reported runtimes correspond to wall-clock times measured on this hardware.

We consider three classes of problems:
\begin{enumerate}
    \item univariate newsvendor problems, including a comparison with the relaxation in \cite{cho_wasserstein_2024},
    \item a two-item newsvendor system with dependent demands, and
    \item a mean--CVaR portfolio allocation problem.
\end{enumerate}

\subsection{Newsvendor Problem}\label{sec:experiments_newsvendor}

We begin with the classical univariate newsvendor model. In Section~\ref{sec:motivation} we compared the ERM, DRO, and DRRO order quantities 
for the case where the wholesale price is small relative to the retail price. We now revisit this problem to evaluate the convex relaxation 
in more detail and additionally study scenarios with different buy--sell ratios.

Figure~\ref{fig:comparison_newsvendor_extended} displays the optimal order quantities and exact regrets for ERM, DRO, DRRO, and the DRRO 
relaxation for the setting with selling price $s=2$, wholesale price $b=0.1$, Wasserstein order $p=2$, and $N=1000$ samples drawn from $N(100,10^2)$. 
The relaxation nearly perfectly tracks the DRRO optimal order quantity for radii up to $\delta\approx 6$. 
The regret gap is even smaller, illustrating the tightness of the relaxation.

\begin{figure}[h!]
    \centering
    \includegraphics[width=0.8\textwidth]{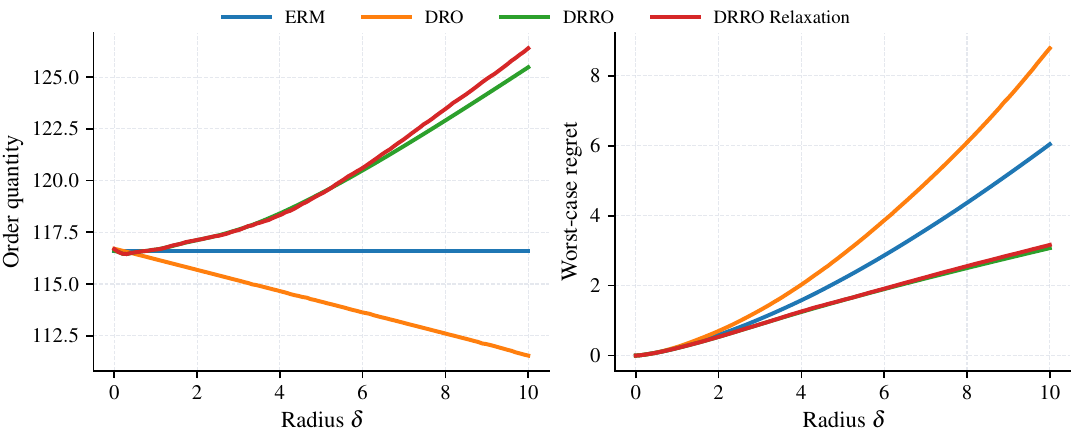}
    \caption{Comparison of ERM, DRO, DRRO, and DRRO relaxation for $b=0.1$, $s=2$, $N=1000$.}
    \label{fig:comparison_newsvendor_extended}
\end{figure}

Next we consider a case where the wholesale price is closer to the retail price, with $b=1.5$ and $s=2$. 
The order quantities and regrets are shown in Figure~\ref{fig:comparison_newsvendor_extended_other}. 
Once again, the relaxation tracks the true DRRO solution very closely across the full range of uncertainty radii shown, both in order quantity and regret.
A key difference between this scenario and the previous one is the qualitative behavior of the regret-optimal policy: 
here, the DRRO order quantity becomes \emph{more conservative} as the radius increases, in contrast to the earlier example 
where it became more aggressive. This is consistent with the underlying payoff asymmetry. When the wholesale price is negligible, 
overordering poses little downside risk while the upside from high demand is large. Thus aggressive ordering is beneficial. 
When the wholesale price is high relative to the retail price, overordering carries substantial downside risk, 
and conservatism becomes optimal. The threshold between these two behaviors occurs when $b=s/2$. 
Figure~\ref{fig:comparison_newsvendor_extended_other2} shows that at this price ratio ($b=1$, $s=2$), 
the DRRO policy is numerically indistinguishable from the ERM solution across all radii we tested. The relaxed DRRO policy 
overlaps with the exact DRRO solution as well, again illustrating the tightness of the relaxation in this experiment.

\begin{figure}[h!]
    \centering
    \includegraphics[width=0.8\textwidth]{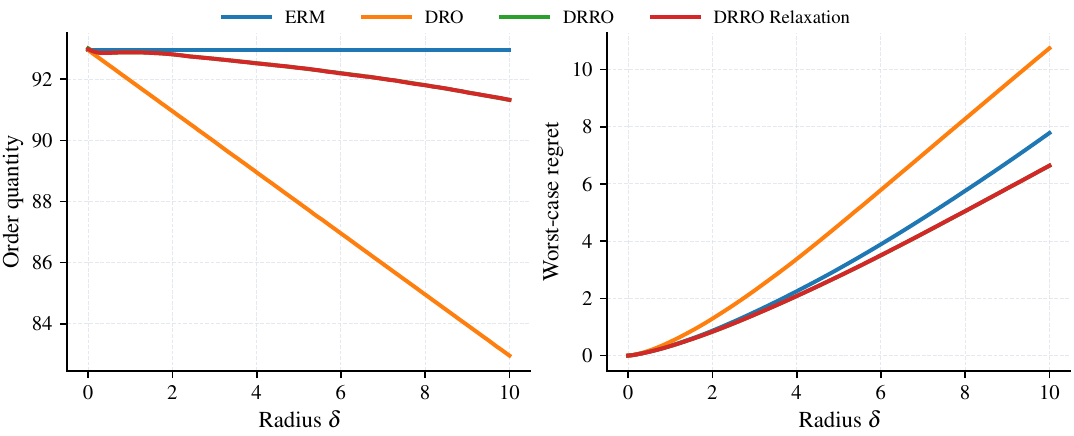}
    \caption{Comparison for $b=1.5$, $s=2$, $N=1000$. The DRRO and relaxed DRRO policies are numerically indistinguishable for all $\delta$.}
    \label{fig:comparison_newsvendor_extended_other}
\end{figure}

\begin{figure}[h!]
    \centering
    \includegraphics[width=0.8\textwidth]{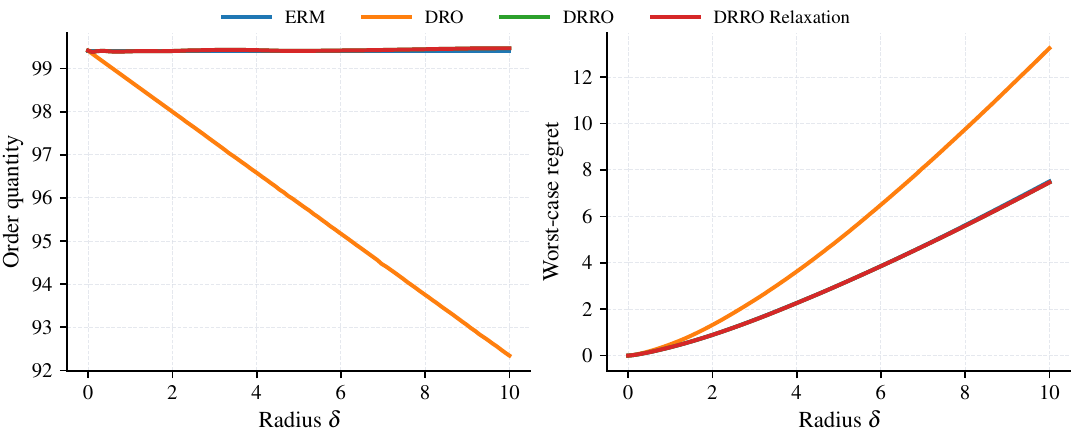}
    \caption{Behavior at the threshold ratio $b=s/2$. The DRRO solution is numerically indistinguishable from ERM for all $\delta$ shown.}
    \label{fig:comparison_newsvendor_extended_other2}
\end{figure}

\paragraph{Comparison with Cho and Yang.}
For this comparison, we switch to a type-$1$ Wasserstein distance because the relaxation in \cite{cho_wasserstein_2024} applies only when $p=1$. 
We consider the one-dimensional newsvendor setup with support and decision spaces $\Xi=\Theta=[0,100]$. We generate $N=300$ samples from 
$\mathbb{P}_0 = N^{-1}\sum_{i=1}^N \delta_{X_i}$, where $X_i\sim N(80,10^2)$ and values are clipped to remain in~$\Xi$. The wholesale and 
retail prices are $b=0.4$ and $s=2$, and we evaluate radii $\delta\in\{0,10,20,\ldots,100\}$. For each $\delta$ we repeat the experiment 
on $n_{\mathrm{sims}}=30$ independently drawn datasets and report averages.

Figure~\ref{fig:newsvendor_cho_comp1} compares the policies returned by the two relaxations. 
The left panel reports the true WDRRO regret of each policy, computed using the exact newsvendor algorithm. 
The right panel reports the model-specific upper bound returned by the relaxation that generated that policy. 
Figure~\ref{fig:gap_vs_delta_newsvendor_cho_comp} shows the corresponding gaps. 
Its left panel reports the difference between the true regret of each relaxation policy and the optimal WDRRO regret. 
Its right panel reports the difference between each model-specific upper bound and the true regret of the policy it certifies. 
As expected, both relaxations become 
exact for sufficiently large Wasserstein radii and also at $\delta=0$, in accordance with Theorem~\ref{thm:exact_relaxation}. 
Overall, the two relaxations deliver comparable regret performance. The relaxation in \cite{cho_wasserstein_2024} provides slightly 
tighter upper bounds, but neither method uniformly dominates the other in terms of the actual regret of the resulting policy.

The decisive distinction is computational. The relaxation in \cite{cho_wasserstein_2024} remains NP-hard and becomes rapidly 
impractical as $N$ increases, whereas our formulation yields a single convex optimization problem that scales efficiently to large sample sizes. 
In our experiments, both the exact regret-optimal policy and our relaxation require on the order of $0.1$ seconds to compute, while the 
relaxation in \cite{cho_wasserstein_2024} can take well over $90$ seconds for the same instances (see Figure~\ref{fig:relaxation_compute_time_cho_comp}).
This runtime separation can be made even larger by increasing $N$, since the NP-hard relaxation becomes increasingly difficult while our convex formulation remains directly scalable.
Since the exact regret-optimal policy for the one-dimensional newsvendor problem can be computed efficiently using the specialized algorithm
from Appendix~\ref{app:newsvendor_policy_algorithms}, this comparison should be viewed as a benchmark of relaxation quality rather than as a
recommended practical procedure for this particular model class.

\begin{figure}[h!]
    \centering
    \includegraphics[width=0.8\textwidth]{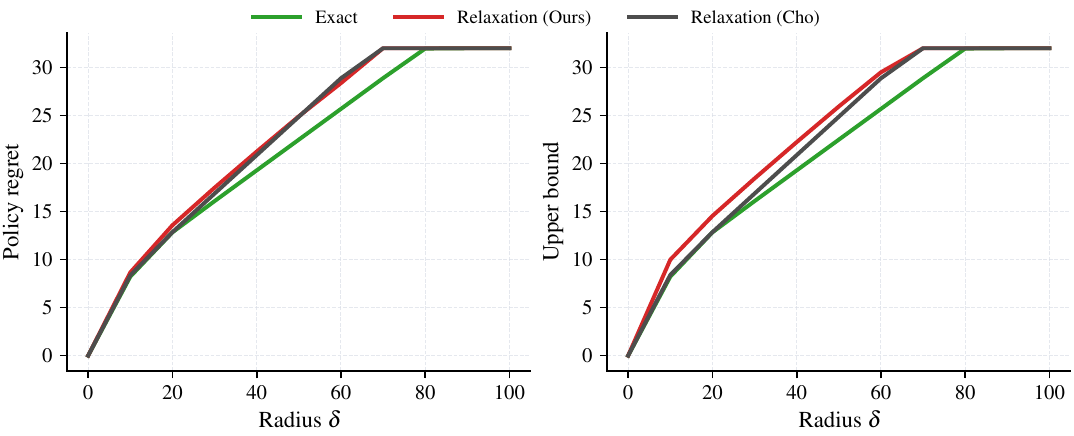}
    \caption{Regrets and corresponding upper bounds for our relaxation and the relaxation in \cite{cho_wasserstein_2024}.}
    \label{fig:newsvendor_cho_comp1}
\end{figure}

\begin{figure}[h!]
    \centering
    \includegraphics[width=0.8\textwidth]{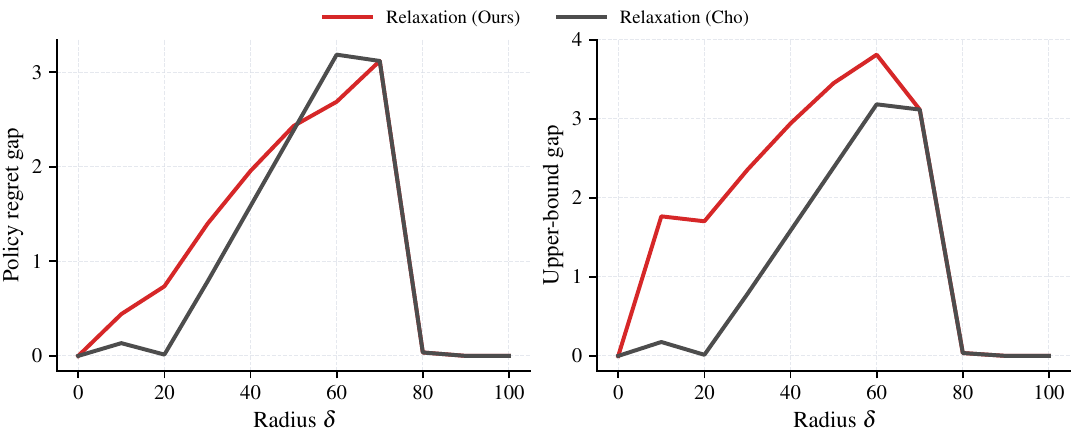}
    \caption{Gaps for the two relaxations: policy regret above the optimal WDRRO regret (left) and model-specific upper-bound gaps above true policy regret (right).}
    \label{fig:gap_vs_delta_newsvendor_cho_comp}
\end{figure}

\begin{figure}[h!]
    \centering
    \includegraphics[width=0.45\textwidth]{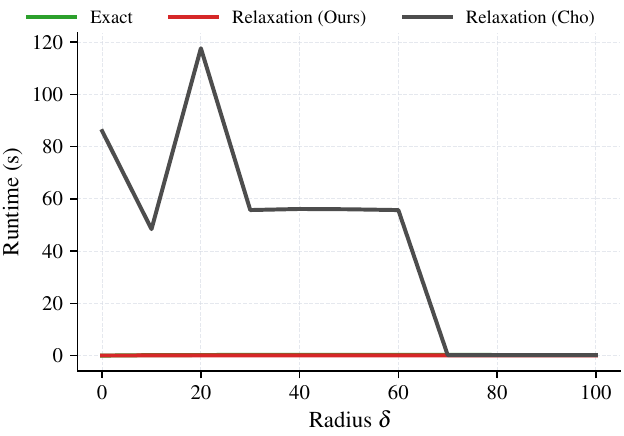}
    \caption{Computation time required to obtain the optimal policy under each relaxation.}
    \label{fig:relaxation_compute_time_cho_comp}
\end{figure}

\subsection{Two-Item Newsvendor Problem}

We now consider a setting with two products, $A$ and $B$, and a dependent demand effect: a fraction $\phi$ of customers purchasing $A$ will also purchase $B$, 
but only if $A$ is in stock. Let $\theta_A$ and $\theta_B$ denote the order quantities, $(b_A,b_B)$ the wholesale prices, and $(s_A,s_B)$ the retail prices. 
The negative profit can be written as a maximum of four affine functions in $(\theta_A,\theta_B)$, so the model fits Assumption~\ref{ass:model_choice} with $K=4$.
Here the bilinear coefficients satisfy $d_k=0$, so this example tests the general exact algorithm and the convex relaxation in the max-affine no-bilinear case.
Unlike the univariate newsvendor model, no specialized tractable algorithm based on the piecewise concavity of the regret objective is known for this two-dimensional setting.

We consider $D_A \sim N(32.5, 76.5625)$, $D_B \sim N(40,25)$, $b_A=b_B=6$, $s_A=20$, $s_B=7$, $\phi=0.1$, and $N=100$. We use type-$2$ Wasserstein ambiguity with an $\ell_2$ ground norm. The order quantities and demand support are restricted to $[0,100]^2$, and we evaluate radii $\delta\in\{0,1,\ldots,10\}$. The small sample size is chosen so 
that the exact DRRO problem can be solved via the cutting-plane algorithm, enabling a full comparison between the exact and relaxed models. Larger samples 
make the exact computation intractable, although the relaxed DRRO model remains easy to solve.

Figure~\ref{fig:two_item_newsvendor} displays the regret of all four methods and the optimal order quantities for products $A$ and $B$. DRRO and its relaxation significantly outperform DRO and ERM, and the 
regret of the relaxation is nearly indistinguishable from the true DRRO regret. A notable feature is that the optimal strategy for product $A$ becomes more aggressive as uncertainty increases, 
while the optimal strategy for product $B$ becomes aggressive initially but later turns conservative. This demonstrates that under regret minimization, 
different components of a vector decision variable may react differently to increasing uncertainty. The irregularities in the true DRRO $\theta_B$ are due 
to the flatness of the regret landscape in that direction; many different $\theta_B$ values achieve nearly identical regret. Nevertheless, 
the relaxation captures the overall trend extremely well.

\begin{figure}[h!]
    \centering
    \includegraphics[width=0.95\textwidth]{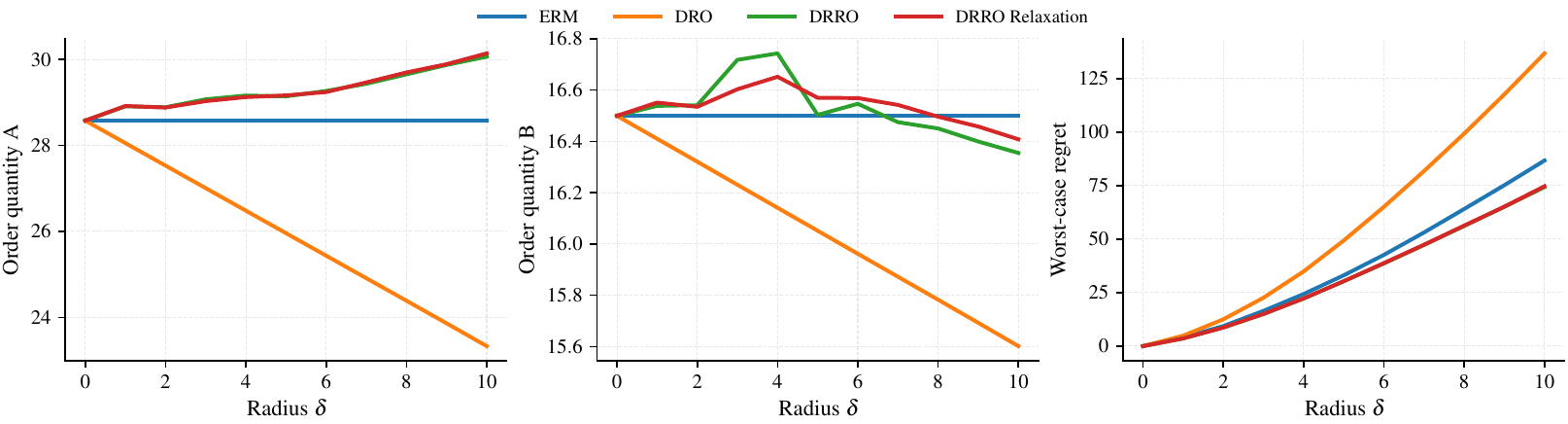}
    \caption{Regret and optimal order quantities for ERM, DRO, DRRO, and the DRRO relaxation in the two-item newsvendor model.}
    \label{fig:two_item_newsvendor}
\end{figure}

\subsection{Mean--Risk Portfolio Optimization}

Finally, we study a mean--CVaR portfolio allocation problem. Unlike the newsvendor experiments above, this setting exercises the general max-affine case with bilinear interaction terms $d_k\neq0$ and therefore tests the local-cover extension of the relaxation. Let $\xi\in\Xi=[-1,\infty)^d$ denote the random returns of $d$ risky 
assets, and let $r$ denote the deterministic return of a riskless asset. The lower bound excludes losses below $100\%$ on any risky
asset. We impose the no-short-selling and no-leverage condition by taking $\Theta=\{\theta\in\mathbb R_+^d:\vec{1}^{\top}\theta\leq 1\}$,
where $\theta_j$ is the portfolio weight invested in risky asset $j$. The remaining wealth $1-\vec{1}^{\top}\theta$ is invested
in the riskless asset, so the portfolio return is $R_\theta(\xi)=\theta^\top\xi+(1-\vec{1}^{\top}\theta)r$.
The corresponding loss is $L_\theta(\xi)=-R_\theta(\xi)$, and the classical mean--CVaR problem is
\begin{align}
    \min_{\theta\in\Theta}
    \mathbb{E}_{\mathbb{P}_0}[L_\theta(\xi)]
    +\rho\,\mathrm{CVaR}_{\alpha}^{\mathbb{P}_0}(L_\theta(\xi)).
    \label{eq:mean_cvar_original}
\end{align}
Using the standard variational representation of CVaR \citep{rockafellar2000optimization},
\[
\mathrm{CVaR}_{\alpha}^{\mathbb{P}}(L)
=
\inf_{\tau \in \mathbb{R}} \left\{
    \tau + \frac{1}{1-\alpha}\mathbb{E}_{\mathbb{P}}[\max\{L - \tau,0\}]
\right\},
\]
\eqref{eq:mean_cvar_original} can be written as
\begin{align}
    \min_{\theta\in\Theta}\inf_{\tau\in\mathbb{R}}
    \mathbb{E}_{\mathbb{P}_0}
    \left[
        L_\theta(\xi)
        +\rho\tau
        +\frac{\rho}{1-\alpha}\max\{L_\theta(\xi)-\tau,0\}
    \right].
    \label{eq:mean_cvar_with_tau}
\end{align}

For our WDRRO formulation, we pull the auxiliary CVaR variable $\tau$ out and treat $(\theta,\tau)$ as the decision variable. 
This yields a max-affine loss of the form required by Assumption~\ref{ass:model_choice},
\begin{align}
    \ell((\theta,\tau),\xi)
    =
    L_\theta(\xi)
    +\rho\tau
    +\frac{\rho}{1-\alpha}\max\{L_\theta(\xi)-\tau,0\}.
    \label{eq:mean_cvar_framework_loss}
\end{align}
To write this in the notation of Assumption~\ref{ass:model_choice}, set
$\vartheta=(\theta,\tau)\in\mathbb{R}^{d+1}$ and augment the uncertainty by
$\tilde\xi=(\xi,0)\in\mathbb{R}^{d+1}$. Then $\ell(\vartheta,\tilde\xi)$ is of max-affine form with the following coefficients
where $\lambda=\rho/(1-\alpha)$:
\[
\begin{array}{c|cccc}
    k & a_k & b_k & c_k & d_k \\ \hline
    1
    & 0
    & (r\vec{1},\,\rho)
    & -r
    & (-\vec{1},\,0) \\
    2
    & 0
    & ((1+\lambda)r\vec{1},\,\rho-\lambda)
    & -(1+\lambda)r
    & (-(1+\lambda)\vec{1},\,0).
\end{array}
\]
The framework-consistent DRRO problem is therefore
\begin{align}
    \inf_{\theta,\tau} \sup_{\mathbb{P}\in B_\delta^p(\mathbb{P}_0)}
    \left\{
        \mathbb{E}_{\mathbb{P}}[\ell((\theta,\tau),\xi)]
        -\inf_{\beta,\kappa}\mathbb{E}_{\mathbb{P}}[\ell((\beta,\kappa),\xi)]
    \right\}.
    \label{eq:actually_solved_cvar}
\end{align}
Conceptually, one could instead eliminate the auxiliary variable for each candidate policy before taking the worst case. We call this 
the nested-CVaR DRRO variant:
\begin{align}
    &\inf_{\theta} \sup_{\mathbb{P}\in B_\delta^p(\mathbb{P}_0)}
    \left\{
        \inf_{\tau}\mathbb{E}_{\mathbb{P}}[\ell((\theta,\tau),\xi)]
        -\inf_{\beta,\kappa}\mathbb{E}_{\mathbb{P}}[\ell((\beta,\kappa),\xi)]
    \right\} \notag \\
    &=
    \inf_{\theta} \sup_{\mathbb{P}\in B_\delta^p(\mathbb{P}_0)}
    \left\{
        \mathbb{E}_{\mathbb{P}}[L_\theta(\xi)]
        +\rho\,\mathrm{CVaR}_{\alpha}^{\mathbb{P}}(L_\theta(\xi))
        -\inf_{\beta}
        \left(
            \mathbb{E}_{\mathbb{P}}[L_\beta(\xi)]
            +\rho\,\mathrm{CVaR}_{\alpha}^{\mathbb{P}}(L_\beta(\xi))
        \right)
    \right\}.
    \label{eq:nested_cvar_drro}
\end{align}
In the WDRO setting this distinction is largely immaterial: under suitable integrability conditions one can invoke a minimax theorem 
to interchange the infimum over $\tau$ with the supremum over $\mathbb{P}$, so that the WDRO versions of 
\eqref{eq:actually_solved_cvar} and \eqref{eq:nested_cvar_drro} without the infimum over $(\beta,\kappa)$ coincide; see, e.g., 
\cite[Theorem~5.18]{kuhn2025distributionallyrobustoptimization}. For WDRRO, such an interchange is not justified in general, and 
the two regret problems need not agree. Nevertheless, \eqref{eq:actually_solved_cvar} is a conservative upper bound on 
\eqref{eq:nested_cvar_drro} by weak duality, and such conservative formulations are standard in robust optimization. We therefore 
use \eqref{eq:actually_solved_cvar} as the formulation consistent with our general WDRRO framework.

\subsubsection{One Risky Asset}

We first consider the case of a single risky asset, where the decision set reduces to $\Theta=[0,1]$ and $L_\theta(\xi)=-r+\theta(r-\xi)$.
By translation invariance and positive homogeneity of CVaR, for any distribution $\mathbb{P}$,
\begin{align}
    \mathbb{E}_{\mathbb{P}}[L_\theta(\xi)]
    +\rho\,\mathrm{CVaR}_{\alpha}^{\mathbb{P}}(L_\theta(\xi))
    =
    -(1+\rho)r
    +\theta\,c(\mathbb{P}),
    \label{eq:one_asset_mean_cvar_linear}
\end{align}
where
\[
    c(\mathbb{P})
    =
    \mathbb{E}_{\mathbb{P}}[r-\xi]
    +\rho\,\mathrm{CVaR}_{\alpha}^{\mathbb{P}}(r-\xi).
\]
Thus the ERM problem is linear in $\theta$: the investor is fully invested in the risky asset if 
$c(\hat{\mathbb{P}}_N)<0$, invests nothing in the risky asset if $c(\hat{\mathbb{P}}_N)>0$, and is indifferent if 
$c(\hat{\mathbb{P}}_N)=0$.

The same all-or-nothing structure holds for DRO\@. Indeed, the DRO objective is
\[
    -(1+\rho)r
    +\theta\,\overline c_\delta,
    \qquad
    \overline c_\delta
    =
    \sup_{\mathbb{P}\in B_\delta^p(\mathbb{P}_0)} c(\mathbb{P}).
\]
Hence the DRO policy is fully invested when $\overline c_\delta<0$ and invests nothing when $\overline c_\delta>0$. If investing is 
optimal at $\delta=0$ and $\overline c_\delta$ crosses zero as $\delta$ increases, the DRO policy therefore jumps instantly from 
$1$ to $0$.

The nested-CVaR DRRO variant admits a similarly explicit one-dimensional representation. Define $\underline c_\delta=\inf_{\mathbb{P}\in B_\delta^p(\mathbb{P}_0)} c(\mathbb{P})$
and $c_\delta=\sup_{\mathbb{P}\in B_\delta^p(\mathbb{P}_0)} c(\mathbb{P})$.
Using \eqref{eq:one_asset_mean_cvar_linear}, the regret under a fixed distribution is
\[
    \theta c(\mathbb{P})-\min_{\beta\in[0,1]}\beta c(\mathbb{P})
    =
    \max\{\theta c(\mathbb{P}),(\theta-1)c(\mathbb{P})\}.
\]
Therefore \eqref{eq:nested_cvar_drro} reduces to
\begin{align}
    \min_{\theta\in[0,1]}
    \max_{c\in[\underline c_\delta,\overline c_\delta]}
    \max\{\theta c,(\theta-1)c\}
    =
    \min_{\theta\in[0,1]}
    \max\left\{
        \theta\,\overline c_\delta,
        (1-\theta)(-\underline c_\delta)
    \right\}.
    \label{eq:nested_cvar_one_asset}
\end{align}
When $\underline c_\delta<0<\overline c_\delta$, the optimizer balances the two terms and is
\[
    \theta^{\mathrm{nested}}_\delta
    =
    \frac{-\underline c_\delta}{\overline c_\delta-\underline c_\delta}.
\]
If $\underline c_\delta\geq0$, the optimizer is $\theta^{\mathrm{nested}}_\delta=0$, while if $\overline c_\delta\leq0$, the optimizer is 
$\theta^{\mathrm{nested}}_\delta=1$, up to the usual nonuniqueness at equality. By contrast, the framework-consistent formulation 
\eqref{eq:actually_solved_cvar} keeps a single auxiliary variable $\tau$ outside the distributional supremum and does not admit this 
closed-form reduction.

In the one-asset experiment, we solve ERM, DRO, the framework-consistent DRRO problem \eqref{eq:actually_solved_cvar}, the convex 
relaxation of the same problem, and the nested-CVaR DRRO variant \eqref{eq:nested_cvar_drro}. We use $r=0$, $\rho=5$,
$\alpha=0.8$, type-$2$ Wasserstein ambiguity with an $\ell_2$ ground norm, and $N=10$ samples drawn from $N(0.04,0.03^2)$ and clipped to $[-1, \infty)$. The small sample size is chosen only so that we can
compute the ground-truth framework-consistent DRRO solution, which becomes costly for larger samples. ERM, DRO, the nested-CVaR
formula, and the convex DRRO relaxation remain efficiently computable at substantially larger sample sizes.

\begin{figure}[htbp]
    \centering
    \begin{subfigure}[t]{0.47\textwidth}
        \centering
        \raisebox{0.04\height}{\includegraphics[width=\linewidth]{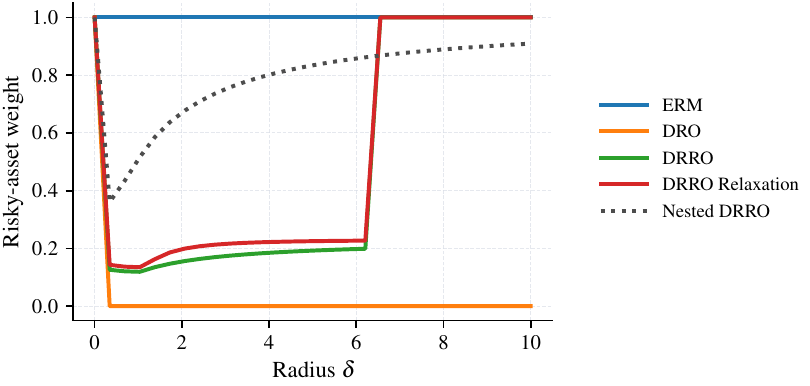}}
        \caption{Risky-asset weight as a function of the Wasserstein radius.}
        \label{fig:mean_cvar_1asset_thetas}
    \end{subfigure}
    \hfill
    \begin{subfigure}[t]{0.47\textwidth}
        \centering
        \includegraphics[width=\linewidth]{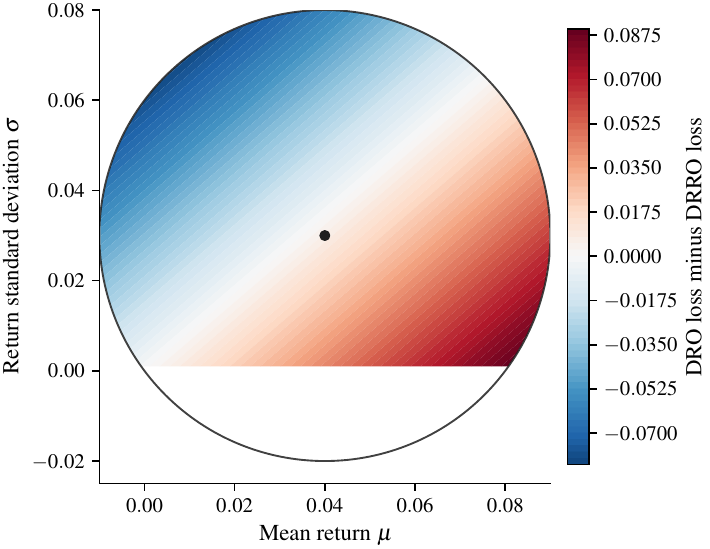}
        \caption{Gaussian mean--CVaR loss difference between DRO and DRRO at $\delta=0.05$.}
        \label{fig:mean_cvar_1asset_heatmap}
    \end{subfigure}
    \caption{One-asset mean--CVaR experiment. Positive heatmap values indicate distributions under which DRRO has lower loss than DRO.}
    \label{fig:mean_cvar_1asset}
\end{figure}

Figure~\ref{fig:mean_cvar_1asset_thetas} illustrates the qualitative difference between robust loss minimization and regret
minimization. ERM remains fully invested because the empirical mean--CVaR coefficient is favorable at the nominal distribution.
DRO is all-or-nothing: once the worst-case coefficient becomes nonnegative, it exits the risky asset completely. Both DRRO
formulations avoid this discontinuity. The framework-consistent DRRO policy first reduces exposure to limit regret against adverse
nearby distributions, but it remains invested and eventually returns to full investment. The DRRO relaxation tracks the exact framework-consistent DRRO policy closely enough to capture this nonmonotone behavior.
The nested-CVaR policy follows the scalar balancing formula \eqref{eq:nested_cvar_one_asset}; it moves more smoothly and
approaches full investment more gradually over the displayed range.

This large-radius behavior is not an artifact of the experiment. One can show that, as $\delta\to\infty$, both the
nested-CVaR and framework-consistent DRRO policies converge to full investment in the risky asset. This is natural
because the return support $\Xi=[-1,\infty)$ is asymmetric: downside losses are bounded by $100\%$, whereas upside
returns are unbounded. When the Wasserstein radius is large, the reference distribution carries little information,
and this support asymmetry becomes the dominant signal. Since the upside is unbounded, the regret-optimal large-radius
policy bets on the risky asset. The proof is a technical limiting argument, and we omit it to save space.

Figure~\ref{fig:mean_cvar_1asset_heatmap} evaluates the optimal DRO and DRRO policies at the smaller radius $\delta=0.05$ on a local Gaussian
slice around the data-generating distribution. The radius is small enough that the distributions in this ball place no significant mass
below the admissible return boundary $-1$, so evaluating the mean--CVaR objective under the Gaussian slice or its truncation at $-1$
is essentially identical. This radius corresponds to
moderate perturbations of the mean and standard deviation around $N(0.04,0.03^2)$, in contrast to the much larger radii in
Figure~\ref{fig:mean_cvar_1asset_thetas}, which are useful for displaying the full path of the policies over a wide range of
ambiguity levels. The plotted value is the mean--CVaR objective of the DRO policy minus that of the DRRO policy, so red regions
favor DRRO and blue regions favor DRO\@. For this instance at $\delta=0.05$, the computed DRO risky-asset weight is essentially
zero, while DRRO keeps a positive risky allocation. Consequently, DRRO performs better when the mean return is sufficiently high relative
to volatility, because it still captures upside that DRO has abandoned. When volatility is high or the mean return is low, the
riskless DRO decision can have lower mean--CVaR loss. The heatmap therefore shows the same phenomenon as in the motivating example from Section~\ref{sec:motivation}:
the regret-optimal policy improves performance on a substantial part of the ambiguity set
by avoiding the abrupt worst-case-loss-driven exit.

\subsubsection{Multiple Risky Assets}

Unlike in the single-asset case, the nested-CVaR DRRO policy for multiple assets no longer admits a closed-form scalar
representation, and computing the true framework-consistent DRRO policy is too costly at this sample size. We therefore evaluate
only the DRRO relaxation in this multi-asset experiment.

We use the return model from \cite{mohajerin_esfahani_data-driven_2018}. Our experiments are based on
a market with $d=10$ risky assets. In view of the capital asset pricing model, we assume that the return of asset $j$ decomposes
as $\xi_j=\psi+\zeta_j$, where the systematic risk factor $\psi\sim N(0,(2\%)^2)$ is common to all assets, while the
idiosyncratic risk factor $\zeta_j\sim N(3\%\,j,(2.5\%\,j)^2)$ is specific to asset $j$. The common factor and all idiosyncratic
factors are independent normal random variables. By construction, assets with higher indices have higher mean returns and higher
risk. We generate $N=300$ independent observations from this model, take $\alpha=0.8$, $\rho=10$, and riskless return $r=0$,
and evaluate radii up to $\delta_{\max}=10$. For this comparison we use
the type-$1$ Wasserstein distance with $\ell_1$ ground norm, matching the setting used in the original portfolio experiment.
For the DRRO relaxation, we use the iterative refinement described in Section~\ref{sec:convex_relaxations} with five
refinement steps.

\begin{figure}[htbp]
    \centering
    \includegraphics[width=0.8\textwidth]{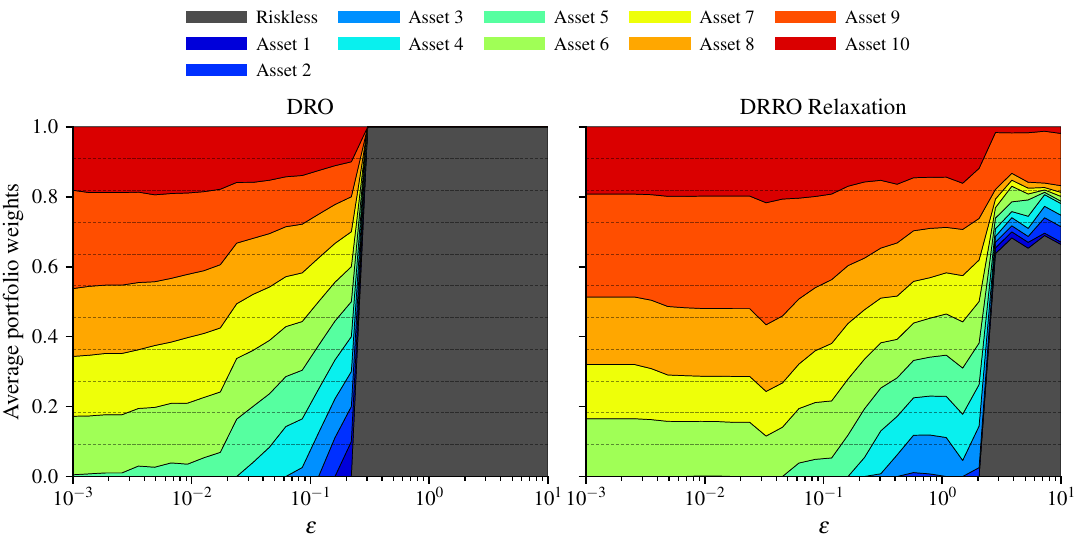}
    \caption{Portfolio weights for DRO and the DRRO relaxation with five refinement steps in the multi-asset mean--CVaR experiment with a riskless asset.}
    \label{fig:mean_cvar_multiasset}
\end{figure}

Figure~\ref{fig:mean_cvar_multiasset} shows the resulting portfolio weights. The DRO policy first reallocates across the risky
assets, but then jumps to the riskless asset as the radius grows. The DRRO relaxation is less conservative in this experiment and
continues to allocate positive weight to risky assets over the same range of radii.

For the exact regret formulations, one can again show that the large-radius limit is fully risky: both the nested-CVaR
and framework-consistent DRRO policies converge to the uniform portfolio over the risky assets, $\theta_j=1/d$, leaving
no wealth in the riskless asset. By contrast, DRO eventually invests fully in the riskless asset. The proof is a
technical limiting argument, which we omit. The intuition is the same as in the one-asset case: for very large
Wasserstein radii the reference distribution becomes largely uninformative, while the support still reveals that all
risky assets have bounded downside and unbounded upside. Symmetry across assets then leads regret minimization to spread
wealth uniformly over the risky assets. This asymptotic behavior is not yet visible in Figure~\ref{fig:mean_cvar_multiasset}
because the displayed radii are still moderate. We do not report much larger radii here because this multi-asset example
uses only the relaxation, whose approximation quality can deteriorate in the large-radius regime.

\section{Conclusion}\label{sec:conclusion}

This paper develops a theory of Wasserstein DRRO organized around a simple dichotomy. In regular
small-radius regimes, WDRRO is not a reason to abandon ERM: it selects among nonunique ERM optima,
has zero first-order sensitivity when the ERM optimizer is unique, and coincides exactly with ERM
for convex quadratic losses. Thus regret-based robustification is most useful outside this benign
regime, where nonsmoothness or non-infinitesimal radii can make WDRRO select genuinely different policies.

Those same regimes are computationally difficult. We showed that WDRRO regret evaluation is NP-hard
even for simple max-affine losses, while identifying tractable structure in the univariate newsvendor
problem through Wasserstein piecewise concavity. For the broader max-affine class, we developed exact
algorithms for ground-truth computation and convex relaxations with finite-radius guarantees,
including a local-cover extension for bilinear interaction terms. The experiments show that these
relaxations closely track exact WDRRO when comparison is possible and capture loss-dependent
aggressive or conservative behavior that worst-case-loss DRO misses.

\section*{Acknowledgments}
This material is based upon work supported by the Air Force Office of Scientific Research under award number FA9550-20-1-0397, and by the National Science Foundation under Grant Nos. 2118199 and 2229012.

Any opinions, findings, and conclusions or recommendations expressed in this material are those of the author(s) and do not necessarily reflect the views of the National Science Foundation.

\bibliographystyle{abbrvnat}
\bibliography{references}

\clearpage
\appendix
\section{DRRO for the Univariate Newsvendor Problem}\label{app:newsvendor}

In this section we prove the piecewise concavity of the function $h(\beta; \theta)$ and describe in detail the 
algorithm for computing the DRRO optimal order quantity in the univariate newsvendor problem. 

\subsection{\texorpdfstring{Reformulations of $h(\beta; \theta)$ as Finite Convex Programs}{Reformulations of h(beta; theta) as Finite Convex Programs}}\label{app:newsvendor_reformulations}

We start by describing a reformulation of $h(\beta; \theta)$ from \eqref{eq:g_def} as the solution of a 
finite-dimensional convex optimization problem. This representation corresponds to Theorem~9 in 
\cite{kuhn_wasserstein_2024}, which we also used in Section~\ref{sec:algorithms}. We begin with the case $\beta\leq \theta$ where
\begin{align*}
    \ell(\theta, x) - \ell(\beta, x) &=-(s - b)(\theta - \beta) + s \max(0, \min(\theta - \beta, \theta - x)).
\end{align*}
Since the random variable $x$ does not appear in the first summand, we can focus on just the second 
summand when computing the supremum over $\mathbb{P}\in B_\delta^p(\mathbb{P}_0)$. 
An application of \cite[Theorem~9]{kuhn_wasserstein_2024} to $\tilde\ell(x) = \max(0, \min(\theta - \beta, \theta - x))$ 
gives that $\sup_{\mathbb{P}\in B_\delta^p(\mathbb{P}_0)}\mathbb{E}_{\mathbb{P}}[\tilde\ell(x)]$ is equal to
\begin{align}
\begin{aligned}
\max \quad 
& \frac{1}{N}\sum_{i=1}^N \min\left(\gamma_i(\theta - \beta),\, \gamma_i(\theta - \hat{\xi}_i) - q_i\right) \\[1mm]
\mathrm{s.t.}\quad 
& \gamma_i \in [0, 1], \quad q_i \in \mathbb{R}, &\quad &\forall i \in [N], \\[1mm]
& \gamma_i\,\hat{\xi}_i + q_i \ge 0, &\quad &\forall i \in [N], \\[1mm]
& \frac{1}{N}\sum_{i=1}^N \gamma_i \left|\frac{q_i}{\gamma_i}\right|^p \le \delta^p.
\end{aligned}
\label{eq:left_maximization}
\end{align}

When $\theta \leq \beta$, we can write
\[
    \ell(\theta, x) - \ell(\beta, x)
    =
    b(\theta - \beta) + s \max(0, \min(x - \theta, \beta - \theta)).
\]
Again, it suffices to investigate $\tilde\ell(x) = \max(0, \min(x - \theta, \beta - \theta))$. 
Theorem~9 in \cite{kuhn_wasserstein_2024} again allows us to rewrite this as
\begin{align}
\begin{aligned}
\max \quad 
& \frac{1}{N}\sum_{i=1}^N  \min\left(\gamma_i(\hat{\xi}_i - \theta) + q_i,\, \gamma_i(\beta - \theta)\right) \\[1mm]
\mathrm{s.t.}\quad 
& \gamma_i \in [0, 1], \quad q_i \in \mathbb{R}, &\quad &\forall i \in [N], \\[1mm]
& \gamma_i\,\hat{\xi}_i + q_i \ge 0, &\quad &\forall i \in [N], \\[1mm]
& \frac{1}{N}\sum_{i=1}^N \gamma_i \left|\frac{q_i}{\gamma_i}\right|^p \le \delta^p.
\end{aligned}
\label{eq:right_maximization}
\end{align}

Both \eqref{eq:left_maximization} and \eqref{eq:right_maximization} have the same reduced form.  Define
\[
\begin{array}{c|ccc}
& a & d_i & r_i \\ \hline
\beta\leq\theta & \theta-\beta & (\hat\xi_i-\beta)_+ & -q_i \\
\theta\leq\beta & \beta-\theta & (\beta-\hat\xi_i)_+ & q_i.
\end{array}
\]
The support constraint becomes $r_i\leq\gamma_i\hat\xi_i$ in the left case and
$r_i\geq-\gamma_i\hat\xi_i$ in the right case.  In either case, we may restrict without loss to
$0\leq r_i\leq\gamma_i d_i$: negative $r_i$ only lowers the objective and uses budget, while values above
$\gamma_i d_i$ can be truncated to $\gamma_i d_i$ without changing the minimum and with weakly smaller budget
usage.  Since $\beta$ and the demand samples are nonnegative, this interval also implies the corresponding support
constraint.  Under this restriction, the minimum equals $\gamma_i(a-d_i)+r_i$.  Using the usual perspective
convention, so that $\gamma_i=0$ forces $r_i=0$, both programs reduce to
\begin{align}
\begin{aligned}
\max \quad 
& \frac{1}{N}\sum_{i=1}^{N} (\gamma_i(a-d_i) + r_i ) \\
\text{s.t.}\quad
& 0 \leq \gamma_i \leq 1, \qquad i=1,\dots,N, \\
& 0 \leq r_i \leq \gamma_i d_i, \qquad i=1,\dots,N, \\
& \frac{1}{N}\sum_{i=1}^N \frac{r_i^p}{\gamma_i^{p-1}} \le \delta^p.
\end{aligned}
\label{eq:canonical_newsvendor_maximization}
\end{align}

\subsection{\texorpdfstring{Algorithms for Evaluating $h(\beta;\theta)$ and Piecewise Concavity}{Algorithms for Evaluating h(beta; theta) and Piecewise Concavity}}
\label{app:newsvendor_wc_algorithms}

The goal of this section is to derive tractable algorithms for solving \eqref{eq:left_maximization} and
\eqref{eq:right_maximization}, equivalently \eqref{eq:canonical_newsvendor_maximization}, and then use the structure
of the constructed optimizers to prove the piecewise concavity of $h(\beta;\theta)$.  We focus on
$\beta\leq\theta$ as the case $\beta\geq\theta$ is symmetric.  Let $\beta_1\leq\beta_2\leq\theta$ and suppose that
we can construct optimal solutions $(\gamma^*(\beta_j),q^*(\beta_j))$, $j=1,2$, of \eqref{eq:left_maximization}
such that $\gamma_i^*(\beta_1)\leq\gamma_i^*(\beta_2)$ for all $i$.  For any convex weight $\lambda\in[0,1]$, the
convex combination $(\gamma_\lambda,q_\lambda)=\lambda(\gamma^*(\beta_1),q^*(\beta_1))
+(1-\lambda)(\gamma^*(\beta_2),q^*(\beta_2))$ is feasible because the feasible set of
\eqref{eq:left_maximization} is convex and independent of $\beta$.  Set
$\beta_\lambda=\lambda\beta_1+(1-\lambda)\beta_2$ and abbreviate
\[
    A_{ij}:=\gamma_i^*(\beta_j)(\theta-\beta_j),
    \qquad
    B_{ij}:=\gamma_i^*(\beta_j)(\theta-\hat\xi_i)-q_i^*(\beta_j).
\]
For the minimum in the $i$th summand of \eqref{eq:left_maximization}, the second argument at
$(\beta_\lambda,\gamma_\lambda,q_\lambda)$ is affine under convex combinations:
$\gamma_{i,\lambda}(\theta-\hat\xi_i)-q_{i,\lambda}
=\lambda B_{i1}+(1-\lambda)B_{i2}$.  For the first argument, monotonicity of the constructed optimizers gives
\[
    \gamma_{i,\lambda}(\theta-\beta_\lambda)
    -
    \bigl(\lambda A_{i1}+(1-\lambda)A_{i2}\bigr)
    =
    \lambda(1-\lambda)(\beta_2-\beta_1)
    \bigl(\gamma_i^*(\beta_2)-\gamma_i^*(\beta_1)\bigr)
    \geq 0.
\]
Since the minimum is nondecreasing in each argument and concave, the $i$th objective term at
$(\beta_\lambda,\gamma_\lambda,q_\lambda)$ is at least
$\lambda\min(A_{i1},B_{i1})+(1-\lambda)\min(A_{i2},B_{i2})$.  After summing over $i$ and dividing by $N$, the
objective value of this feasible point at $\beta_\lambda$ is therefore at least the convex combination of the two
optimal values. The optimal value at $\beta_\lambda$ is even larger, proving concavity on $[0,\theta]$.
We now treat the cases $p=1$ and $p>1$ separately, construct optimizers, and verify the required monotonicity of
the optimal $\gamma_i$'s.

\subsubsection{\texorpdfstring{The case $p=1$}{The case p=1}}

For $p=1$, the budget constraint in \eqref{eq:canonical_newsvendor_maximization} becomes
$\sum_{i=1}^N r_i\leq N\delta$ and is independent of $\gamma$.  Fixing $r_i$, the optimal choice of
$\gamma_i$ is immediate.  If $a\geq d_i$, then the coefficient of $\gamma_i$ in the objective is
nonnegative, and we choose $\gamma_i=1$.  If $a<d_i$, then the coefficient is negative, and we choose the
smallest feasible value, $\gamma_i=r_i/d_i$.  Thus, after eliminating $\gamma$, the problem becomes
\begin{align}
\begin{aligned}
\max \quad
& \frac{1}{N}\left[
    \sum_{i:a\geq d_i}(a-d_i)
    + \sum_{i=1}^N w_i r_i
\right] \\
\text{s.t.}\quad
& 0\leq r_i\leq d_i,\qquad i=1,\dots,N,\\
& \sum_{i=1}^N r_i\leq N\delta,
\end{aligned}
\label{eq:p1_newsvendor_knapsack}
\end{align}
where
\[
    w_i=
    \begin{cases}
        1, & a\geq d_i,\\
        a/d_i, & a<d_i.
    \end{cases}
\]
For the left problem, $a=\theta-\beta$ and $d_i=(\hat\xi_i-\beta)_+$, so $a\geq d_i$ is equivalent to
$\hat\xi_i\leq\theta$.  If the samples are sorted increasingly, then for all points with
$\hat\xi_i>\beta$ the coefficients $w_i$ are nonincreasing in $i$: they are equal to one for
$\beta<\hat\xi_i\leq\theta$ and equal to $(\theta-\beta)/(\hat\xi_i-\beta)$ for
$\hat\xi_i>\theta$.  Hence \eqref{eq:p1_newsvendor_knapsack} is a fractional knapsack problem whose optimal
solution fills the capacities $d_i$ from left to right among the points with $\hat\xi_i>\beta$, using a fractional
amount only for the first point at which the remaining budget is insufficient.

Equivalently, with $B=N\delta$ and $P_i(\beta)=\sum_{\ell<i}d_\ell(\beta)$, an optimal solution is
\[
    r_i^*(\beta)=\min\{d_i(\beta),\, (B-P_i(\beta))_+\}.
\]
The corresponding optimizer of \eqref{eq:canonical_newsvendor_maximization} is obtained by taking
\[
    \gamma_i^*(\beta)=
    \begin{cases}
        1, & d_i(\beta)\leq a(\beta),\\
        r_i^*(\beta)/d_i(\beta), & d_i(\beta)>a(\beta).
    \end{cases}
\]
The greedy rule is optimal by the standard exchange argument: if $i<j$, $w_i\geq w_j$, $r_i<d_i$, and
$r_j>0$, shifting a small amount of budget from $j$ to $i$ weakly increases the objective.  Finally, this
construction has the required monotonicity in $\beta$.  For $\hat\xi_i\leq\theta$ we have
$\gamma_i^*(\beta)=1$.  For $\hat\xi_i>\theta$, $P_i(\beta)$ is nonincreasing and $d_i(\beta)$ is decreasing in
$\beta$, and therefore
\[
    \gamma_i^*(\beta)
    =
    \min\left\{1,\frac{(B-P_i(\beta))_+}{d_i(\beta)}\right\}
\]
is nondecreasing in $\beta$.  Hence $\gamma_i^*(\beta_1)\leq\gamma_i^*(\beta_2)$ for all
$\beta_1\leq\beta_2\leq\theta$.

The resulting evaluator for the full function $h(\beta;\theta)$ and an optimal solution
$(\gamma^*,q^*)$ is summarized in
Algorithm~\ref{alg:newsvendor_p1}. The first step records which side of $\theta$ the candidate $\beta$
lies on, because the canonical maximization only computes the max-min term. The constant part of
$\ell(\theta,x)-\ell(\beta,x)$ is added back in the final line. Since $w_i$ is nonincreasing in $d_i$,
the greedy step fills the coordinates in increasing order of $d_i$. With the samples sorted increasingly, this
is the sample order when $\beta\leq\theta$ and the reverse sample order when $\beta>\theta$.

\begin{algorithm}[!ht]
\caption{Evaluation of $h(\beta;\theta)$ and an optimizer for $p=1$}
\label{alg:newsvendor_p1}
\begin{algorithmic}[1]
\Require Sorted samples $\hat\xi_1\leq\cdots\leq\hat\xi_N$, radius $\delta$, prices $b,s$, decisions $\beta,\theta$
\If{$\beta\leq\theta$}
    \State $a\gets \theta-\beta$, $c_0\gets -(s-b)a$, $d_i\gets(\hat\xi_i-\beta)_+$ for $i=1,\ldots,N$
    \State $\eta\gets -1$
    \State $\sigma(j)\gets j$ for $j=1,\ldots,N$
\Else
    \State $a\gets \beta-\theta$, $c_0\gets -ba$, $d_i\gets(\beta-\hat\xi_i)_+$ for $i=1,\ldots,N$
    \State $\eta\gets 1$
    \State $\sigma(j)\gets N+1-j$ for $j=1,\ldots,N$
\EndIf
\State $B\gets N\delta$
\State $R\gets B$
\State $q_i^*\gets0$, and $\gamma_i^*\gets\mathbf{1}_{\{d_i\leq a\}}$ for $i=1,\ldots,N$
\State $V\gets \sum_{i:d_i\leq a}(a-d_i)$
\For{$j=1,\ldots,N$}
    \State $i\gets\sigma(j)$
    \If{$R=0$}
        \State \textbf{break}
    \EndIf
    \If{$d_i\leq a$}
        \State $w_i\gets 1$
    \Else
        \State $w_i\gets a/d_i$
    \EndIf
    \State $r_i^*\gets \min\{d_i,R\}$
    \State $q_i^*\gets \eta r_i^*$
    \If{$d_i>a$}
        \State $\gamma_i^*\gets r_i^*/d_i$
    \EndIf
    \State $V\gets V+w_i r_i^*$
    \State $R\gets R-r_i^*$
\EndFor
\State \Return $\left(c_0+sV/N,\gamma^*,q^*\right)$
\end{algorithmic}
\end{algorithm}

\subsubsection{\texorpdfstring{The case $p>1$}{The case p>1}}

For $p>1$, we derive a water-filling solution. Put $B=N\delta^p$ and multiply the objective in
\eqref{eq:canonical_newsvendor_maximization} by $N$.  Write $r_i=\gamma_i x_i$, with
$x_i\in[0,d_i]$. If $\gamma_i=0$, then $r_i=0$ and $x_i$ is immaterial.  The problem becomes
\[
    \max_{(\gamma,x)\in C}\sum_{i=1}^N\gamma_i(a-d_i+x_i)
    \quad\text{s.t.}\quad
    \sum_{i=1}^N\gamma_i x_i^p\leq B,
\]
where $C=\{(\gamma,x):0\leq\gamma_i\leq 1,\ 0\leq x_i\leq d_i\}$ and the displayed objective is denoted by $f_a(\gamma,x)$.  This change of variables makes the
program nonconvex in $(\gamma,x)$, but this is harmless because we only use the following global Lagrangian
certificate.  For $\lambda\geq0$, let
\[
    L_a(\gamma,x,\lambda)
    =
    \sum_{i=1}^N\gamma_i(a-d_i+x_i)
    +\lambda\left(B-\sum_{i=1}^N\gamma_i x_i^p\right).
\]
If $(\gamma^*,x^*)\in C$ is feasible, $\lambda^*\geq0$,
$\lambda^*(B-\sum_i\gamma_i^*(x_i^*)^p)=0$, and
$(\gamma^*,x^*)\in\arg\max_{(\gamma,x)\in C}L_a(\gamma,x,\lambda^*)$, then $(\gamma^*,x^*)$ is globally
optimal, since every feasible $(\gamma,x)$ satisfies
$f_a(\gamma,x)\leq L_a(\gamma,x,\lambda^*)\leq L_a(\gamma^*,x^*,\lambda^*)=f_a(\gamma^*,x^*)$.

For fixed $\lambda$ and $\gamma_i$, the $i$th term of the Lagrangian requires maximizing
$x_i-\lambda x_i^p$ over $x_i\in[0,d_i]$. The maximizer is $x_i(\lambda)=\min\{d_i,t(\lambda)\}$ where $t(\lambda)=(p\lambda)^{-1/(p-1)}$ for $\lambda>0$
and $t(0)=\infty$, and the maximum value is
\[
    \phi_i(\lambda)=
    \begin{cases}
        d_i-\lambda d_i^p, & 0\leq\lambda\leq \dfrac{1}{p d_i^{p-1}},\\[1.1em]
        \dfrac{p-1}{p}(p\lambda)^{-1/(p-1)}, & \lambda>\dfrac{1}{p d_i^{p-1}}.
    \end{cases}
\]
This is the water-filling step. The multiplier $\lambda$ determines a common level $t(\lambda)$, and each
coordinate is filled up to the smaller of this level and its cap $d_i$.
After eliminating $x_i$, the coefficient of $\gamma_i$ is $m_i(\lambda,a)=a-d_i+\phi_i(\lambda)$.
The remaining maximization is therefore
$\max_{0\leq\gamma_i\leq1}\gamma_i m_i(\lambda,a)$, so $\gamma_i=1$ if $m_i>0$, $\gamma_i=0$ if
$m_i<0$, while the case $m_i=0$ is not determined by this maximization alone and will be fixed below.

For $d_i>0$, define $\Lambda_i(a)=\sup(\{\lambda\geq0:m_i(\lambda,a)>0\}\cup\{0\})$. If
$d_i=0$, set $\Lambda_i(a)=\infty$, since this coordinate does not affect the budget. Then
\begin{equation}
\label{eq:newsvendor_lambda_threshold}
    \Lambda_i(a)=
    \begin{cases}
        \infty, & a\geq d_i,\\[0.7em]
        \dfrac{a}{d_i^p}, & 0\leq a\leq \dfrac{d_i}{p},\\[1.1em]
        \dfrac{(p-1)^{p-1}}{p^p(d_i-a)^{p-1}}, & \dfrac{d_i}{p}<a<d_i.
    \end{cases}
\end{equation}
Indeed, let $\bar\lambda_i=1/(p d_i^{p-1})$. For $\lambda\leq\bar\lambda_i$, we have
$m_i(\lambda,a)=a-\lambda d_i^p$. If $a\leq d_i/p$, this term crosses zero at
$a/d_i^p\leq\bar\lambda_i$. In the same case, the second-branch expression is nonpositive at
$\bar\lambda_i$ and decreases for $\lambda>\bar\lambda_i$, so there is no second crossing. If
$d_i/p<a<d_i$, the first-branch expression remains positive at $\bar\lambda_i$, and the unique crossing
occurs in the second branch, where $x_i(\lambda)<d_i$. Solving
$a-d_i+\frac{p-1}{p}(p\lambda)^{-1/(p-1)}=0$ gives the third line. Finally, if
$a\geq d_i$, no finite crossing occurs. For fixed $a$, these thresholds are nonincreasing in $d_i$, so we sort
indices so that $\Lambda_1(a)\geq\cdots\geq\Lambda_N(a)$, breaking ties deterministically.

We first treat the case $a>0$. If $B\geq\sum_i d_i^p$, then $\gamma_i^*=1$, $x_i^*=d_i$, and $r_i^*=d_i$ is optimal: each summand is at most
$a$, and this choice attains that upper bound.  Suppose next that $B<\sum_i d_i^p$.  Then the budget
constraint binds.  If it were slack and some coordinate with positive $\gamma_i$ had $x_i<d_i$, increasing
$x_i$ slightly would improve the objective.  Otherwise, all positive-mass coordinates have $x_i=d_i$. Since
$B<\sum_i d_i^p$, not all coordinates with $d_i>0$ can have $\gamma_i=1$, so choose one with
$d_i>0$ and $\gamma_i<1$. If $\gamma_i=0$, set $x_i=d_i$ first, which changes neither the objective nor
the budget. Since the budget is slack, increasing this $\gamma_i$ by a sufficiently small amount remains
feasible and increases the objective by $a>0$ per unit.

For the non-full-budget case with $a>0$, choose $\lambda^*$ such that
\[
    D_>(\lambda^*;a)
    \leq B\leq
    D_{\geq}(\lambda^*;a),
    \qquad
    D_>(\lambda;a)=\sum_{\Lambda_i(a)>\lambda}x_i(\lambda)^p,
    \quad
    D_{\geq}(\lambda;a)=\sum_{\Lambda_i(a)\geq\lambda}x_i(\lambda)^p.
\]
The functions $D_>$ and $D_{\geq}$ are nonincreasing in $\lambda$, so $\lambda^*$ can be found by bisection.
Set $x_i^*=x_i(\lambda^*)$ and, in the sorted order,
\[
    \gamma_i^*=
    \begin{cases}
        1, & x_i^*=0,\\[0.6em]
        \min\left\{1,\dfrac{\left(B-\sum_{k<i}(x_k^*)^p\right)_+}{(x_i^*)^p}\right\}, & x_i^*>0,
    \end{cases}
\]
Finally set $r_i^*=\gamma_i^*x_i^*$. For the left problem this corresponds to $q_i^*=-r_i^*$, while for
the right problem it corresponds to $q_i^*=r_i^*$. If the finite thresholds are distinct, at most one $\gamma_i^*$ is
fractional.  If several thresholds equal $\lambda^*$, the deterministic ordering simply fills those tied indices
one at a time. This is optimal because all tied boundary indices have $m_i(\lambda^*,a)=0$.

The case $a=0$ is simpler. In that case every summand satisfies
$\gamma_i(x_i-d_i)\leq0$, so the optimal value is zero. One optimal solution is obtained by taking
$\gamma_i^*=0$ and $q_i^*=0$ for all $i$.

Algorithm~\ref{alg:newsvendor_pgt1} gives the corresponding evaluator for the full function
$h(\beta;\theta)$ and an optimal solution $(\gamma^*,q^*)$ when $p>1$. We again assume that the
samples are sorted increasingly. Since the thresholds $\Lambda_i(a)$ are nonincreasing in $d_i$, the
required threshold order is the sample order when $\beta\leq\theta$ and the reverse sample order when
$\beta>\theta$.

\begin{algorithm}[H]
\caption{Evaluation of $h(\beta;\theta)$ and an optimizer for $p>1$}
\label{alg:newsvendor_pgt1}
\begin{algorithmic}[1]
\Require Sorted samples $\hat\xi_1\leq\cdots\leq\hat\xi_N$, radius $\delta$, order $p>1$, prices $b,s$, decisions $\beta,\theta$
\If{$\beta\leq\theta$}
    \State $a\gets \theta-\beta$, $c_0\gets -(s-b)a$, $d_i\gets(\hat\xi_i-\beta)_+$ for all $i$
    \State $\eta\gets -1$
    \State $\sigma(j)\gets j$ for $j=1,\ldots,N$
\Else
    \State $a\gets \beta-\theta$, $c_0\gets -ba$, $d_i\gets(\beta-\hat\xi_i)_+$ for all $i$
    \State $\eta\gets 1$
    \State $\sigma(j)\gets N+1-j$ for $j=1,\ldots,N$
\EndIf
\If{$a=0$}
    \State $\gamma_i^*\gets0$, $q_i^*\gets0$ for $i=1,\ldots,N$
    \State \Return $\left(0,\gamma^*,q^*\right)$
\EndIf
\State $B\gets N\delta^p$
\If{$B=0$}
    \State $\gamma_i^*\gets\mathbf{1}_{\{d_i\leq a\}}$, $q_i^*\gets0$ for $i=1,\ldots,N$
    \State \Return $\left(c_0+\dfrac{s}{N}\sum_{i=1}^N(a-d_i)_+,\gamma^*,q^*\right)$
\EndIf
\If{$B\geq \sum_{i=1}^N d_i^p$}
    \State $\gamma_i^*\gets1$, $q_i^*\gets\eta d_i$ for $i=1,\ldots,N$
    \State \Return $\left(c_0+sa,\gamma^*,q^*\right)$
\EndIf
\State Compute $\Lambda_i(a)$ for $i=1,\ldots,N$ using \eqref{eq:newsvendor_lambda_threshold}
\State For $\lambda>0$, write $x_i(\lambda)=\min\{d_i,(p\lambda)^{-1/(p-1)}\}$
\State Find $\lambda^*>0$ by bisection such that $D_>(\lambda^*)\leq B\leq D_{\geq}(\lambda^*)$, where $D_>(\lambda)=\sum_{\Lambda_i>\lambda}x_i(\lambda)^p$ and $D_{\geq}(\lambda)=\sum_{\Lambda_i\geq\lambda}x_i(\lambda)^p$
\State $R\gets B$
\State $V\gets0$
\State $\gamma_i^*\gets0$, $r_i^*\gets0$, $q_i^*\gets0$, $x_i^*\gets x_i(\lambda^*)$ for $i=1,\ldots,N$
\For{$j=1,\ldots,N$}
    \State $i\gets\sigma(j)$
    \If{$R\leq0$}
        \State \textbf{break}
    \EndIf
    \If{$x_i^*=0$}
        \State $\gamma_i^*\gets 1$
    \Else
        \State $\gamma_i^*\gets\min\{1,(R)_+/(x_i^*)^p\}$
    \EndIf
    \State $r_i^*\gets\gamma_i^*x_i^*$
    \State $q_i^*\gets\eta r_i^*$
    \State $V\gets V+\gamma_i^*(a-d_i+x_i^*)$
    \State $R\gets R-\gamma_i^*(x_i^*)^p$
\EndFor
\State \Return $\left(c_0+sV/N,\gamma^*,q^*\right)$
\end{algorithmic}
\end{algorithm}

We now return to the interval $\beta\in[0,\theta]$ and check that this construction has the
monotonicity used in the concavity argument. As $\beta$ increases toward $\theta$, each
$\gamma_i^*(\beta)$ can be chosen nondecreasing. Sort the samples increasingly and keep
this order for all $\beta$.  If $\hat\xi_i\leq\theta$, then $d_i(\beta)\leq a(\beta)$, so
$\Lambda_i(a(\beta))=\infty$ and the construction gives $\gamma_i^*(\beta)=1$.  If $\hat\xi_i>\theta$, write
$c_i=\hat\xi_i-\theta>0$ and set $a=\theta-\beta$. Then $d_i=a+c_i$. The breakpoint
$a=d_i/p$ from the general formula becomes $a=(a+c_i)/p$, or equivalently
$a=c_i/(p-1)$. Since $a<d_i$ automatically in this case, substituting $d_i=a+c_i$ and
$d_i-a=c_i$ gives
\[
    \Lambda_i(a)=
    \begin{cases}
        \dfrac{a}{(a+c_i)^p}, & 0\leq a\leq \dfrac{c_i}{p-1},\\[1.1em]
        \dfrac{(p-1)^{p-1}}{p^p c_i^{p-1}}, & a>\dfrac{c_i}{p-1}.
    \end{cases}
\]
Thus the finite thresholds are ordered by the sample order.  We also need that, for every
finite-threshold index $i$ and every $k\leq i$, the quantity $x_k(\beta,\Lambda_i(\beta))$ is
nonincreasing in $\beta$. As $\beta$ increases, $a=\theta-\beta$ decreases, and so does
$d_k(\beta)=a+c_k$. If $a\leq c_i/(p-1)$, then $t(\Lambda_i)\geq d_i$. Since the samples are sorted,
$d_k\leq d_i$ for $k\leq i$, and therefore $x_k(\beta,\Lambda_i(\beta))=d_k(\beta)$, which is
nonincreasing in $\beta$. If $a>c_i/(p-1)$, then $t(\Lambda_i)=p c_i/(p-1)$ is constant in $\beta$.
In this case $x_k(\beta,\Lambda_i(\beta))=\min\{d_k(\beta),p c_i/(p-1)\}$,
which is again nonincreasing in $\beta$.

For such an index define
$H_i(\beta)=\sum_{k<i}x_k(\beta,\Lambda_i(\beta))^p$ and
$Q_i(\beta)=x_i(\beta,\Lambda_i(\beta))^p$.  The previous paragraph shows that both are nonincreasing in
$\beta$, and the deterministic fill rule is equivalently
$\gamma_i^*(\beta)=\min\{1,(B-H_i(\beta))_+/Q_i(\beta)\}$.  Therefore, as $\beta$ increases, the
numerator is nondecreasing and the denominator is nonincreasing.  If the numerator is initially nonpositive, the
value is initially zero. Once it is positive, the raw ratio is nondecreasing.  Hence
$\gamma_i^*(\beta)$ is nondecreasing.  The full-budget case is also monotone because $\sum_i d_i(\beta)^p$
decreases with $\beta$.  Hence
$\gamma_i^*(\beta_1)\leq\gamma_i^*(\beta_2)$ for all $\beta_1\leq\beta_2\leq\theta$, completing the construction
needed for concavity on $[0,\theta]$.

The optimizers returned by Algorithms~\ref{alg:newsvendor_p1} and~\ref{alg:newsvendor_pgt1} also identify a
worst-case distribution attaining $h(\beta;\theta)$. Indeed, if $(\gamma^*,q^*)$ solves the corresponding
finite reformulation, then \cite{kuhn_wasserstein_2024} shows that one worst-case distribution is
\begin{align}
\label{eq:wc_dist_newsvendor}
    \mathbb{Q}^*
    =
    \sum_{i:\gamma_i^*>0}\frac{\gamma_i^*}{N}
    \delta_{\hat\xi_i+q_i^*/\gamma_i^*}
    +
    \sum_{i:\gamma_i^*<1}\frac{1-\gamma_i^*}{N}
    \delta_{\hat\xi_i}.
\end{align}
Thus each observation is either moved, left fixed, or split into one moved and one unmoved part. By
\cite[Corollary 2]{gao_distributionally_2022}, at most one observation needs to be split in an optimal
transport plan. This is consistent with the deterministic fill rules above, which fill the coordinates one at a
time and leave at most one fractional $\gamma_i^*$.

\subsection{Computing the DRRO Policy}
\label{app:newsvendor_policy_algorithms}

The preceding subsection provides an evaluation oracle for $h(\beta;\theta)$ and an optimizer of the
corresponding finite transport problem. We use this oracle in a nested one-dimensional scheme to compute
the DRRO policy. For each fixed $\theta$, Theorem~\ref{thm:piecewise_concavity} gives the decomposition
\eqref{eq:newsvendor_regret_split}. Hence the evaluation of $R(\theta)$ reduces to two scalar concave
maximization problems, one on $[0,\theta]$ and one on $[\theta,\infty)$. The oracle is Algorithm~\ref{alg:newsvendor_p1}
when $p=1$ and Algorithm~\ref{alg:newsvendor_pgt1} when $p>1$.

In our implementation, the two inner maximization problems are solved by golden-section search. On the
right interval $[\theta,\infty)$, we first bracket a maximizer and then apply the same scalar search on the
resulting finite interval. Equivalently, this minimizes the scalar convex function $-h(\cdot;\theta)$ on
each of the two pieces. A bisection method based on a supergradient of $h$ with respect to $\beta$ would
also be natural, but a simple closed-form supergradient is not immediate at nonunique optimizers.

The outer minimization uses that $R(\theta)$ is convex in $\theta$. Once the active inner maximizer has
been found, the optimizer returned by the evaluation oracle gives a subgradient through the Danskin formula
\eqref{eq:subgradient}. The sign of this subgradient determines which half of the current interval can
contain a minimizer.

\begin{algorithm}[!ht]
\caption{Computation of the univariate DRRO newsvendor policy}
\label{alg:newsvendor_drro_policy}
\begin{algorithmic}[1]
\Require Sorted samples $\hat\xi_1\leq\cdots\leq\hat\xi_N$, prices $b,s$, radius $\delta$, order $p$, interval $[\underline\theta,\overline\theta]$, tolerance $\varepsilon$
\Ensure Approximate DRRO order quantity $\theta^*$
\While{$\overline\theta-\underline\theta>\varepsilon$}
    \State $\theta\gets(\underline\theta+\overline\theta)/2$
    \State Let $\operatorname{Eval}_\theta(\beta)$ return $(h(\beta;\theta),\gamma^*(\beta),q^*(\beta))$ by Algorithm~\ref{alg:newsvendor_p1} if $p=1$ and by Algorithm~\ref{alg:newsvendor_pgt1} if $p>1$
    \State Let $\mathcal{H}_\theta(\beta)$ denote the value component of $\operatorname{Eval}_\theta(\beta)$
    \State Compute $\beta_L\in\argmax_{\beta\in[0,\theta]}\mathcal{H}_\theta(\beta)$ and $\beta_R\in\argmax_{\beta\in[\theta,\infty)}\mathcal{H}_\theta(\beta)$ by scalar concave maximization
    \State $(v_L,\gamma^L,q^L)\gets\operatorname{Eval}_\theta(\beta_L)$
    \State $(v_R,\gamma^R,q^R)\gets\operatorname{Eval}_\theta(\beta_R)$
    \If{$v_L\geq v_R$}
        \State $(R(\theta),\beta^*,\gamma^*,q^*)\gets(v_L,\beta_L,\gamma^L,q^L)$
    \Else
        \State $(R(\theta),\beta^*,\gamma^*,q^*)\gets(v_R,\beta_R,\gamma^R,q^R)$
    \EndIf
    \State Compute $g\in\partial R(\theta)$ from \eqref{eq:subgradient} using $(\beta^*,\gamma^*,q^*)$
    \If{$|g|\leq\varepsilon$}
        \State \Return $\theta$
    \ElsIf{$g>0$}
        \State $\overline\theta\gets\theta$
    \Else
        \State $\underline\theta\gets\theta$
    \EndIf
\EndWhile
\State \Return $(\underline\theta+\overline\theta)/2$
\end{algorithmic}
\end{algorithm}

\section{Proofs for Section~\ref{sec:sensitivity}}
\label{app:sensitivity}

We collect the sensitivity proofs.

\subsection{A first-order fixed-pair expansion}

\begin{proposition}[First-order fixed-pair expansion]
\label{prop:first_order_fixed_pair}
Assume the standing support and moment assumptions and Assumption~\ref{ass:fo_regularity}. Let
$\theta_\delta\to\bar\theta$ and $\beta_\delta\to\bar\beta$. Then
\[
    F_\delta(\theta_\delta,\beta_\delta)
    =
    L(\theta_\delta)-L(\beta_\delta)+\delta\Gamma(\bar\theta,\bar\beta)+o(\delta).
\]
The $o(\delta)$ remainder is locally uniform in the decision--benchmark pair: it is uniform over
convergent pairs $(\theta_\delta,\beta_\delta)\to(\bar\theta,\bar\beta)$ that remain in a common bounded
subset of $\Theta\times\Theta$.
\end{proposition}

\begin{proof}
Set $g(\theta, \beta, x)=\ell(\theta,x)-\ell(\beta,x)$. Write
\[
    F_\delta(\theta_\delta,\beta_\delta)
    =L(\theta_\delta)-L(\beta_\delta)+\Delta_\delta,\qquad
    \Delta_\delta=\sup_{\mathbb P\in B_\delta^p(\mathbb{P}_0)}
    \{\mathbb E_{\mathbb P}[g(\theta_\delta, \beta_\delta, X)]-\mathbb E_{\mathbb{P}_0}[g(\theta_\delta, \beta_\delta, X)]\}.
\]
It remains to show $\Delta_\delta=\delta\Gamma(\bar\theta,\bar\beta)+o(\delta)$. Let $C_\delta(\mathbb{P}_0)$ be the set of couplings
$\pi$ of $(X,Y)$ with $X\sim\mathbb{P}_0$, $Y\in\Xi$, and $\mathbb E_\pi[\|Y-X\|^p]\le\delta^p$. Since the infimum
over couplings in the definition of the Wasserstein distance is attained, the ambiguity constraint can be written
equivalently in terms of such couplings, and hence
\[
    \Delta_\delta=\sup_{\pi\in C_\delta(\mathbb{P}_0)}\mathbb E_\pi[g(\theta_\delta, \beta_\delta, Y)-g(\theta_\delta, \beta_\delta, X)].
\]
By convexity of $\Xi$ and the fundamental theorem of calculus,
\[
    g(\theta_\delta, \beta_\delta, Y)-g(\theta_\delta, \beta_\delta, X)=\int_0^1
    \langle \nabla_x g(\theta_\delta, \beta_\delta, X+t(Y-X)),Y-X\rangle\,dt.
\]
H\"older's inequality and the transport budget $\mathbb E_\pi[\|Y-X\|^p]\le\delta^p$ give
\[
    \mathbb E_\pi[g(\theta_\delta, \beta_\delta, Y)-g(\theta_\delta, \beta_\delta, X)]
    \le
    \delta\int_0^1
    \Big(\mathbb E_\pi[\|\nabla_xg(\theta_\delta, \beta_\delta, X+t(Y-X))\|_*^q]\Big)^{1/q}dt.
\]
It remains to take the limit on the right-hand side as $\delta\downarrow0$. Equip product spaces with the sum norm. Let
$\delta_n\downarrow0$, $\pi_n\in C_{\delta_n}(\mathbb{P}_0)$, and $(X_n,Y_n)\sim\pi_n$. For $D(x)=(x,x)$,
the coupling $((X_n,Y_n),(X_n,X_n))$ gives
\[
    W_p(\pi_n,D_\#\mathbb{P}_0)^p
    \le \mathbb E_{\pi_n}[\|Y-X\|^p]\le\delta_n^p,
\]
so $\pi_n\to D_\#\mathbb{P}_0$ in $W_p$. Since $\theta_{\delta_n}\to\bar\theta$ and
$\beta_{\delta_n}\to\bar\beta$, we also have
$\delta_{(\theta_{\delta_n},\beta_{\delta_n})}\otimes\pi_n\to
\delta_{(\bar\theta,\bar\beta)}\otimes D_\#\mathbb{P}_0$ in $W_p$. For fixed $t\in[0,1]$, the function
\[
    (\theta,\beta,x,y)\mapsto \|\nabla_x g(\theta,\beta,x+t(y-x))\|_*^q
\]
is continuous and, by Assumption~\ref{ass:fo_regularity} and $q=p/(p-1)$, is bounded by
$C(1+\|x\|^p+\|y\|^p)$ uniformly in $t$ and in bounded parameter sets. Hence the $W_p$ characterization for
continuous functions with $p$-growth \citep[Theorem~6.9]{villani_optimal_2009} gives
\[
    \mathbb E_{\pi_n}[\|\nabla_xg(\theta_{\delta_n},\beta_{\delta_n},X+t(Y-X))\|_*^q]
    \to
    \mathbb E_{\mathbb{P}_0}[\|\nabla_xg(\bar\theta,\bar\beta,X)\|_*^q]
    =
    \Gamma(\bar\theta,\bar\beta)^q .
\]
The same growth bound and the transport budget give a uniform bound on the left-hand side over $n$ and
$t\in[0,1]$, so dominated convergence in $t$ gives
\[
    \int_0^1
    \Big(\mathbb E_{\pi_n}[\|\nabla_xg(\theta_{\delta_n},\beta_{\delta_n},X+t(Y-X))\|_*^q]\Big)^{1/q}dt
    \to \Gamma(\bar\theta,\bar\beta).
\]
Since the sequence $\pi_n\in C_{\delta_n}(\mathbb{P}_0)$ was arbitrary, the same conclusion holds
for near maximizers of the supremum over $C_\delta(\mathbb{P}_0)$, and therefore
\[
\limsup_{\delta\downarrow0}\frac{\Delta_\delta}{\delta} \leq \Gamma(\bar\theta, \bar\beta).
\]

For the lower bound, assume first that $\Gamma(\bar\theta,\bar\beta)>0$ and set
$G_{\bar\theta,\bar\beta}=\nabla_x\ell(\bar\theta,\cdot)-\nabla_x\ell(\bar\beta,\cdot)$. Choose a measurable duality selector $h$ with
$\|h(z)\|=1$, $\langle z,h(z)\rangle=\|z\|_*$, and $h(0)=0$, and define
\[
    T_{\bar\theta,\bar\beta}(x)=h(G_{\bar\theta,\bar\beta}(x))
    \frac{\|G_{\bar\theta,\bar\beta}(x)\|_*^{q-1}}{\Gamma(\bar\theta,\bar\beta)^{q-1}}.
\]
Then $\mathbb E_{\mathbb{P}_0}[\|T_{\bar\theta,\bar\beta}(X)\|^p]=1$ and
$\mathbb E_{\mathbb{P}_0}[\langle G_{\bar\theta,\bar\beta}(X),T_{\bar\theta,\bar\beta}(X)\rangle]=\Gamma(\bar\theta,\bar\beta)$. If
$\Xi=\mathbb R^n$, let $T_{\bar\theta,\bar\beta,\delta}=T_{\bar\theta,\bar\beta}$; otherwise let
\[
    T_{\bar\theta,\bar\beta,\delta}(x)
    =
    T_{\bar\theta,\bar\beta}(x)
    \mathbf 1_{\{\operatorname{dist}(x,\partial\Xi)>\delta\|T_{\bar\theta,\bar\beta}(x)\|\}}.
\]
Then $x+\delta T_{\bar\theta,\bar\beta,\delta}(x)\in\Xi$ and, since $\mathbb{P}_0(\partial\Xi)=0$,
$T_{\bar\theta,\bar\beta,\delta}\to T_{\bar\theta,\bar\beta}$ pointwise $\mathbb{P}_0$-a.s. Moreover,
$\|T_{\bar\theta,\bar\beta,\delta}-T_{\bar\theta,\bar\beta}\|\le\|T_{\bar\theta,\bar\beta}\|$ and
$T_{\bar\theta,\bar\beta}\in L^p(\mathbb{P}_0)$, so dominated convergence gives
$T_{\bar\theta,\bar\beta,\delta}\to T_{\bar\theta,\bar\beta}$ in $L^p(\mathbb{P}_0)$. The coupling
$(x\mapsto(x,x+\delta T_{\bar\theta,\bar\beta,\delta}(x)))_\#\mathbb{P}_0$ is feasible, and therefore
\[
    \begin{aligned}
    \Delta_\delta
    &\ge
    \mathbb E_{\mathbb{P}_0}\!\left[
    g(\theta_\delta,\beta_\delta,X+\delta T_{\bar\theta,\bar\beta,\delta}(X))-g(\theta_\delta,\beta_\delta,X)
    \right] \\
    &=
    \delta\int_0^1\mathbb E_{\mathbb{P}_0}\!\left[
    \left\langle
    \nabla_xg(\theta_\delta,\beta_\delta,X+t\delta T_{\bar\theta,\bar\beta,\delta}(X)),
    T_{\bar\theta,\bar\beta,\delta}(X)
    \right\rangle\right]dt .
    \end{aligned}
\]
For each fixed $t\in[0,1]$, $\delta T_{\bar\theta,\bar\beta,\delta}(X)\to0$ pointwise $\mathbb{P}_0$-a.s. and
$(\theta_\delta,\beta_\delta)\to(\bar\theta,\bar\beta)$. Continuity in Assumption~\ref{ass:fo_regularity} therefore gives
pointwise convergence of the gradient to $G_{\bar\theta,\bar\beta}(X)$. Moreover, since
$\|T_{\bar\theta,\bar\beta,\delta}\|\le\|T_{\bar\theta,\bar\beta}\|$, the growth bound in
Assumption~\ref{ass:fo_regularity} and $\delta\le1$ bound
\[
    \big\|\nabla_xg(\theta_\delta,\beta_\delta,X+t\delta T_{\bar\theta,\bar\beta,\delta}(X))
    -G_{\bar\theta,\bar\beta}(X)\big\|_*^q
\]
by an integrable multiple of $1+\|X\|^p+\|T_{\bar\theta,\bar\beta}(X)\|^p$. Dominated convergence thus gives
\[
    \nabla_xg(\theta_\delta,\beta_\delta,X+t\delta T_{\bar\theta,\bar\beta,\delta}(X))
    \to G_{\bar\theta,\bar\beta}(X)
    \quad\text{in }L^q(\mathbb{P}_0),
\]
while we already showed that
$T_{\bar\theta,\bar\beta,\delta}(X)\to T_{\bar\theta,\bar\beta}(X)$ in $L^p(\mathbb{P}_0)$. Thus the inner
expectation converges pointwise in $t$ to
$\mathbb E_{\mathbb{P}_0}[\langle G_{\bar\theta,\bar\beta}(X),T_{\bar\theta,\bar\beta}(X)\rangle]
=\Gamma(\bar\theta,\bar\beta)$. Moreover, H\"older's inequality, the growth bound in
Assumption~\ref{ass:fo_regularity}, and $\|T_{\bar\theta,\bar\beta,\delta}\|_{L^p}\le1$ bound the inner
expectation uniformly in $t$ and $\delta$. Dominated convergence in $t$ therefore yields
$\liminf_{\delta\downarrow0}\Delta_\delta/\delta\ge\Gamma(\bar\theta,\bar\beta)$. If $\Gamma(\bar\theta,\bar\beta)=0$, the lower
bound is trivial because $\Delta_\delta\ge0$.
\end{proof}

We use the following terminology for near worst-case measures. For a given pair $(\theta_\delta,\beta_\delta)$, a
sequence $\mathbb P_\delta\in B_\delta^p(\mathbb{P}_0)$ is $o(\delta^r)$-adversarial if
\[
    \mathbb E_{\mathbb P_\delta}[\ell(\theta_\delta,X)-\ell(\beta_\delta,X)]
    \ge F_\delta(\theta_\delta,\beta_\delta)-o(\delta^r).
\]

\subsection{Proof of Theorem~\ref{thm:fo_wdrro_selection}}

\begin{proof}
For the upper bound, fix $\theta^*\in\Theta^*$. Let $\delta_n\downarrow0$ attain
$\limsup_{\delta\downarrow0}R_\delta(\theta^*)/\delta$, and take an $o(\delta_n)$-active
$\beta_{\delta_n}$ for $\theta^*$. Assumption~\ref{ass:fo_bounded_beta} gives boundedness, so the active benchmarks
lie in a common bounded set, which is what is needed for the locally uniform remainder in
Proposition~\ref{prop:first_order_fixed_pair}. If a cluster point $\bar\beta$ of $\beta_{\delta_n}$ were
not in $\Theta^*$, Proposition~\ref{prop:first_order_fixed_pair}, applied along the convergent subsequence with its
locally uniform remainder, would give
$F_{\delta_n}(\theta^*,\beta_{\delta_n})\to L^*-L(\bar\beta)<0$. This contradicts activity, because
$F_{\delta_n}(\theta^*,\theta^*)=0$ implies
$F_{\delta_n}(\theta^*,\beta_{\delta_n})\ge-o(\delta_n)$ along the same subsequence. Thus every cluster point lies in
$\Theta^*$. Passing to a subsequence, not relabeled, with $\beta_{\delta_n}\to\bar\beta\in\Theta^*$, activity,
Proposition~\ref{prop:first_order_fixed_pair}, and $L^* - L(\beta_{\delta_n})\leq 0$ give
\[
    R_{\delta_n}(\theta^*)
    =F_{\delta_n}(\theta^*,\beta_{\delta_n}) + o(\delta_n)=
    L^*-L(\beta_{\delta_n})
    +\delta_n\Gamma(\theta^*,\bar\beta)
    +o(\delta_n)
    \le
    \delta_n A(\theta^*)+o(\delta_n).
\]
Hence $\limsup_{\delta\downarrow0}R_\delta(\theta^*)/\delta\le A(\theta^*)$.
Taking $\theta^*$ with $A(\theta^*)\le\alpha+\varepsilon$ and then $\varepsilon\downarrow0$ gives
$\limsup_{\delta\downarrow0}V_\delta/\delta\le\alpha$.

For the matching lower bound, take $\delta_n\downarrow0$ with
$V_{\delta_n}/\delta_n\to\liminf_{\delta\downarrow0}V_\delta/\delta$, and let $\theta_{\delta_n}$ be
$o(\delta_n)$-optimal. The upper bound gives $V_{\delta_n}=O(\delta_n)$, hence
$R_{\delta_n}(\theta_{\delta_n})=O(\delta_n)$. Since
$R_{\delta_n}(\theta_{\delta_n})\ge R_0(\theta_{\delta_n})=L(\theta_{\delta_n})-L^*$, we have
$L(\theta_{\delta_n})-L^*=O(\delta_n)$. Assumption~\ref{ass:fo_bounded_theta} gives boundedness of
$\theta_{\delta_n}$, so, after passing to a subsequence, $\theta_{\delta_n}\to\bar\theta$. If
$\bar\theta\notin\Theta^*$, then continuity of $L$ gives $L(\bar\theta)-L^*>0$, contradicting
$L(\theta_{\delta_n})-L^*\to0$. Thus $\bar\theta\in\Theta^*$. For every $\beta^*\in\Theta^*$,
\[
    R_{\delta_n}(\theta_{\delta_n})\ge F_{\delta_n}(\theta_{\delta_n},\beta^*)
    =
    L(\theta_{\delta_n})-L^*+\delta_n\Gamma(\bar\theta,\beta^*)+o(\delta_n).
\]
Since $L(\theta_{\delta_n})\ge L^*$, taking the liminf for each fixed $\beta^*$ and then the supremum over
$\beta^*$ gives
\[
    \liminf_{n\to\infty}\frac{R_{\delta_n}(\theta_{\delta_n})}{\delta_n}\ge A(\bar\theta).
\]
By near-optimality and the choice of $\delta_n$,
\[
    \liminf_{\delta\downarrow0}\frac{V_\delta}{\delta}
    =
    \lim_{n\to\infty}\frac{R_{\delta_n}(\theta_{\delta_n})}{\delta_n}
    \ge A(\bar\theta)\ge\alpha.
\]
Together with the upper bound, this proves $V_\delta=\alpha\delta+o(\delta)$.

The preceding argument used only one sequence chosen to attain the value liminf. We next return to an arbitrary
$o(\delta)$-optimal sequence to prove the optimizer-selection claims. Let
$\theta_\delta$ be any $o(\delta)$-optimal sequence. The value expansion implies
$R_\delta(\theta_\delta)=O(\delta)$, so the preceding localization argument shows that every cluster point lies in
$\Theta^*$. Along any subsequence $\delta_n\downarrow0$ with $\theta_{\delta_n}\to\bar\theta$, the same fixed-$\beta^*$
argument gives
\[
    A(\bar\theta)\le \liminf_{n\to\infty}\frac{R_{\delta_n}(\theta_{\delta_n})}{\delta_n}
    =
    \alpha.
\]
Since $\alpha\le A(\bar\theta)$ by definition, $A(\bar\theta)=\alpha$, and every cluster point $\bar\theta$ lies in
$\Theta_1^*$.
We first record a nominal-risk refinement for the decision sequence, which will be used for the active-benchmark
claims and the sharper rates under quadratic growth. Choose $\beta_\varepsilon^*\in\Theta^*$ with
$\Gamma(\bar\theta,\beta_\varepsilon^*)\ge A(\bar\theta)-\varepsilon=\alpha-\varepsilon$. Since
$R_{\delta_n}(\theta_{\delta_n})=V_{\delta_n}+o(\delta_n)=\alpha\delta_n+o(\delta_n)$, while
\[
    R_{\delta_n}(\theta_{\delta_n})
    \ge F_{\delta_n}(\theta_{\delta_n},\beta_\varepsilon^*)
    =
    L(\theta_{\delta_n})-L^*
    +\delta_n\Gamma(\bar\theta,\beta_\varepsilon^*)
    +o(\delta_n),
\]
we get
\[
    L(\theta_{\delta_n})-L^*
    \le
    \alpha\delta_n-\delta_n(\alpha-\varepsilon)+o(\delta_n)
    =
    \varepsilon\delta_n+o(\delta_n).
\]
Together with $L(\theta_{\delta_n})\ge L^*$, this gives
$0\le L(\theta_{\delta_n})-L^*\le\varepsilon\delta_n+o(\delta_n)$, and then $\varepsilon\downarrow0$ yields
$L(\theta_{\delta_n})-L^*=o(\delta_n)$. Since every subsequence has a further convergent subsequence, this implies
$L(\theta_\delta)-L^*=o(\delta)$.

It remains to identify which ERM benchmarks can remain active. Let $\beta_\delta$ be $o(\delta)$-active for
$\theta_\delta$. Assumption~\ref{ass:fo_bounded_beta} gives
boundedness. Along a subsequence $\delta_n\downarrow0$, suppose
$\theta_{\delta_n}\to\bar\theta$ and $\beta_{\delta_n}\to\bar\beta$. The preceding paragraph gives
$\bar\theta\in\Theta_1^*$ and $L(\theta_{\delta_n})-L^*=o(\delta_n)$. The same localization argument as above gives
$\bar\beta\in\Theta^*$: otherwise Proposition~\ref{prop:first_order_fixed_pair} would give
$F_{\delta_n}(\theta_{\delta_n},\beta_{\delta_n})\to L^*-L(\bar\beta)<0$, contradicting activity and
$R_{\delta_n}(\theta_{\delta_n})\ge0$. Activity, near-optimality, and
$V_\delta=\alpha\delta+o(\delta)$ give
$F_{\delta_n}(\theta_{\delta_n},\beta_{\delta_n})=\alpha\delta_n+o(\delta_n)$, while
\[
    F_{\delta_n}(\theta_{\delta_n},\beta_{\delta_n})
    =
    L(\theta_{\delta_n})-L(\beta_{\delta_n})+\delta_n\Gamma(\bar\theta,\bar\beta)+o(\delta_n).
\]
Hence
\[
    -\bigl(L(\beta_{\delta_n})-L^*\bigr)+\delta_n\Gamma(\bar\theta,\bar\beta)
    =\alpha\delta_n+o(\delta_n).
\]
Since $L(\beta_{\delta_n})\ge L^*$ and $\Gamma(\bar\theta,\bar\beta)\le A(\bar\theta)=\alpha$, we get
$\Gamma(\bar\theta,\bar\beta)=\alpha$ and $L(\beta_{\delta_n})-L^*=o(\delta_n)$, so
$\bar\beta\in\mathcal B_1(\bar\theta)$. Since every subsequence of active benchmarks has a further convergent
subsequence of this form, $L(\beta_\delta)-L^*=o(\delta)$.

If $L$ has quadratic growth around $\Theta^*$, then $L(\theta_\delta)-L^*=o(\delta)$ and
$L(\beta_\delta)-L^*=o(\delta)$ imply
$\operatorname{dist}(\theta_\delta,\Theta^*)=o(\sqrt{\delta})$ and
$\operatorname{dist}(\beta_\delta,\Theta^*)=o(\sqrt{\delta})$.
\end{proof}

\subsection{Proof of Lemma~\ref{lemma:bounded_near_optimizers}}

\begin{proof}
If $\Theta$ is bounded, there is nothing to prove. In the coercive case, set $D(\theta)=L(\theta)-L^*$. The
transport-slope domination condition and the coupling/H\"older argument used above give
\[
    F_\delta(\theta,\beta)
    \le
    L(\theta)-L(\beta)+C\delta\bigl(1+D(\theta)^\eta+D(\beta)^\eta\bigr).
\]
Fix $\theta^*\in\Theta^*$. The display above implies
\[
    R_\delta(\theta^*)\le
    \sup_{\beta\in\Theta}\{-D(\beta)+C\delta(1+D(\beta)^\eta)\}=O(\delta),
\]
because $D\ge0$, $\eta\le1$, and for small $\delta$ the expression is at most $2C\delta$ when $D(\beta)\le1$ and at
most $C\delta-(1-C\delta)D(\beta)\le C\delta$ when $D(\beta)\ge1$. Now let $\theta_\delta$ be
$o(\delta)$-optimal. Near-optimality gives
$R_\delta(\theta_\delta)\le R_\delta(\theta^*)+o(\delta)=O(\delta)$. Since
$R_\delta(\theta_\delta)\ge D(\theta_\delta)$, $D(\theta_\delta)$ is bounded, and coercivity makes
$\theta_\delta$ bounded.

Now let $\theta_\delta$ be bounded with $L(\theta_\delta)\to L^*$, and let $\beta_\delta$ be $o(\delta)$-active.
By activity,
$F_\delta(\theta_\delta,\beta_\delta)\ge R_\delta(\theta_\delta)-o(\delta)$, and
$R_\delta(\theta_\delta)\ge F_\delta(\theta_\delta,\theta_\delta)=0$. Hence
$F_\delta(\theta_\delta,\beta_\delta)\ge-o(\delta)$. The preceding bound gives
\[
    F_\delta(\theta_\delta,\beta_\delta)
    \le
    D(\theta_\delta)-D(\beta_\delta)
    +C\delta\bigl(1+D(\theta_\delta)^\eta+D(\beta_\delta)^\eta\bigr).
\]
Since $D(\theta_\delta)\to0$ and $\eta\le1$, the right-hand side tends to $-\infty$ if
$D(\beta_\delta)\to\infty$, contradicting activity. Thus $D(\beta_\delta)$ is bounded, and coercivity makes
$\beta_\delta$ bounded.
\end{proof}

\subsection{Proof of Lemma~\ref{lemma:attainment}}

\begin{proof}
Since $\Xi$ is compact, $\mathcal P(\Xi)$ is compact in the weak topology and the closed Wasserstein ball
$B_\delta^p(\mathbb{P}_0)$ is compact. The map
$(\theta,\beta,\mathbb P)\mapsto\mathbb E_{\mathbb P}[\ell(\theta,X)-\ell(\beta,X)]$ is continuous on the compact
set $\Theta\times\Theta\times B_\delta^p(\mathbb{P}_0)$. The maximum theorem gives attainment and continuity of
$F_\delta$, attainment of the supremum over $\beta$, continuity of $R_\delta$, and finally attainment of $V_\delta$ on
compact $\Theta$.
\end{proof}

\subsection{Proof of the discontinuity claim in Example~\ref{ex:rank_deficient_regression}}

\begin{proof}
Take $\Theta=\mathbb R^m$ and $\hat\sigma^2>0$. Let $v\in\ker X$ be nonzero with $\|v\|=1$. Since
$\hat\theta^\dagger\perp\ker X$, define $\beta_t=\hat\theta^\dagger+tv$ and
$\eta_i^t=r_i/(t+v^\top\hat\theta^\dagger)$. With
$\tilde x_i^t=x_i+\eta_i^t v$ and $\tilde y_i^t=y_i$, we have $x_i^\top v=0$ and hence
$(\tilde x_i^t)^\top\beta_t=y_i$. Thus $\beta_t$ interpolates the perturbed sample exactly, while
$N^{-1}\sum_i\|\tilde x_i^t-x_i\|^p=N^{-1}\sum_i|\eta_i^t|^p\to0$ as $t\to\infty$. Therefore, for every fixed $\delta>0$, these
perturbed empirical laws are feasible for all large $t$.

Fix $\delta>0$ and $\theta\in\Theta$. For all sufficiently large $t$, the perturbed empirical law is feasible, and
$\beta_t$ has zero loss under it. Hence
\[
    R_\delta(\theta)\ge F_\delta(\theta,\beta_t)
    \ge \frac1N\sum_{i=1}^N(y_i-(\tilde x_i^t)^\top\theta)^2.
\]
Letting $t\to\infty$ gives
\[
    R_\delta(\theta)\ge \frac1N\sum_{i=1}^N(y_i-x_i^\top\theta)^2
\]
for every $\theta$. Taking the infimum over $\theta$ yields $V_\delta\ge\hat\sigma^2$.
Conversely, evaluating at $\hat\theta^\dagger$ gives
\[
    V_\delta\le R_\delta(\hat\theta^\dagger)
    \le \sup_{\mathbb P\in B_\delta^p(\mathbb{P}_0)}
    \mathbb E_{\mathbb P}[(Y-X^\top\hat\theta^\dagger)^2].
\]
For $p\ge2$, Wasserstein continuity of the squared loss yields
$\limsup_{\delta\downarrow0}V_\delta\le\hat\sigma^2$. Thus
$\lim_{\delta\downarrow0}V_\delta=\hat\sigma^2$. Since $V_0=0$, the value is discontinuous at zero.
\end{proof}

\subsection{A second-order fixed-pair expansion around a unique ERM}

\begin{proposition}[Second-order fixed-pair expansion]
\label{prop:second_order_fixed_pair}
Assume the standing support and moment assumptions,
Assumptions~\ref{ass:fo_regularity},~\ref{ass:unique_second_order}, and~\ref{ass:mixed_smoothness}.
For bounded $a,b\in\mathbb R^d$,
\[
    F_\delta(\theta_0+\delta a,\theta_0+\delta b)
    =
    \delta^2\left[
        \frac12a^\top Ha-\frac12b^\top Hb+\gamma(a-b)
    \right]+o(\delta^2),
\]
locally uniformly in $(a,b)$.
\end{proposition}

\begin{proof}
Let $g_{\delta,a,b}(x)=\ell(\theta_0+\delta a,x)-\ell(\theta_0+\delta b,x)$ and write
$F_\delta(\theta_0+\delta a,\theta_0+\delta b)=\mathbb E_{\mathbb{P}_0}[g_{\delta,a,b}(X)]+\Delta_\delta(a,b)$. The nominal
term in this split is
\[
    \mathbb E_{\mathbb{P}_0}[g_{\delta,a,b}(X)]
    =
    L(\theta_0+\delta a)-L(\theta_0+\delta b)
    =
    \frac{\delta^2}{2}a^\top Ha-\frac{\delta^2}{2}b^\top Hb+o(\delta^2),
\]
 locally uniformly in $(a,b)$ by Assumption~\ref{ass:unique_second_order}. It remains to show
$\Delta_\delta(a,b)=\delta^2\gamma(a-b)+o(\delta^2)$.

As before,
$\Delta_\delta(a,b)=\sup_{\pi\in C_\delta(\mathbb{P}_0)}\mathbb E_\pi[g_{\delta,a,b}(Y)-g_{\delta,a,b}(X)]$. Define
$G_{\delta,a,b}=\delta^{-1}\nabla_xg_{\delta,a,b}$. The same calculus and H\"older estimate as in the first-order
proof gives
\[
    \Delta_\delta(a,b)
    \le
    \delta^2\sup_{\pi\in C_\delta(\mathbb{P}_0)}\int_0^1
    \Big(\mathbb E_\pi[\|G_{\delta,a,b}(X+t(Y-X))\|_*^q]\Big)^{1/q}dt.
\]
By Assumption~\ref{ass:mixed_smoothness}, $G_{\delta,a,b}(z)\to B(z)(a-b)$ with the required uniform growth on bounded
sets of $(a,b)$. The same argument as in Proposition~\ref{prop:first_order_fixed_pair} therefore gives
$\limsup_{\delta\downarrow0}\Delta_\delta(a,b)/\delta^2\le\gamma(a-b)$, locally uniformly in $(a,b)$.

For the lower bound, let $u=a-b$. If $\gamma(u)=0$, use $\Delta_\delta(a,b)\ge0$. Otherwise, with the
measurable duality selector $h$ from the proof of Proposition~\ref{prop:first_order_fixed_pair}, set
\[
    T_u(x)=h(B(x)u)\frac{\|B(x)u\|_*^{q-1}}{\gamma(u)^{q-1}}.
\]
Then $\mathbb E_{\mathbb{P}_0}[\|T_u(X)\|^p]=1$ and
$\mathbb E_{\mathbb{P}_0}[\langle B(X)u,T_u(X)\rangle]=\gamma(u)$. Use $T_{u,\delta}=T_u$ if $\Xi=\mathbb R^n$ and otherwise
truncate it by the boundary indicator from the lower-bound proof of Proposition~\ref{prop:first_order_fixed_pair}. The
same truncation argument gives $T_{u,\delta}\to T_u$ in $L^p(\mathbb{P}_0)$ and
$x+\delta T_{u,\delta}(x)\in\Xi$. Hence the coupling induced by
$Y=X+\delta T_{u,\delta}(X)$ is feasible. Applying the fundamental theorem of calculus to
$g_{\delta,a,b}$ and using $\nabla_xg_{\delta,a,b}=\delta G_{\delta,a,b}$ gives
\[
    \frac{\Delta_\delta(a,b)}{\delta^2}
    \ge
    \int_0^1
    \mathbb E_{\mathbb{P}_0}[
    \langle G_{\delta,a,b}(X+t\delta T_{u,\delta}(X)),T_{u,\delta}(X)\rangle]dt,
\]
Applying the same dominated-convergence and H\"older argument as in the lower-bound proof of
Proposition~\ref{prop:first_order_fixed_pair}, with Assumption~\ref{ass:mixed_smoothness} in place of
Assumption~\ref{ass:fo_regularity}, gives
$G_{\delta,a,b}(X+t\delta T_{u,\delta}(X))\to B(X)u$ in $L^q(\mathbb{P}_0)$ for each fixed $t$, while H\"older's
inequality gives an integrable bound that is uniform in $t\in[0,1]$. Dominated convergence in $t$ therefore gives
convergence of the right-hand side in the last equation to
$\mathbb E_{\mathbb{P}_0}[\langle B(X)u,T_u(X)\rangle]=\gamma(u)$.
\end{proof}

\subsection{Proof of Theorem~\ref{thm:unique_second_order}}

\begin{proof}
We first show that active benchmarks localize. Let $a_\delta$ be bounded, set
$\theta_{\delta,a_\delta}=\theta_0+\delta a_\delta$, and let $\beta_\delta$ be $o(\delta^2)$-active for this decision.
By Assumption~\ref{ass:so_bounded_beta}, $\beta_\delta$ is bounded. If a subsequence had a cluster point
$\bar\beta\ne\theta_0$, then $L(\bar\beta)>L(\theta_0)$, and Proposition~\ref{prop:first_order_fixed_pair} would give
$F_\delta(\theta_{\delta,a_\delta},\beta_\delta)\to L(\theta_0)-L(\bar\beta)<0$ along that subsequence. This contradicts activity:
since $F_\delta(\theta_{\delta,a_\delta},\theta_{\delta,a_\delta})=0$, an active benchmark must attain
$F_\delta(\theta_{\delta,a_\delta},\beta_\delta)\ge-o(\delta^2)$ and therefore have a nonnegative limit. Thus
$\beta_\delta\to\theta_0$.

We now sharpen this convergence to $\beta_\delta=\theta_0+O(\delta)$. Let $s_\delta=\|\beta_\delta-\theta_0\|$.
Since $\theta_{\delta,a_\delta}\to\theta_0$ and $\beta_\delta\to\theta_0$, both parameters lie in the neighborhood
from Assumption~\ref{ass:mixed_smoothness} for small $\delta$. Thus, for
$g_\delta(x)=\ell(\theta_{\delta,a_\delta},x)-\ell(\beta_\delta,x)$,
\[
    \nabla_x g_\delta(x)
    =
    \int_0^1\nabla_x\nabla_\theta\ell(\beta_\delta+t(\theta_{\delta,a_\delta}-\beta_\delta),x)
    (\theta_{\delta,a_\delta}-\beta_\delta)\,dt,
\]
 and hence, by the growth bound in Assumption~\ref{ass:mixed_smoothness},
\[
    \|\nabla_x g_\delta(x)\|_*
    \le C(1+\|x\|^{p-1})\|\theta_{\delta,a_\delta}-\beta_\delta\|.
\]
The same coupling/H\"older estimate as in Proposition~\ref{prop:first_order_fixed_pair} then gives
\[
    F_\delta(\theta_{\delta,a_\delta},\beta_\delta)
    \le
    L(\theta_{\delta,a_\delta})-L(\beta_\delta)
    +C\delta\|\theta_{\delta,a_\delta}-\beta_\delta\|.
\]
Assumption~\ref{ass:unique_second_order} gives local quadratic upper and lower bounds around $\theta_0$ because
$H\succ0$; hence
$L(\theta_{\delta,a_\delta})-L(\theta_0)\le C\delta^2$ and
$L(\beta_\delta)-L(\theta_0)\ge cs_\delta^2$. Also,
\[
    \|\theta_{\delta,a_\delta}-\beta_\delta\|
    \le
    \|\theta_{\delta,a_\delta}-\theta_0\|+\|\beta_\delta-\theta_0\|
    \le C\delta+s_\delta .
\]
Combining these bounds with the previous display and increasing $C$ if necessary gives
\[
    F_\delta(\theta_{\delta,a_\delta},\beta_\delta)
    \le
    C\delta^2-cs_\delta^2+C\delta s_\delta .
\]
On the other hand, activity and
$F_\delta(\theta_{\delta,a_\delta},\theta_{\delta,a_\delta})=0$ imply
$F_\delta(\theta_{\delta,a_\delta},\beta_\delta)\ge-o(\delta^2)$. Hence
\[
    C\delta^2-cs_\delta^2+C\delta s_\delta\ge-o(\delta^2).
\]
This forces $s_\delta=O(\delta)$; otherwise, along a subsequence with $s_\delta/\delta\to\infty$, the negative
quadratic term $-cs_\delta^2$ would dominate $C\delta^2+C\delta s_\delta$ and contradict the last display.

For bounded $a_\delta$, active benchmarks can be written as
$\beta_\delta=\theta_0+\delta b_\delta$ with $b_\delta=O(1)$. Proposition~\ref{prop:second_order_fixed_pair} then gives
\[
    R_\delta(\theta_0+\delta a_\delta)
    =
    \delta^2\sup_b
    \left\{\frac12a_\delta^\top Ha_\delta-\frac12b^\top Hb+\gamma(a_\delta-b)\right\}
    +o(\delta^2) = \delta^2\rho(a_\delta)+o(\delta^2)
\]
where
\[
    \rho(a)=\sup_w\{w^\top Ha-\tfrac12w^\top Hw+\gamma(w)\}.
\]

Let $\theta_\delta$ be $o(\delta^2)$-optimal. The expansion above gives
$R_\delta(\theta_0)=\delta^2\rho(0)+o(\delta^2)$. Since
$0\leq R_\delta(\theta_\delta)\leq R_\delta(\theta_0)+o(\delta^2)$, it follows that
$R_\delta(\theta_\delta)=O(\delta^2)$. Since also
$R_\delta(\theta_\delta)\ge R_0(\theta_\delta)=L(\theta_\delta)-L(\theta_0)$, it must hold
$L(\theta_\delta)-L(\theta_0)=O(\delta^2)$. Since $\theta_\delta$ is bounded by Assumption~\ref{ass:so_bounded_theta}
and the ERM solution is unique, this implies
$\theta_\delta\to\theta_0$. Since Assumption~\ref{ass:unique_second_order} implies the local lower bound
$L(\theta)-L(\theta_0)\ge c\|\theta-\theta_0\|^2$, we get
$\|\theta_\delta-\theta_0\|=O(\delta)$, so $a_\delta=(\theta_\delta-\theta_0)/\delta$ is bounded.

Since $\theta_\delta=\theta_0+\delta a_\delta$ with $a_\delta=O(1)$, the expansion
$R_\delta(\theta_0+\delta a)=\delta^2\rho(a)+o(\delta^2)$ holds uniformly over a bounded set containing
$a_\delta$. Hence $R_\delta(\theta_\delta)/\delta^2=\rho(a_\delta)+o(1)$. Therefore
$V_\delta/\delta^2\to\inf_a\rho(a)$. Since $\gamma$ is even, for any $a,w$ both signs of $w$ imply
$\rho(a)\ge\gamma(w)-\frac12w^\top Hw$. Taking the supremum over $w$ gives $\rho(a)\ge\rho(0)$.
Moreover, because $\gamma$ is a finite seminorm, it is the support function of its polar set
$C_\gamma:=\{z:z^\top w\le\gamma(w)\ \text{for all }w\}$. Thus
$\gamma(w)=\sup_{z\in C_\gamma}z^\top w$, and
\[
    \begin{aligned}
    \rho(a)
    &=
    \sup_w\sup_{z\in C_\gamma}
    \{w^\top(Ha+z)-\tfrac12w^\top Hw\} \\
    &=
    \frac12\sup_{z\in C_\gamma}(Ha+z)^\top H^{-1}(Ha+z),
    \end{aligned}
\]
where the last equality maximizes the strictly concave quadratic in $w$. The set $C_\gamma$ is compact, convex, and
centrally symmetric. If $a\ne0$, choose a maximizer $z_0$ of $z^\top H^{-1}z$ with sign such that $a^\top z_0\ge0$. Then
$(Ha+z_0)^\top H^{-1}(Ha+z_0)>z_0^\top H^{-1}z_0$, so $\rho(a)>\rho(0)$. Thus $0$ is the unique minimizer,
$V_\delta=\kappa\delta^2+o(\delta^2)$, and $a_\delta\to0$, i.e. $\theta_\delta=\theta_0+o(\delta)$.

Let $\beta_\delta$ be $o(\delta^2)$-active for $\theta_\delta$ and set
$w_\delta=(\theta_\delta-\beta_\delta)/\delta$. Since $\theta_\delta=\theta_0+o(\delta)$ and the active-benchmark
localization above gives $\beta_\delta=\theta_0+O(\delta)$, the sequence $w_\delta$ is bounded. Along any subsequence
$w_\delta\to w$, write $\beta_\delta=\theta_0+\delta(a_\delta-w_\delta)$. Proposition~\ref{prop:second_order_fixed_pair}
gives
\[
    F_\delta(\theta_\delta,\beta_\delta)
    =
    \delta^2\left[
    \frac12a_\delta^\top Ha_\delta
    -\frac12(a_\delta-w_\delta)^\top H(a_\delta-w_\delta)
    +\gamma(w_\delta)
    \right]+o(\delta^2).
\]
Since $a_\delta\to0$ and $w_\delta\to w$, the bracket converges to
$\gamma(w)-\frac12w^\top Hw$. Activity and $V_\delta=\kappa\delta^2+o(\delta^2)$ also give
$F_\delta(\theta_\delta,\beta_\delta)=\kappa\delta^2+o(\delta^2)$. Hence
\[
    \kappa=\gamma(w)-\frac12w^\top Hw.
\]
Thus $w\in\mathcal W^*$ and $\beta_\delta=\theta_0-\delta w+o(\delta)$. Conversely, if
$w\in\mathcal W^*$ and $\beta_\delta^w=\theta_0-\delta w$, then the fixed-pair expansion gives
$F_\delta(\theta_0,\beta_\delta^w)=\kappa\delta^2+o(\delta^2)=R_\delta(\theta_0)+o(\delta^2)$, so this branch is
$o(\delta^2)$-active.

Writing $w=tu$ with $u\ne0$ and $t\ge0$ gives
$\gamma(tu)-\frac12t^2u^\top Hu=t\gamma(u)-\frac12t^2u^\top Hu$, maximized at
$t^*(u)=\gamma(u)/(u^\top Hu)$. Therefore
\[
    \kappa=\frac12\sup_{u\ne0}\frac{\gamma(u)^2}{u^\top Hu}.
\]
\end{proof}

\subsection{Proof of Theorem~\ref{thm:exact_quadratic}}

\begin{proof}
Let $\mathbb P\in B_\delta^p(\mathbb{P}_0)$ have mean $\mu$ and covariance matrix $\Sigma$. Then
\[
    \mathbb E_{\mathbb P}[\ell(\theta,X)]
    =
    \theta^\top Q\theta+2\theta^\top(S^\top \mu+u)
    +\mu^\top R\mu+\operatorname{Tr}(\Sigma R)+2\mu^\top r.
\]
For fixed $\mathbb P$, the minimizer is $\beta^*(\mu)=-Q^{-1}(S^\top \mu+u)$ and the terms involving $R,r,\Sigma$
cancel in the regret. Thus
\[
    R_\delta(\theta)
    =
    \sup_{\mathbb P\in B_\delta^p(\mathbb{P}_0)}
    \left\{\theta^\top Q\theta+2\theta^\top(S^\top \mu+u)
    +(S^\top \mu+u)^\top Q^{-1}(S^\top \mu+u)\right\}.
\]
The expression depends on $\mathbb P$ only through $\mu$. For any ground norm and $p\ge1$, the attainable means are exactly
$\{\mu:\|\mu-\hat\mu\|\le\delta\}$: one inclusion is Jensen's inequality, and the other follows by translating $\mathbb{P}_0$ by a
constant vector. Hence
\[
    R_\delta(\theta)=
    \max_{\|\mu-\hat\mu\|\le\delta}
    \left\{\theta^\top Q\theta+2\theta^\top(S^\top \mu+u)
    +(S^\top \mu+u)^\top Q^{-1}(S^\top \mu+u)\right\}.
\]
At $\theta_0=-Q^{-1}(S^\top\hat\mu+u)$ this becomes
\[
    R_\delta(\theta_0)
    =
    \max_{\|\mu-\hat\mu\|\le\delta}
    (\mu-\hat\mu)^\top SQ^{-1}S^\top(\mu-\hat\mu)
    =
    \delta^2\sup_{\|z\|\le1}z^\top SQ^{-1}S^\top z.
\]
This coefficient equals $\kappa=\sup_{w\ne0}\frac{\|Sw\|_*^2}{w^\top Qw}$ because
\[
    \begin{aligned}
    \sup_{\|z\|\le1}z^\top SQ^{-1}S^\top z
    &=
    \sup_{\|z\|\le1}\sup_{w}
    \{2z^\top Sw-w^\top Qw\} \\
    &=
    \sup_w
    \{2\|Sw\|_*-w^\top Qw\} \\
    &=
    \sup_{w\ne0}\sup_{t\ge0}
    \{2t\|Sw\|_*-t^2w^\top Qw\} \\
    &=
    \sup_{w\ne0}\frac{\|Sw\|_*^2}{w^\top Qw},
    \end{aligned}
\]
where the first equality uses completion of squares in the $Q$-norm and the second uses the dual
representation of $\|\cdot\|_*$.

It remains to prove global optimality. Let
$f(\theta,\mu)=\theta^\top Q\theta+2\theta^\top(S^\top \mu+u)+(S^\top \mu+u)^\top Q^{-1}(S^\top \mu+u)$. Then
$R_\delta(\theta)=\max_{\|\mu-\hat\mu\|\le\delta}f(\theta,\mu)$ is convex in $\theta$ as a pointwise maximum of
convex quadratics. Let $z_*$ maximize
$z^\top SQ^{-1}S^\top z$ over $\|z\|\le1$. At $\theta_0$, the points $\hat\mu+\delta z_*$ and
$\hat\mu-\delta z_*$ are maximizers in the problem defining $R_\delta(\theta_0)$, with $\theta$-gradients
$\pm2\delta S^\top z_*$. By the Levin--Valadier theorem \citep[Theorem~9.27]{shapiro_lectures_2009},
\[
    \partial R_\delta(\theta_0)
    =
    \operatorname{conv}\{\nabla_\theta f(\theta_0,\mu):
    \mu\in\operatorname*{argmax}_{\|\mu-\hat\mu\|\le\delta}f(\theta_0,\mu)\},
\]
so $0\in\partial R_\delta(\theta_0)$, which is sufficient for global optimality by convexity. The translated measures
$(x\mapsto x\pm\delta z_*)_\#\mathbb{P}_0$ are feasible, attain the optimal means, and therefore are worst-case
adversarial measures for $\theta_0$.
\end{proof}

\section{Proofs of Section~\ref{sec:problems_of_interest}}
\label{app2}

We prove Theorem~\ref{thm:np_hardness_regret} by reducing the NP-complete 1-in-3 SAT+ problem to the regret computation problem. 
The 1-in-3 SAT+ problem is a variant of the well-known NP-complete 3-SAT problem and is defined as follows:

\begin{definition}
    Let $W=\{w_1,\ldots,w_N\}$ be a set of disjunctive clauses of the form $w_i=\bigvee_{j\in\mathcal{S}_i} v_j$ of only positive literals 
    in the variables $V=\{v_1,\ldots,v_d\}$ where $\mathcal{S}_i\subset[d]$ are sets of size three. The 1-in-3 SAT+ problem asks whether 
    there is a truth assignment for $V$ such that in each of the $N$ clauses in $W$ exactly one variable is true.
\end{definition}

Just like the 3-SAT problem its 1-in-3 SAT+ variant is also NP-complete \cite[p. 259]{garey_computers_1990}. 
We occasionally need to access the individual elements in $\mathcal{S}_i$. For this purpose we write $\mathcal{S}_i=\{j^i_1,j^i_2,j^i_3\}$.

\subsection{Auxiliary Lemmas}

For the proof of Theorem~\ref{thm:np_hardness_regret}, we need the following two auxiliary lemmas.
\begin{lemma}\label{lemma:np_hardness_helper} For $\beta_1, \beta_2, \beta_3\in[0, 1]$ and $\lambda\geq 0$ it holds
\begin{align}
    &\max_{k\in[3]}\sup_{z\in \mathbb{R}^3}\min_{m\in[3]}\left\{z_k - z_m + \beta_m - \lambda \Vert z\Vert_1 \right\} \nonumber\\
    &\quad = \max\left\{\min(\beta_1, \beta_2, \beta_3), \lambda \min(\beta_1, \beta_2, \beta_3) + (1 - \lambda)\max(\beta_1, \beta_2, \beta_3)\right\}\label{eq:formula_helper_lemma}.
\end{align}
For $\lambda \geq 1$ a maximizer is given by $z^*=\vec{0}_3$ and for $\lambda\in[0, 1]$ there is a maximizer that satisfies $z^*\in[-1, 1]^3$. 
\end{lemma}

\begin{proof}
    We assume w.l.o.g.\ that $\beta_1\geq \beta_2\geq \beta_3$. First we consider the case $\lambda \geq 1$. Now note that the inner supremum over $z$ can be upper bounded as follows:
    \begin{align*}
        \sup_{z\in \mathbb{R}^3}\min_{m\in[3]}\left\{z_k - z_m + \beta_m - \lambda \Vert z\Vert_1 \right\}\leq \sup_{z\in\mathbb{R}^3} \beta_3 + z_k - z_3 - \lambda\Vert z\Vert_1\leq \beta_3 = \min(\beta_1, \beta_2, \beta_3)
    \end{align*}
    Here we first upper bounded the inner minimum by picking $m=3$ and then used that $z_k - z_3 - \lambda\Vert z\Vert_1\leq 0$ for $\lambda \geq 1$. 
    It is now easy to check that this upper bound is attained by $z=\vec{0}_3$, i.e., it holds for every $k\in[3]$
    \begin{align*}
        \sup_{z\in \mathbb{R}^3}\min_{m\in[3]}\left\{z_k - z_m + \beta_m - \lambda \Vert z\Vert_1 \right\} = \min(\beta_1, \beta_2, \beta_3)
    \end{align*}
    and therefore also
    \begin{align*}
        \max_{k=1,\ldots,3}\sup_{z\in \mathbb{R}^3}\min_{m\in[3]}\left\{z_k - z_m + \beta_m - \lambda \Vert z\Vert_1 \right\} = \min(\beta_1, \beta_2, \beta_3).
    \end{align*}
    Now consider the case $\lambda \in [0, 1)$. We inspect all three different values for $k$ that make up the outer maximum individually: 
    \begin{enumerate}
        \item \underline{$k=1$}: We can upper bound the sup over $z$ as follows:
        \begin{align*}
            \sup_{z\in \mathbb{R}^3}\min_{m\in[3]}\left\{z_1 - z_m + \beta_m - \lambda \Vert z\Vert_1 \right\}&\leq \sup_{z\in \mathbb{R}^3}\left\{(1 - \lambda) \beta_1 + \lambda (z_1 - z_3 + \beta_3) - \lambda \Vert z\Vert_1\right\}\\
            &\leq (1 - \lambda)\beta_1 + \lambda\beta_3\\
            &= (1 - \lambda)\max(\beta_1, \beta_2, \beta_3) + \lambda\min(\beta_1, \beta_2, \beta_3)
        \end{align*}
        In the first inequality we used that $\min(a, b, c)\leq (1-\lambda)a + \lambda b$. In the second inequality we used that $\lambda (z_1 - z_3) - \lambda \Vert z\Vert_1\leq 0$.\\
        The upper bound is tight because it is attained by the vector with coordinates $z_1=\beta_1 - \beta_2$, $z_2=0$, and $z_3=-(\beta_2 - \beta_3)$ and so
        \begin{align*}
            \sup_{z\in \mathbb{R}^3}\min_{m\in[3]}\left\{z_1 - z_m + \beta_m - \lambda \Vert z\Vert_1 \right\} = (1 - \lambda)\max(\beta_1, \beta_2, \beta_3) + \lambda\min(\beta_1, \beta_2, \beta_3)
        \end{align*}
        \item \underline{$k=2$}: We can upper bound the supremum over $z$ as follows:
        \begin{align*}
            \sup_{z\in \mathbb{R}^3}\min_{m\in[3]}\left\{z_2 - z_m + \beta_m - \lambda \Vert z\Vert_1 \right\}&\leq \sup_{z\in \mathbb{R}^3}\left\{(1 - \lambda) \beta_2 + \lambda (z_2 - z_3 + \beta_3) - \lambda \Vert z\Vert_1\right\}\\
            &\leq (1 - \lambda)\beta_2 + \lambda\beta_3\\
            &= (1 - \lambda)\mathrm{middle}(\beta_1, \beta_2, \beta_3) + \lambda\min(\beta_1, \beta_2, \beta_3).
        \end{align*}
        The justification of the steps is the same as for the case $k=1$ above. Again, this upper bound is attained by the vector with coordinates $z_1=0$, $z_2=\beta_2 - \beta_3$, and $z_3=0$ and so
        \begin{align*}
            \sup_{z\in \mathbb{R}^3}\min_{m\in[3]}\left\{z_2 - z_m + \beta_m - \lambda \Vert z\Vert_1 \right\} = (1 - \lambda)\mathrm{middle}(\beta_1, \beta_2, \beta_3) + \lambda\min(\beta_1, \beta_2, \beta_3)
        \end{align*}
        \item \underline{$k=3$}: We can upper bound the supremum over $z$ as follows:
        \begin{align*}
            \sup_{z\in \mathbb{R}^3}\min_{m\in[3]}\left\{z_3 - z_m + \beta_m - \lambda \Vert z\Vert_1 \right\}\leq \sup_{z\in\mathbb{R}^3}\beta_3 - \lambda \Vert z\Vert_1 = \beta_3 = \min(\beta_1, \beta_2, \beta_3)
        \end{align*}
        Here we again upper bounded the minimum by picking $m = 3$. This upper bound is attained by setting $z=\vec{0}_3$ and so
        \begin{align*}
            \sup_{z\in \mathbb{R}^3}\min_{m\in[3]}\left\{z_3 - z_m + \beta_m - \lambda \Vert z\Vert_1 \right\} = \min(\beta_1, \beta_2, \beta_3). 
        \end{align*}
    \end{enumerate}
    Taking the max over $k$ of these three values we obtain that for $\lambda \in[0, 1)$ it holds
    \begin{align*}
        \max_{k\in[3]}\sup_{z\in \mathbb{R}^3}\min_{m\in[3]}\left\{z_k - z_m + \beta_m - \lambda \Vert z\Vert_1 \right\} = \lambda \min(\beta_1, \beta_2, \beta_3) + (1 - \lambda)\max(\beta_1, \beta_2, \beta_3).
    \end{align*}
    The maximizer in this case is $z^*=(\beta_1 - \beta_2, 0, -(\beta_2 - \beta_3))$ which clearly satisfies $z^*\in[-1, 1]^3$.
    Combining this with the value for $\lambda \geq 1$, we obtain the formula given in the lemma.
\end{proof}

\begin{lemma}\label{lemma:opt_prob_3sat}
    Associate with a 1-in-3 SAT+ instance $W=\{w_1,\ldots,w_N\}$ the optimization problem
    \begin{align}
    \max_{\beta\in\mathbb{R}^d} \quad & \sum_{i=1}^N \max_{j\in\mathcal{S}_i}\beta_j \label{eq:one_in_three_sat_opt_problem}\\
        \mathrm{s.t.} \quad & \quad 0\leq\beta_k\leq 1\quad \forall k\in[d], \nonumber\\
                           & \quad \sum_{j\in\mathcal{S}_i} \beta_j = 1\quad \forall i\in[N]\nonumber
    \end{align}
    Then $W$ is satisfiable if and only if the optimal value of \eqref{eq:one_in_three_sat_opt_problem} is $N$. If $W$ is not satisfiable, 
    then the optimal value of \eqref{eq:one_in_three_sat_opt_problem} is upper bounded by $N - 1$.
\end{lemma}

\begin{proof}
    W.l.o.g.\ assume that each clause $w_i=\bigvee_{j\in\mathcal{S}_i}v_j$ consists of \textit{distinct} literals. 
    We start by proving that the optimal value of \eqref{eq:one_in_three_sat_opt_problem} is $N$ if and only if $W$ is satisfiable. 
    Indeed, when $W$ is satisfiable, then the vector $\beta\in\mathbb{R}^d$ defined by 
    \begin{align*}
        \beta_i = \begin{cases}
            1, & \text{if } v_i \text{ is true}\\
            0, & \text{otherwise}
        \end{cases}
    \end{align*}
    is feasible for \eqref{eq:one_in_three_sat_opt_problem} and attains the value $N$ which is also an upper bound for the optimal value 
    since $\max_{j\in\mathcal{S}_i}\beta_j\leq 1$ by the constraints $\beta_k\in[0, 1]$. Conversely, if the optimal value of 
    \eqref{eq:one_in_three_sat_opt_problem} is $N$, then each summand $\max_{j\in\mathcal{S}_i}\beta_j$ must be $1$ which together with 
    the constraint $\sum_{j\in\mathcal{S}_i} \beta_j = 1$ implies that exactly one of $\beta_j$, $j\in\mathcal{S}_i$, is $1$ and 
    the other two are $0$. The corresponding assignments $v_k$ then clearly satisfy $W$.\\

    We now show that if the optimal value of \eqref{eq:one_in_three_sat_opt_problem} is strictly greater than $N - 1$, then it must indeed be equal to 
    $N$ and so $W$ is satisfiable. Hence, if $W$ is not satisfiable, then the corresponding optimal value of \eqref{eq:one_in_three_sat_opt_problem} 
    is at most $N - 1$. Assume that $\beta\in\mathbb{R}^d$ is feasible for \eqref{eq:one_in_three_sat_opt_problem} with optimal value strictly greater than $N - 1$. Define the set
    \begin{align*}
        \mathcal{I} = \bigcup_{i=1}^N \left\{k\,\Big\vert\, k = \min\argmax_{j\in\mathcal{S}_i}\beta_j\right\},
    \end{align*}
    i.e., for each summand $\max_{j\in\mathcal{S}_i}\beta_j$ we pick the maximizing index $k$ (if there are ties, we pick the one with the smaller index). 
    We show that $\bar\beta$ defined by
    \begin{align*}
        \bar\beta_i = \begin{cases}
            1, \quad \text{if } i\in\mathcal{I}\\
            0, \quad \text{else}
        \end{cases}
    \end{align*}
    is feasible and attains the value $N$ and hence corresponds to an instance that satisfies $W$. Clearly, $\bar\beta$ attains the value $N$ because we took 
    the maximizer of each summand and set it to $1$. Therefore, it only remains to check feasibility. Since we assumed distinct indices in each $\mathcal{S}_i$, 
    the only two problematic cases arise when there is a constraint $\sum_{j\in\mathcal{S}_i} \beta_j = 1$  for which two or three of the $\bar\beta_j$ are 
    set to $1$ (one $\bar\beta_j$ is definitely set to $1$ -- the maximizing one). We now argue that such situations cannot arise under our assumptions. 
    First assume that there is a constraint such that $\mathcal{S}_i\subset\mathcal{I}$, i.e., all three indices of $\mathcal{S}_i$ appear in $\mathcal{I}$. 
    Since $\beta_k\geq \frac{1}{3}$ for $k\in\mathcal{I}$, the constraint $\sum_{j\in\mathcal{S}_i}\beta_j=1$ implies $\beta_{j}=\frac{1}{3}$ for each $j\in\mathcal{S}_i$. 
    This means that $\max_{j\in\mathcal{S}_i}\beta_j=\frac{1}{3}$, i.e., we already lose $\frac{2}{3}$ towards the optimal value of $N$ in \eqref{eq:one_in_three_sat_opt_problem}. 
    Further, since each of $\beta_j$, $j\in\mathcal{S}_i$, is a maximizer in a clause, there must be at least two other clauses where the maximum is also $\frac{1}{3}$ 
    and so we lose $\frac{2}{3}$ towards the optimum in these clauses as well. Hence, in total we lose $2$ units towards the optimal value of $N$ which contradicts 
    the assumption that $\beta$ attains a value strictly greater than $N-1$. Now assume that there is a constraint $\sum_{j\in\mathcal S_i}\beta_j=1$
    such that exactly two indices in $\mathcal S_i$ belong to $\mathcal I$, say $\mathcal S_i\cap\mathcal I=\{i_1,i_2\}$. Let
    $r:=\min\argmax_{j\in\mathcal S_i}\beta_j$ be the index selected from the $i$-th summand, and let $s$ be the other element of
    $\mathcal S_i\cap\mathcal I$. Since $s\in\mathcal I$ but $s$ was not selected from the $i$-th summand, there exists another clause $h\neq i$ from which $s$
    was selected, i.e., $s=\min\argmax_{j\in\mathcal S_h}\beta_j$. Hence the loss relative to the
    upper bound $N$ is $1-\beta_r$ in the $i$-th summand and $1-\beta_s$ in the $h$-th summand. Since $r,s\in\mathcal S_i$ and $\beta$ is feasible, we have
    $\beta_r+\beta_s\leq\sum_{j\in\mathcal S_i}\beta_j=1$. Thus the combined loss over these two summands is at least
    $2-(\beta_r+\beta_s)\geq1$, contradicting the assumption that $\beta$ attains a value strictly greater than $N-1$. Therefore no constraint can contain two
    indices from $\mathcal I$.
    Hence, in each clause exactly one $\bar\beta_{j}$ with $j\in\mathcal{S}_i$ is set to $1$ and so by modifying $\beta$ to $\bar\beta$ we indeed find a feasible solution that attains $N$. 
\end{proof}

\subsection{Proof of Theorem~\ref{thm:np_hardness_regret}}

\begin{proof}
    We prove Theorem~\ref{thm:np_hardness_regret} by showing that the NP-complete 1-in-3 SAT+ problem can be reduced to verifying whether $R(\theta) \geq \frac{2}{3}$ 
    for a WDRRO problem with model parameters chosen as follows:
    
    \begin{itemize}
        \item $\delta = 2$.
        \item $\Theta=\{\theta\in [0, 1]^d \, \vert \, \sum_{j\in\mathcal{S}_i}\theta^\top e_j = 1 \,\, \forall i\in[N]\}$. Here, $e_{j}$ denotes the $j$-th standard basis vector in $\mathbb{R}^d$. 
        The $i$-th constraint in the definition of $\Theta$ corresponds to the $i$-th clause.
        \item $\theta = \frac{1}{3}\vec{1}_d\in \mathbb{R}^d$. Clearly, for this vector it holds $\theta\in\Theta$. 
        \item $\Xi=\mathbb{R}^{3N}$. For $z\in\Xi$ we write $z_k$ for the $k$-th block of length $3$ and $(z_k)_i$ for the $i$-th entry of the $k$-th block.
        \item $\hat\xi_i=(-M \vec{1}_3, \ldots, -M\vec{1}_3, \vec{0}_3, -M \vec{1}_3, \ldots, -M\vec{1}_3)^\top\in\mathbb{R}^{3N}$ with $M =5N$. The $\vec{0}_3$ vector is the $i$-th 
        block of $\hat\xi_i$.
        \item $K=3N$
        \item For $k\in[N]$ define $a_{k, 1}=(\vec{0}_3, \ldots, \vec{0}_3, e_{1}, \vec{0}_3, \ldots, \vec{0}_3)\in\mathbb{R}^{3N}$, $a_{k, 2}=(\vec{0}_3, \ldots, \vec{0}_3, e_{2}, \vec{0}_3, \ldots, \vec{0}_3)\in\mathbb{R}^{3N}$, $a_{k,3}=(\vec{0}_3, \ldots, \vec{0}_3, e_{3}, \vec{0}_3, \ldots, \vec{0}_3)\in\mathbb{R}^{3N}$. Again, $e_{i}$ is the $i$-th standard basis vector in $\mathbb{R}^3$. The standard basis vectors $e_1$, $e_2$, and $e_3$ are located in the $k$-th block of dimension three of the vectors $a_{k, 1}$, $a_{k, 2}$, and $a_{k, 3}$.
        \item $b_{k, 1}=-e_{j^k_1}\in\mathbb{R}^d$, $b_{k, 2}=-e_{j^k_2}\in\mathbb{R}^d$, $b_{k, 3}=-e_{j^k_3}\in\mathbb{R}^d$ where $\mathcal{S}_k=\{j^k_1, j^k_2, j^k_3\}$, i.e., $b_{k, i}$ corresponds to the $i$-th variable appearing in the $k$-th clause of the 1-in-3 SAT+ problem.
        \item $c_{k, i}=0$ for all $k\in[N]$, $i\in[3]$. 
        \item $p=1$
        \item $\Vert\cdot\Vert = \Vert\cdot\Vert_1$. 
    \end{itemize}

    Plugging these choices into the dual reformulation of the regret from \eqref{eq:regret_data_driven} we obtain
    \begin{footnotesize}
    \begin{align}
        R(\theta) &= \sup_{\beta\in\Theta}\inf_{\lambda\geq 0}\left\{\lambda\delta + \frac{1}{N}\sum_{i=1}^N\sup_{z\in \Xi}\left\{\max_{\substack{k\in[N] \\ j\in[3]}}\left\{a_{k,j}^\top z + b_{k, j}^\top\theta\right\} - \max_{\substack{m\in[N] \\ l\in[3]}}\left\{a_{m,l}^\top z + b_{m, l}^\top\beta\right\} - \lambda \Vert z - \hat\xi_i\Vert_1\right\}\right\}\nonumber\\
        &= \sup_{\beta\in\Theta}\inf_{\lambda\geq 0}\left\{-\frac{1}{3} + \lambda\delta + \frac{1}{N}\sum_{i=1}^N\sup_{z\in \Xi}\left\{\max_{\substack{k\in[N] \\ j\in[3]}}(z_k)_j - \max_{\substack{m\in[N] \\ l\in[3]}}\left\{(z_m)_l - \beta_{j^m_l}\right\} - \lambda \Vert z_i\Vert_1 - \lambda\sum_{r\neq i}\Vert z_r + M\vec{1}_3\Vert_1\right\}\right\}\label{eq:regret_np_hardness_proof}.
    \end{align}
    \end{footnotesize}

    We now bound the regret from above and below. We begin with the bound from below. In the $i$-th summand of \eqref{eq:regret_np_hardness_proof} the supremum over $z\in\Xi$ can be 
    lower bounded by letting $z$ be the vector with $z_r = -M\vec{1}_3$ for $r\neq i$ and $z_i$ being the optimizer from Lemma~\ref{lemma:np_hardness_helper}. 
    This particular choice corresponds to switching off all components of $z$ except $z_i$ and results in the following lower bound:
    \begin{align}
        &\sup_{z\in \Xi}\left\{\max_{\substack{k\in[N] \\ j\in[3]}}(z_k)_j - \max_{\substack{m\in[N] \\ l\in[3]}}\left\{(z_m)_l - \beta_{j^m_l}\right\} - \lambda \Vert z_i\Vert_1 - \lambda\sum_{r\neq i}\Vert z_r + M\vec{1}_3\Vert_1\right\}\nonumber\\
        &\quad \geq \max_{j\in[3]}(z_i)_j - \max_{l\in[3]} \left\{(z_i)_l - \beta_{j^i_l}\right\} - \lambda \Vert z_i\Vert_1\label{eq:inequality_becoming_equality}\\
        &\quad = \max_{j\in[3]}\sup_{z\in\mathbb{R}^3}\min_{l\in[3]}\left\{z_j - z_l + \beta_{j^i_l}-\lambda\Vert z\Vert_1\right\}\nonumber\\
        &\quad = \max\left\{\min_{j\in\mathcal{S}_i}\beta_j, \lambda \min_{j\in\mathcal{S}_i}\beta_j + (1 - \lambda)\max_{j\in\mathcal{S}_i}\beta_j\right\}\nonumber
    \end{align}
    The first inequality holds by our specific choice of $z$. Indeed, by Lemma~\ref{lemma:np_hardness_helper} we have $z_i\in[-1, 1]^3$ and $z_r=-M\vec{1}_3$ for $r\neq i$. 
    By our choice of $M$ we thus have $(z_i)_j\geq (z_r)_l + 1$ which together with $\beta_k\in[0, 1]$ shows that the $i$-th term always wins in the maxima. 
    The following equality results from rearranging terms and using the fact that $z_i$ was chosen to be the optimizer of problem \eqref{eq:formula_helper_lemma} 
    from Lemma~\ref{lemma:np_hardness_helper}. The final equality is Lemma~\ref{lemma:np_hardness_helper}. Plugging this bound back into \eqref{eq:regret_np_hardness_proof} 
    yields the following lower bound on the regret:
    \begin{align}
        R(\theta) &\geq \sup_{\beta\in\Theta}\inf_{\lambda\geq 0}\left\{-\frac{1}{3} + \lambda\delta + \frac{1}{N}\sum_{i=1}^N\max\left\{\min_{j\in\mathcal{S}_i}\beta_j, \lambda \min_{j\in\mathcal{S}_i}\beta_j + (1 - \lambda)\max_{j\in\mathcal{S}_i}\beta_j\right\}\right\}\nonumber\\
        &= -\frac{1}{3} + \sup_{\beta\in\Theta}\frac{1}{N}\sum_{i=1}^N\max_{j\in\mathcal{S}_i}\beta_j\label{eq:lower_bound_np_hardness}
    \end{align}
    The equality here holds as the infimum over $\lambda$ is attained at $\lambda=0$. This is because by $\delta = 2$ and $\max_{j\in\mathcal{S}_i}\beta_j - \min_{j\in\mathcal{S}_i}\beta_j\in[0, 1]$ the slope of the inner expression as a function of $\lambda$ is always positive and so the infimum is attained at $\lambda=0$. 

    We now turn to proving an upper bound on the regret. To do so we first show that, for every $\lambda>2/M$, the inequality in \eqref{eq:inequality_becoming_equality} is in fact tight. 
    Fix $i\in[N]$, $\beta\in\Theta$, and $\lambda>2/M$, and denote the left-hand side of \eqref{eq:inequality_becoming_equality} by $\Gamma_i(\beta,\lambda)$. Let 
    $\mathcal Z_i(\lambda):=\{z\in\Xi: z_i\in[-1/\lambda,1/\lambda]^3,\ z_r\in[-M-1/\lambda,-M+1/\lambda]^3\ \forall r\neq i\}$. Then
    \begin{align*}
    \Gamma_i(\beta,\lambda)
    &=
    \sup_{z\in\mathcal Z_i(\lambda)}
    \left\{
    \max_{\substack{k\in[N]\\ j\in[3]}}(z_k)_j
    -
    \max_{\substack{m\in[N]\\ l\in[3]}}
    \{(z_m)_l-\beta_{j_l^m}\}
    -\lambda\|z_i\|_1
    -\lambda\sum_{r\neq i}\|z_r+M\vec 1_3\|_1
    \right\} \\
    &=
    \sup_{z\in\mathcal Z_i(\lambda)}
    \left\{
    \max_{j\in[3]}(z_i)_j
    -
    \max_{\substack{m\in[N]\\ l\in[3]}}
    \{(z_m)_l-\beta_{j_l^m}\}
    -\lambda\|z_i\|_1
    -\lambda\sum_{r\neq i}\|z_r+M\vec 1_3\|_1
    \right\} \\
    &\le
    \sup_{z\in\mathcal Z_i(\lambda)}
    \left\{
    \max_{j\in[3]}(z_i)_j
    -
    \max_{l\in[3]}\{(z_i)_l-\beta_{j_l^i}\}
    -\lambda\|z_i\|_1
    -\lambda\sum_{r\neq i}\|z_r+M\vec 1_3\|_1
    \right\} \\
    &\le
    \sup_{u\in[-1/\lambda,1/\lambda]^3}
    \left\{
    \max_{j\in[3]}u_j
    -
    \max_{l\in[3]}\{u_l-\beta_{j_l^i}\}
    -\lambda\|u\|_1
    \right\} \\
    &=
    \max_{j\in[3]}
    \sup_{u\in[-1/\lambda,1/\lambda]^3}
    \min_{l\in[3]}
    \left\{
    u_j-u_l+\beta_{j_l^i}-\lambda\|u\|_1
    \right\} \\
    &=
    \max_{j\in[3]}
    \sup_{u\in\mathbb R^3}
    \min_{l\in[3]}
    \left\{
    u_j-u_l+\beta_{j_l^i}-\lambda\|u\|_1
    \right\} \\
    &=
    \max\left\{
    \min_{j\in\mathcal S_i}\beta_j,\,
    \lambda\min_{j\in\mathcal S_i}\beta_j
    +
    (1-\lambda)\max_{j\in\mathcal S_i}\beta_j
    \right\}.
    \end{align*}
    The first equality follows because $\beta\in[0,1]^d$ implies that the difference
    of the two maxima in the left-hand side \eqref{eq:inequality_becoming_equality} always lies in $[0,1]$. Hence, if
    $z\notin\mathcal Z_i(\lambda)$, the transportation penalty is strictly larger
    than $1$, so the objective value is negative; such points cannot improve the
    supremum, since the lower-bound construction in \eqref{eq:inequality_becoming_equality} shows that the supremum must be nonnegative. The
    second equality uses $\lambda>2/M$, which makes
    $[-M-1/\lambda,-M+1/\lambda]$ and $[-1/\lambda,1/\lambda]$ disjoint, so the
    $i$-th block wins the first maximum. For the second maximum we use
    $\max_{m\in[N],l\in[3]}\{(z_m)_l-\beta_{j_l^m}\}\ge\max_{l\in[3]}\{(z_i)_l-\beta_{j_l^i}\}$, which gives the first
    inequality. The second inequality drops the nonpositive penalty terms involving
    $r\neq i$ and renames $z_i$ as $u$. The next equality rewrites the difference of
    maxima as a max-min expression. The restriction on $u$ can then be removed
    because Lemma~\ref{lemma:np_hardness_helper} admits an optimizer in
    $[-1/\lambda,1/\lambda]^3$: if $\lambda\ge1$, the optimizer is $u=0$, and if
    $\lambda<1$, the lemma gives an optimizer in $[-1,1]^3\subseteq
    [-1/\lambda,1/\lambda]^3$. The final equality is
    Lemma~\ref{lemma:np_hardness_helper}. Since the lower-bound construction already
    shows the reverse inequality, equality in \eqref{eq:inequality_becoming_equality} follows.

    This allows us to obtain the following upper bound on the regret:
    \begin{small}
    \begin{align}
        R(\theta)&\leq \sup_{\beta\in\Theta}\inf_{\lambda> \frac{2}{M}}\left\{-\frac{1}{3} + \lambda\delta + \frac{1}{N}\sum_{i=1}^N\sup_{z\in \Xi}\left\{\max_{\substack{k\in[N] \\ j\in[3]}}(z_k)_j - \max_{\substack{m\in[N] \\ l\in[3]}}\left\{(z_m)_l - \beta_{j^m_l}\right\} - \lambda \Vert z_i\Vert_1 - \lambda\sum_{r\neq i}\Vert z_r + M\vec{1}_3\Vert_1\right\}\right\}\nonumber\\
        &=\sup_{\beta\in\Theta}\inf_{\lambda> \frac{2}{M}}\left\{-\frac{1}{3} + \lambda\delta + \frac{1}{N}\sum_{i=1}^N\max\left\{\min_{j\in\mathcal{S}_i}\beta_j, \lambda \min_{j\in\mathcal{S}_i}\beta_j + (1 - \lambda)\max_{j\in\mathcal{S}_i}\beta_j\right\}\right\}\nonumber\\
        &= -\frac{1}{3} + \frac{2\delta}{M} + \sup_{\beta\in\Theta}\frac{1}{N}\sum_{i=1}^N\max\left\{\min_{j\in\mathcal{S}_i}\beta_j, \frac{2}{M} \min_{j\in\mathcal{S}_i}\beta_j + \left(1 - \frac{2}{M}\right)\max_{j\in\mathcal{S}_i}\beta_j\right\}\nonumber\\
        &\leq -\frac{1}{3} + \frac{2\delta}{M} + \sup_{\beta\in\Theta}\frac{1}{N}\sum_{i=1}^N\max_{j\in\mathcal{S}_i}\beta_j\label{eq:upper_bound_np_hardness}
    \end{align}
    \end{small}
    The first inequality results from the restriction of the minimization interval for $\lambda$ from $[0, \infty)$ to $(\frac{2}{M}, \infty)$. 
    The first equality is $\Gamma_i(\beta, \lambda)=\max\left\{
    \min_{j\in\mathcal S_i}\beta_j,\,
    \lambda\min_{j\in\mathcal S_i}\beta_j
    +
    (1-\lambda)\max_{j\in\mathcal S_i}\beta_j
    \right\}$ which holds for $\lambda > \frac{2}{M}$. Then we again argue that because of $\delta + \min_{j\in\mathcal{S}_i}\beta_j  - \max_{j\in\mathcal{S}_i}\beta_j\geq \delta - 1 >0$ 
    the piecewise slopes of the inner expression as a function of $\lambda$ are positive and so the infimum equals the limit as $\lambda\downarrow\frac{2}{M}$. 
    Finally, we upper bound 
    $\max\left\{\min_{j\in\mathcal{S}_i}\beta_j, \frac{2}{M}\min_{j\in\mathcal{S}_i}\beta_j + \left(1 - \frac{2}{M}\right)\max_{j\in\mathcal{S}_i}\beta_j\right\}\leq \max_{j\in\mathcal{S}_i}\beta_j$.

    Combining the lower bound \eqref{eq:lower_bound_np_hardness} and the upper bound \eqref{eq:upper_bound_np_hardness} we obtain
    \begin{align*}
        -\frac{1}{3} + \sup_{\beta\in\Theta}\frac{1}{N}\sum_{i=1}^N\max_{j\in\mathcal{S}_i}\beta_j\leq R(\theta)\leq -\frac{1}{3} + \frac{2\delta}{M} + \sup_{\beta\in\Theta}\frac{1}{N}\sum_{i=1}^N\max_{j\in\mathcal{S}_i}\beta_j.
    \end{align*}

    By Lemma~\ref{lemma:opt_prob_3sat}, we know that $\sup_{\beta\in\Theta}\frac{1}{N}\sum_{i=1}^N \max_{j\in\mathcal{S}_i}\beta_j=1$ if $W$ is 
    satisfiable and otherwise it holds $\sup_{\beta\in\Theta}\frac{1}{N}\sum_{i=1}^N \max_{j\in\mathcal{S}_i}\beta_j\leq 1 - \frac{1}{N}$. Hence, if $W$ 
    is satisfiable, then $R(\theta)\in[\frac{2}{3}, \frac{2}{3} + \frac{2\delta}{M}]$. On the other hand, if $W$ is not satisfiable, 
    then $R(\theta)\leq\frac{2}{3} - \frac{1}{N} + \frac{2\delta}{M} < \frac{2}{3}$ where the last inequality follows from $M > 2\delta N$ for our choice of $M$ and $\delta$. 
    This shows that $R(\theta)\geq \frac{2}{3}$ if and only if $W$ is satisfiable and hence concludes the proof.
\end{proof}

\section{Proofs of Section~\ref{sec:algorithms}}
\label{app:algs_proofs}

\subsection{Proof of Lemma~\ref{lemma:relaxation_full}}

\begin{proof}
We carry out the proof only for the case $p>1$ as the other case is similar. Recall that the constraint $\hat\xi_i +\frac{q_{ik}}{\gamma_{ik}}\in\Xi$ in \eqref{eq:regret_opt_problem} for $\gamma_{ik}=0$ is interpreted as $q_{ik}\in\mathrm{rec}(\Xi)$. By the definition of the extended perspective function this constraint can be written as $\gamma_{ik}\chi_\Xi\left(\hat\xi_i + \frac{q_{ik}}{\gamma_{ik}}\right)\leq 0$ and similarly for the constraint $\frac{z_{ik}}{\gamma_{ik}}\in\Theta$. Instead of having these two constraints we can subtract $\gamma_{ik}\chi_\Xi\left(\hat\xi_i +\frac{q_{ik}}{\gamma_{ik}}\right)$ and $\gamma_{ik}\chi_\Theta\left(\frac{z_{ik}}{\gamma_{ik}}\right)$ from the objective. Since $\Xi$ and $\Theta$ are assumed to be closed convex sets, these two functions are proper, lower-semicontinuous, and convex. Similarly, we can drop the constraint $\beta\in\Theta$ and instead subtract $\chi_\Theta(\beta)$ from the objective. Dualizing the remaining constraints, the program \eqref{eq:regret_opt_problem} can therefore be rewritten as

\begin{align}
    \sup_{\substack{\beta, \gamma_{ik} \geq 0, q_{ik} \\ z_{ik}, t_{ik}}} \quad 
    &\frac{1}{N} \sum_{i=1}^N \sum_{k=1}^K 
    \Biggl\{ \gamma_{ik} \Bigl( a_k^\top \hat{\xi}_i + b_k^\top\theta + c_k\Bigr) 
    + a_k^\top q_{ik} - t_{ik} 
    - \gamma_{ik}\chi_{\Xi} \left(\hat{\xi}_i + \frac{q_{ik}}{\gamma_{ik}}\right) 
    - \gamma_{ik}\chi_{\Theta} \left(\frac{z_{ik}}{\gamma_{ik}}\right) \nonumber\\
    &\quad \quad \quad \quad - \chi_{\Theta}(\beta) + \inf_{\mu_{ik} \geq 0} \mu_{ik}^\top 
    \Bigl( t_{ik} \vec{1}_K - \gamma_{ik} (A\hat{\xi}_i + c) - Bz_{ik} - Aq_{ik} \Bigr) \Biggr\} \label{eq:original_prob}\\
    &+ \frac{1}{N} \sum_{i=1}^N 
    \Biggl\{ \inf_{s_i} s_i \Bigl(1 - \sum_{k=1}^K \gamma_{ik} \Bigr) 
    + \inf_{\tau_i} \tau_i^\top \Bigl(\beta - \sum_{k=1}^K z_{ik} \Bigr) \Biggr\} \nonumber\\
    &+ \inf_{\lambda \geq 0} \lambda 
    \Biggl( \delta^p - \frac{1}{N} \sum_{i=1}^N \sum_{k=1}^K \gamma_{ik} 
    \Bigl\| \frac{q_{ik}}{\gamma_{ik}} \Bigr\|^p \Biggr).\nonumber
\end{align}

Taking the Lagrangian dual we obtain after some calculations

\begin{align}
    \inf_{\substack{\lambda\geq 0,\, s_i, \\  \mu_{ik} \geq 0\, \tau_i}} &\lambda
    \delta^p + \frac{1}{N}\sum_{i=1}^N s_i + \sigma_{\Theta} \left(\frac{1}{N} \sum_{i=1}^N \tau_i \right)\nonumber\\
    &+ \frac{1}{N} \sum_{i=1}^N \sum_{k=1}^K \sup_{\gamma_{ik} \geq 0} 
    \Biggl\{ \gamma_{ik} \Bigl( a_k^\top \hat{\xi}_i + b_k^\top \theta + c_k - 
    \mu_{ik}^\top (A\hat{\xi}_i + c) - s_i \Bigr) \nonumber\\
    &\quad\quad\quad\quad\quad\quad\quad\quad\quad + \sup_{z_{ik}}\left\{-z_{ik}^\top (B^\top \mu_{ik} + \tau_i) - \gamma_{ik}\chi_{\Theta} \left(\frac{z_{ik}}{\gamma_{ik}}\right)\right\}\label{eq:lagrange_dual}\\
    &\quad\quad\quad\quad\quad\quad\quad\quad\quad + \sup_{q_{ik}} \left\{(a_k - A^\top \mu_{ik})^\top q_{ik} 
    - \lambda \gamma_{ik} \left\Vert\frac{q_{ik}}{\gamma_{ik}}\right\Vert^p 
    - \gamma_{ik}\chi_{\Xi} \left(\hat{\xi}_i + \frac{q_{ik}}{\gamma_{ik}}\right) \right\}\Biggr\},\nonumber\\
    \text{s.t.} \quad & \mu_{ik}^\top \vec{1}_K = 1, \quad \forall i \in [N],\, k \in [K].\nonumber
\end{align}

Strong duality can be justified as follows: The primal problem in \eqref{eq:original_prob} is an optimization problem with constraints $\sum_{k=1}^K\gamma_{ik} = 1$, $\sum_{k=1}^Kz_{ik}=\beta$, and $\frac{1}{N}\sum_{i=1}^N\sum_{k=1}^K\gamma_{ik}\left\Vert\frac{q_{ik}}{\gamma_{ik}}\right\Vert^p\leq\delta^p$, objective $\gamma_{ik} \Bigl( a_k^\top \hat{\xi}_i + b_k^\top\theta + c_k\Bigr) + a_k^\top q_{ik} - t_{ik}$ and domain of the objective given by
\begin{align*}
    C = \Theta\times\prod_{\substack{i\in[N]\\k\in[K]}}C_{ik}, \quad C_{ik} = \left\{(\gamma_{ik}, q_{ik}, z_{ik})\Big\vert \hat\xi_i + \frac{q_{ik}}{\gamma_{ik}}\in\Xi, \, \frac{z_{ik}}{\gamma_{ik}}\in\Theta\right\}
\end{align*}
with the usual definition that for $\gamma_{ik}=0$ the constraints require $q_{ik}\in\mathrm{rec}(\Xi)$ and $z_{ik}\in\mathrm{rec}(\Theta)$. By \cite[Proposition 5.3.6]{bertsekas2009convex} strong duality holds if one can find a point $(\beta, \gamma_{ik}, q_{ik}, z_{ik}, t_{ik})_{i\in[N], k\in[K]}$ such that $(\beta, \gamma_{ik}, q_{ik}, z_{ik})_{i\in[N], k\in[K]}\in\mathrm{ri}(C)$ and $\frac{1}{N}\sum_{i=1}^N\sum_{k=1}^K\gamma_{ik}\left\Vert\frac{q_{ik}}{\gamma_{ik}}\right\Vert^p<\delta^p$. However, it holds $\mathrm{ri}(C) = \mathrm{ri}(\Theta)\times\prod_{\substack{i\in[N],\\k\in[K]}}\mathrm{ri}(C_{ik})$
 and so it remains to compute the relative interiors of $C_{ik}$. For this, note that $C_{ik}$ can be written as

\begin{small}
 \begin{align*}
    C_{ik} = \mathrm{cone}(\{1\}\times (\Xi - \hat\xi_i)\times \Theta)\cup (\{0\}\times\mathrm{rec}(\Xi - \hat\xi_i)\times \mathrm{rec}(\Theta)) = \mathrm{cl}(\mathrm{cone}(\{1\}\times (\Xi - \hat\xi_i)\times \Theta))
\end{align*}
\end{small}

where the last equality holds by \cite[Theorem 8.2]{Rockafellar1970} since $\Xi$ and $\Theta$ are closed. From this and Theorem 6.3 and Corollary 6.8.1 in \cite{Rockafellar1970} we obtain
\begin{align*}
    \mathrm{ri}(C_{ik}) &= \mathrm{ri}(\mathrm{cl}(\mathrm{cone}(\{1\}\times (\Xi - \hat\xi_i)\times \Theta))) = \mathrm{ri}(\mathrm{cone}(\{1\}\times (\Xi - \hat\xi_i)\times \Theta)) \\
    &= \left\{(\gamma_{ik}, q_{ik}, z_{ik})\Big\vert \gamma_{ik} > 0,\, \hat\xi_i + \frac{q_{ik}}{\gamma_{ik}}\in\mathrm{ri}(\Xi), \, \frac{z_{ik}}{\gamma_{ik}}\in\mathrm{ri}(\Theta)\right\}
\end{align*}
We can now find a point that satisfies the desired conditions: Pick any $\beta\in\mathrm{ri}(\Theta)$ and set $\gamma_{ik} = \frac{1}{K}$ and $z_{ik} = \gamma_{ik}\beta$ which implies $\frac{z_{ik}}{\gamma_{ik}}\in\mathrm{ri}(\Theta)$. Furthermore, pick a point $\bar\xi$ from $\mathrm{ri}(\Xi)$ and let $q_{ik} = \varepsilon(\bar\xi - \hat{\xi}_i)$. For $\varepsilon$ small enough this ensures $\hat\xi_i + \frac{q_{ik}}{\gamma_{ik}}\in \mathrm{ri}(\Xi)$ and also $\frac{1}{N} \sum_{i=1}^N \sum_{k=1}^K \gamma_{ik} \Bigl\| \frac{q_{ik}}{\gamma_{ik}} \Bigr\|^p< \delta^p$. Here we assume for now that $\delta > 0$ (we deal with the case $\delta=0$ later). Finally, we also pick $t_{ik}$ large enough so that the epigraphical constraint also holds strictly. Then by \cite[Section 5.2.3]{boyd_convex_2023} strong duality holds.

Using our conventions for the definition of perspective functions at $\gamma_{ik}=0$ as well as the convention that $0\cdot\infty=\infty$ we can explicitly solve for the supremum in $z_{ik}$:
\begin{align}
    \sup_{z_{ik}} -z_{ik}^\top (B^\top \mu_{ik} + \tau_i) - \gamma_{ik}\chi_{\Theta} \left(\frac{z_{ik}}{\gamma_{ik}}\right) = \gamma_{ik} \sigma_\Theta(-(B^\top\mu_{ik} + \tau_i))\label{eq:sup_z_constr}.
\end{align}
Now we optimize over $q_{ik}$. If $\gamma_{ik} > 0$ the substitution $\xi_{ik} = \hat{\xi}_i + \frac{q_{ik}}{\gamma_{ik}}$ gives
\begin{align}
    &\sup_{q_{ik}}(a_k - A^\top \mu_{ik})^\top q_{ik} 
    - \lambda \gamma_{ik} \left\Vert\frac{q_{ik}}{\gamma_{ik}}\right\Vert^p 
    - \gamma_{ik}\chi_{\Xi} \left(\hat{\xi}_i + \frac{q_{ik}}{\gamma_{ik}}\right)\nonumber\\
    &\quad=\gamma_{ik}\cdot \sup_{\xi_{ik}}(a_k - A^\top \mu_{ik})^\top (\xi_{ik} - \hat\xi_i) 
    - \lambda \left\Vert\hat\xi_i - \xi_{ik}\right\Vert^p 
    - \chi_{\Xi} \left(\xi_{ik}\right)\label{eq:sup_q_constr0}\\
    &\quad= -\gamma_{ik}(a_k - A^\top \mu_{ik})^\top \hat\xi_i + \gamma_{ik} (f_1 + f_2)^*(a_k - A^\top\mu_{ik})\label{eq:sup_q_constr1}.
\end{align}
where $f_1(\xi) = \lambda \Vert \hat\xi_i-\xi\Vert^p$ and $f_2(\xi) = \chi_\Xi(\xi)$. The same identity continues to hold at $\gamma_{ik}=0$ under the extended-perspective convention used throughout. Since $\mathrm{ri}(\Xi)\neq\emptyset$ and $\mathrm{dom}(\lambda\Vert \hat\xi_i - \cdot\Vert^p)=\mathbb{R}^n$ by infimal convolution (see \cite[Theorem 11.23 (a)]{rockafellar_variational_1998}) one obtains
\begin{align*}
    (f_1 + f_2)^*(y) = \inf_{u}\,  f_1^*(u) + f_2^*(y - u)
\end{align*}
with the infimum $u$ being attained. The respective Fenchel conjugates are easily computed as
\begin{align*}
    f_1^*(y) = y^\top\hat{\xi}_i + \phi(q) \lambda \left\Vert \frac{y}{\lambda} \right\Vert_*^q, \quad f_2^*(y) = \sigma_\Xi(y)
\end{align*}
where $q=\frac{p}{p-1}$ and $\phi(q) = \frac{(q-1)^{q-1}}{q^q}$. Note that the formula for $f_1^*$ is also correct for $\lambda=0$. This then implies that
\begin{align}\label{eq:sup_q_constr2}
    &(f_1 + f_2)^*(a_k - A^\top\mu_{ik}) \nonumber\\
    &\quad= \inf_{u_{ik}} \sigma_\Xi(u_{ik}) + (a_k - A^\top\mu_{ik} - u_{ik})^\top\hat\xi_i  + \phi(q) \lambda \left\Vert \frac{a_k - A^\top\mu_{ik} - u_{ik}}{\lambda} \right\Vert_*^q.
\end{align}
Plugging \eqref{eq:sup_z_constr}, \eqref{eq:sup_q_constr1}, and \eqref{eq:sup_q_constr2} back into \eqref{eq:lagrange_dual} gives
\begin{align*}
    \inf_{\substack{\lambda\geq 0,\, s_i, \\ \mu_{ik} \geq 0\,\tau_i}} &\lambda
    \delta^p + \frac{1}{N}\sum_{i=1}^N s_i + \sigma_{\Theta} \left(\frac{1}{N} \sum_{i=1}^N \tau_i \right)\\
    &+ \frac{1}{N} \sum_{i=1}^N \sum_{k=1}^K \sup_{\gamma_{ik} \geq 0} 
    \Biggl\{ \gamma_{ik} \Bigl( a_k^\top \hat{\xi}_i + b_k^\top \theta + c_k - 
    \mu_{ik}^\top (A\hat{\xi}_i + c) + \sigma_\Theta(-(B^\top \mu_{ik} + \tau_i))\\
    &\quad\quad\quad\quad\quad\quad\quad\quad\quad  + \inf_{u_{ik}}\left\{\sigma_\Xi(u_{ik}) - u_{ik}^\top\hat\xi_i + \phi(q) \lambda \left\Vert \frac{a_k - A^\top\mu_{ik} - u_{ik}}{\lambda} \right\Vert_*^q\right\} - s_i \Bigr) \Biggr\},\\
    \text{s.t.} \quad & \mu_{ik}^\top \vec{1}_K = 1, \quad \forall i \in [N],\, k \in [K].
\end{align*}

Now using the fact the infimum in \eqref{eq:sup_q_constr2} is attained one realizes that
\begin{align*}
    \sup_{\gamma_{ik} \geq 0} 
    \Biggl\{ \gamma_{ik} \Bigl( &a_k^\top \hat{\xi}_i + b_k^\top \theta + c_k - 
    \mu_{ik}^\top (A\hat{\xi}_i + c) + \sigma_\Theta(-(B^\top \mu_{ik} + \tau_i))\\
    &+ \inf_{u_{ik}}\left\{\sigma_\Xi(u_{ik}) - u_{ik}^\top\hat\xi_i + \phi(q) \lambda \left\Vert \frac{a_k - A^\top\mu_{ik} - u_{ik}}{\lambda} \right\Vert_*^q\right\} - s_i \Bigr) \Biggr\}
\end{align*}
is zero if there exists $u_{ik}$ such that
\begin{align*}
    s_i\geq \, &a_k^\top \hat{\xi}_i + b_k^\top \theta + c_k - 
    \mu_{ik}^\top (A\hat{\xi}_i + c) + \sigma_\Theta(-(B^\top \mu_{ik} + \tau_i))\\
    &+ \sigma_\Xi(u_{ik}) - u_{ik}^\top\hat\xi_i + \phi(q) \lambda \left\Vert \frac{a_k - A^\top\mu_{ik} - u_{ik}}{\lambda} \right\Vert_*^q
\end{align*}
and infinity otherwise. Plugging this back into \eqref{eq:lagrange_dual} one obtains the dual stated in the Lemma.

Finally, when $\delta=0$ by a tedious, but straightforward computation one checks that both \eqref{eq:regret_opt_problem} and \eqref{eq:full_relaxation} collapse to 
\begin{align*}
    \sup_{\beta\in\Theta}\frac{1}{N}\sum_{i=1}^N
    \left[
        \max_{k\in[K]}\left(a_k^\top\hat\xi_i + b_k^\top\theta + c_k\right)
        - \max_{m\in[K]}\left(a_m^\top\hat\xi_i + b_m^\top\beta + c_m\right)
    \right]
\end{align*}
and so $\beta$ will be equal to an ERM optimizer in both \eqref{eq:regret_opt_problem} and \eqref{eq:full_relaxation}.
\end{proof}

\subsection{Proof of Corollary~\ref{corr:relaxation_full_polyhedral}}

\begin{proof}
    For $\Xi=\{x\mid Px\leq r\}$, the support function satisfies $\sigma_{\Xi}(u)=\inf\{\zeta^\top r\mid \zeta\geq0,\; P^\top\zeta=u\}$.
    Similarly, for $\Theta=\{\theta\mid M\theta\leq w\}$ it holds $\sigma_{\Theta}(y)=\inf\{\alpha^\top w\mid \alpha\geq0,\; M^\top\alpha=y\}$.
    Applying these identities to the three support functions in \eqref{eq:full_relaxation} gives the variables
    $\eta$, $\zeta_{ik}$, and $\nu_{ik}$. The outer support term becomes
    $w^\top\eta$ with $M^\top\eta=\frac{1}{N}\sum_i\tau_i$. The term
    $\sigma_\Theta(-(B^\top\mu_{ik}+\tau_i))$ becomes $w^\top\nu_{ik}$ with
    $M^\top\nu_{ik}=-(B^\top\mu_{ik}+\tau_i)$. Finally, setting
    $u_{ik}=P^\top\zeta_{ik}$ rewrites
    $\sigma_\Xi(u_{ik})-u_{ik}^\top\hat\xi_i$ as $\zeta_{ik}^\top(r-P\hat\xi_i)$ and yields the stated
    norm term.
\end{proof}

\subsection{Proof of Theorem~\ref{thm:relax_bound}}

\begin{proof}
By the duality of Lemma~\ref{lemma:relaxation_full}, it is enough to show that 
for fixed $\theta$ the optimal value of \eqref{eq:regret_opt_problem} is upper bounded by the ex-post regret for that $\theta$. 
Dropping the constraints $\sum_{k=1}^K z_{ik}=\beta$ for all $i\in[N]$ in \eqref{eq:regret_opt_problem} only enlarges the feasible set.
However, after setting $\beta_{ik}=z_{ik}/\gamma_{ik}$, by \cite[Theorem~9]{kuhn_wasserstein_2024} this new upper bound is exactly the finite bidual
of the DRO problem applied to the ex-post regret loss $\xi\mapsto \ell(\theta,\xi)-\inf_{\beta\in\Theta}\ell(\beta,\xi)$. 
The enlarged problem therefore has value $R^{\mathrm{post}}(\theta)$, which proves the claim.
\end{proof}

\subsection{Proof of Theorem~\ref{thm:exact_relaxation}}
\begin{proof}
We begin with the case $\delta=0$. By Lemma~\ref{lemma:relaxation_full}, it suffices to show that the exact bidual
\eqref{eq:kuhn_dual_applied} and the relaxed bidual \eqref{eq:regret_opt_problem} have the same optimal value.
We do this by explicitly constructing for any fixed $\beta$ an optimal $(\gamma_{ik}, q_{ik}, z_{ik}, t_{ik})$
in \eqref{eq:regret_opt_problem} and then verifying that $z_{ik} = \gamma_{ik}\beta$ for all $i, k$. 
First, when $\delta=0$ the budget constraint in \eqref{eq:regret_opt_problem} implies $q_{ik}=0$ for all $i, k$. The objective then 
splits up as $\frac{1}{N}\sum_{i\in[N]}(S_{i, 1} + S_{i,2})$ where $S_{i,1}=\sum_{k\in[K]}\gamma_{ik}(a_k^\top\hat\xi_i + b_k^\top \theta + c_k)$ 
and $S_{i,2} = -\sum_{k\in[K]}t_{ik}$. We upper bound each of these terms. For $S_{i,1}$ we have using the constraints on $\gamma_{ik}$
\begin{align*}
    S_{i,1} = \sum_{k\in[K]}\gamma_{ik}(a_k^\top\hat\xi_i + b_k^\top \theta + c_k) \leq \max_{k}(a_k^\top\hat\xi_i + b_k^\top \theta + c_k) 
\end{align*}
while for $S_{i,2}$ we obtain from the constraints on $\gamma_{ik}$ and $z_{ik}$ and the optimality of $t_{ik}$ that
\begin{align*}
    S_{i, 2} &= -\sum_{k\in[K]}t_{ik} \leq -\sum_{k\in[K]}\max_m(\gamma_{ik}(a_m^\top\hat\xi_i + c_m) + b_m^\top z_{ik})\\
    &\leq -\max_m\left\{\sum_{k\in[K]}\gamma_{ik}(a_m^\top\hat\xi_i + c_m) + \sum_{k\in[K]}b_m^\top z_{ik}\right\}
    = -\max_m(a_m^\top \hat\xi_i + b_m^\top \beta + c_m).
\end{align*}
Equality in both bounds is attained for the choice $\gamma_{ik_i^*} = 1$, $z_{ik_i^*} = \beta$, 
$t_{ik_i^*} = \max_{m}(a_m^\top\hat\xi_i + b_m^\top \beta + c_m)$ for $k_i^*=\argmax_k (a_k^\top\hat\xi_i + b_k^\top \theta + c_k)$ (with deterministic
tiebreaker in the case of ties) and $\gamma_{ik} = 0$, $z_{ik} = 0$, $t_{ik}= 0$ otherwise.
This guarantees $z_{ik} = \gamma_{ik}\beta$ at the optimum and so the relaxation in \eqref{eq:regret_opt_problem} is exact. 

For $\delta \ge D(\Xi)$, the Wasserstein ball $B^p_\delta(\mathbb P_0)$ contains every probability measure supported on $\Xi$. Hence, for any measurable function $f$,
\[
\sup_{\mathbb P\in B_\delta^p(\mathbb P_0)} \mathbb E_{\mathbb P}[f(X)]
= \sup_{\xi\in\Xi} f(\xi).
\]
Applying this identity to the definitions of ex-ante and ex-post regret yields
\[
\begin{aligned}
R^{\mathrm{ante}}(\theta)
&= \sup_{\beta\in\Theta}\sup_{\mathbb P\in B_\delta^p(\mathbb P_0)}
\mathbb E_{\mathbb P}[\ell(\theta,X)-\ell(\beta,X)] = \sup_{\beta\in\Theta}\sup_{\xi\in\Xi}
\big(\ell(\theta,\xi)-\ell(\beta,\xi)\big) \\
&= \sup_{\xi\in\Xi}
\Big(\ell(\theta,\xi)-\inf_{\beta\in\Theta}\ell(\beta,\xi)\Big) = \sup_{\mathbb P\in B_\delta^p(\mathbb P_0)}
\mathbb{E}_\mathbb{P}\left[\ell(\theta,X)-\inf_{\beta\in\Theta}\ell(\beta,X)\right]\\
&= R^{\mathrm{post}}(\theta).
\end{aligned}
\]
Since the relaxation is bounded above by $R^{\mathrm{post}}(\theta)$ by Theorem~\ref{thm:relax_bound} and bounded below by
$R^{\mathrm{ante}}(\theta)$ by construction, it is exact whenever $\delta \ge D(\Xi)$.
\end{proof}

\subsection{Proof of Theorem~\ref{thm:relax_gap_delta}}

\begin{proof}
Since $d_k=0$ here, $L_\xi=\max_k\|a_k\|_*$ makes $\xi\mapsto\ell(\beta,\xi)$ $L_\xi$-Lipschitz uniformly in
$\beta$. The lower bound $R(\theta)\le\widehat R(\theta)$ follows by setting $z_{ik}=\gamma_{ik}\beta$ in the relaxation,
so it remains to show $\widehat R(\theta)\le R(\theta)+2L_\xi\delta$.
Take any feasible point $(\beta,\gamma,q,z,t)$ of \eqref{eq:regret_opt_problem}. Compactness and the perspective
convention imply that $\gamma_{ik}=0$ forces $q_{ik}=z_{ik}=0$. For $\gamma_{ik}>0$, set
$\xi_{ik}=\hat\xi_i+q_{ik}/\gamma_{ik}$ and $\beta_{ik}=z_{ik}/\gamma_{ik}$, and choose arbitrary values when
$\gamma_{ik}=0$. Then $\sum_k\gamma_{ik}\beta_{ik}=\sum_k z_{ik}=\beta$ for every $i$.

Construct an exact feasible point for \eqref{eq:kuhn_dual_applied} with the same $\beta,\gamma,q$ and
$\bar t_{ik}=\gamma_{ik}\ell(\beta,\xi_{ik})$, which is feasible by the definition of $\ell$ and the perspective convention.
Let $J_{\rm rel}$ and $J_{\rm ex}$ be the relaxed objective value and the constructed exact objective value. Then one can bound
\begin{align*}
    J_{\rm rel} &= J_{\rm ex} + \frac1N\sum_{i=1}^N\sum_{k=1}^K (\bar t_{ik} - t_{ik})\\
    &\le J_{\rm ex}
    +\frac1N\sum_{i=1}^N\sum_{k=1}^K
    \gamma_{ik}\bigl(\ell(\beta,\xi_{ik})-\ell(\beta_{ik},\xi_{ik})\bigr)\\
    &= J_{\rm ex}+\frac1N\sum_{i=1}^N\Bigg[
    \sum_{k=1}^K\gamma_{ik}\bigl(\ell(\beta,\xi_{ik})-\ell(\beta,\hat\xi_i)\bigr)\\
    &\quad+\ell(\beta,\hat\xi_i)-\sum_{k=1}^K\gamma_{ik}\ell(\beta_{ik},\hat\xi_i)
    +\sum_{k=1}^K\gamma_{ik}\bigl(\ell(\beta_{ik},\hat\xi_i)-\ell(\beta_{ik},\xi_{ik})\bigr)
    \Bigg]\\
    &\leq J_{\rm ex}
    +\frac{2L_\xi}{N}\sum_{i=1}^N\sum_{k=1}^K
    \gamma_{ik}\left\|\frac{q_{ik}}{\gamma_{ik}}\right\|\\
    &\leq J_{\rm ex}+2L_\xi\left(\frac1N\sum_{i=1}^N\sum_{k=1}^K
    \gamma_{ik}\left\|\frac{q_{ik}}{\gamma_{ik}}\right\|^p\right)^{1/p}\\
    &\leq R(\theta)+2L_\xi\delta
\end{align*}
where terms with $\gamma_{ik}=0$ are interpreted as zero. The first inequality holds since the relaxed epigraph constraints imply $t_{ik}\ge\gamma_{ik}\ell(\beta_{ik},\xi_{ik})$, 
the second inequality uses convexity of $\beta\mapsto\ell(\beta,\hat\xi_i)$, the identity
$\beta = \sum_{k\in[K]}\gamma_{ik}\beta_{ik}$, and the $L_\xi$-Lipschitz bound in $\xi$, and the last two inequalities use
Jensen's inequality, $J_{\rm ex}\le R(\theta)$, and the budget constraint.
Taking the supremum over all relaxed feasible points yields $\widehat R(\theta)\le R(\theta)+2L_\xi\delta$.
\end{proof}

\end{document}